\documentclass[10pt,english]{article}
\usepackage{graphicx}
\usepackage{subcaption}
\PassOptionsToPackage{dvipsnames,svgnames}{xcolor}
\usepackage[utf8]{inputenc}
\usepackage{amssymb,amsmath,amsthm,epsfig,stmaryrd,bbm, mathrsfs}
\usepackage{fancyhdr}
\usepackage[english]{babel}
\usepackage{titlesec,graphicx}
\usepackage[top=1.5cm,left=2cm,right=2cm,bottom=2cm]{geometry}
\usepackage{pstricks}
\usepackage{pgf,tikz}
\usepackage{thmtools}
\usepackage{tkz-euclide}
\usepackage{hyperref}
\hypersetup{
    colorlinks=true,
    linkcolor=blue,
    filecolor=blue,
    citecolor=blue,      
    urlcolor=blue}
\usetikzlibrary{arrows,patterns,calc}
\declaretheorem[numberwithin=section]{theorem}
\declaretheorem[sibling=theorem]{exemple,definition,proposition,property,lemma,
corollary}
\theoremstyle{definition}
\declaretheorem[numberwithin=section]{remark}
\numberwithin{equation}{section}

\usepackage{tikz}
\usepackage{lettrine}

\theoremstyle{definition}

\def\build#1_#2^#3{\mathrel{
\mathop{\kern 0pt#1}\limits_{#2}^{#3}}}

\makeatletter

\usepackage[top=1.5cm,left=2cm,right=2cm,bottom=2cm]{geometry}

\begin{document} 

\title{Critical points of the one dimensional Ambrosio-Tortorelli functional with an obstacle condition}
\author{Martin Rakovsky \footnote{Université Paris-Saclay, CNRS, Laboratoire de mathématiques d’Orsay, 91405, Orsay, France. Email : martin.rakovsky@universite-paris-saclay.fr}}
\date{}
\maketitle

\textbf{Abstract :} We consider a family of critical points of the
Ambrosio-Tortorelli energy with an obstacle condition on the phase
field variable. This problem can be interpreted as a time discretization
of a quasistatic evolution problem where the obstacle at step $n$ is defined as the solution at step $n-1$. The obstacle condition
now reads as an irreversibility condition (the crack can only increase
in time). The questions tackled here are the regularity of the critical
points, the properties inherited from the obstacle sequence, the position
of the limit points and the equipartition of
the phase field energy. The limits of such critical points turn out to be critical points of the Mumford-Shah energy that inherit the possible discontinuities induced by the obstacle sequence.

\section{Introduction}

The Mumford-Shah functional, introduced by D.Mumford and J.Shah in $1989$ in \cite{MS89} in the context of image segmentation, is also used in fracture mechanics to describe brittle fractures \cite{FM98}. If $\Omega$ is a bounded open subset of $\mathbb{R}^N$, it is defined, for any $u\in SBV^2(\Omega)$, as  
\[MS(u) =  \int_\Omega |\nabla u|^2 + 2 \mathcal{H}^{N-1}(J_u).\]
In the spirit of the Allen-Cahn model for phase transition (\cite{M87}, \cite{S88}), Ambrosio and Tortorelli proposed in \cite{AT92} the following variational phase field regularization, defined for $(u,v) \in H^1(\Omega) \times [H^1(\Omega)\cap L^\infty(\Omega)]$ by 
\[AT_\varepsilon (u,v) = \int_\Omega (\eta_\varepsilon + v^2) |\nabla u|^2 dx + \int_\Omega \left(\varepsilon |\nabla v|^2 + \frac{(1-v)^2}{\varepsilon}\right)dx\]
where $0< \eta_\varepsilon << \varepsilon$ is a small parameter that ensures ellipticity. This regularization has also a mechanical interpretation. The phase field variable $v$ can be seen as a damage variable taking values in the interval $[0,1]$, between the state $\{v=1\}$ where the material is sane and the state $\{v=0\}$ where the material is broken. This regularized formulation is also at the basis of numerical simulations (see \cite{BFM08}).

\smallskip

In \cite{AT92}, a $\Gamma-$convergence result of $AT_\varepsilon$ to $MS$ is proved by suitably extending $MS$ as a two variable functional : 
\[MS(u,v) = \left\{
\begin{array}{ll}
MS(u) & \text{if } v\equiv 1 \; ,\\
+\infty & \text{otherwise.} \\
\end{array}
\right.\]
Under suitable boundary conditions, the existence of minimizers for $AT_\varepsilon$ for a fixed value of $\varepsilon$ is obtained via the direct method of the Calculus of Variation. The $\Gamma$-convergence then ensures that a converging sequence $(u_\varepsilon, v_\varepsilon)$ of minimizers of $AT_\varepsilon$ converges toward a minimizer of $MS$. 

\smallskip

Concerning the numerical implementation, the term $v^2 |\nabla u|^2$ implies a lack of convexity for the functional $AT_\varepsilon$. Consequently, numerical methods might fail to converge toward a minimizer of $AT_\varepsilon$. Using the fact that the functional remains separately strictly convex, \cite{BFM08} proposes to perform an alternate minimization algorithm. It is proven in \cite[Theorem 1]{B07} that the sequence of iterates converge toward a critical point of $AT_\varepsilon$, but not always to a minimizer. This motivates to investigate the convergence of critical points of the Ambrosio-Tortorelli functional as $\varepsilon \rightarrow 0$ and compare this convergence with the original goal to minimize the Mumford-Shah functional. 

\medskip

A critical point for the Ambrosio-Tortorelli functional is a pair $(u_\varepsilon, v_\varepsilon) \in H^1(\Omega)\times [H^1(\Omega) \cap L^\infty (\Omega)]$ such that 

\[\dfrac{d}{dt}\left|_{t=0} \right. AT_\varepsilon (u_\varepsilon + t\psi, v_\varepsilon + t\varphi) = 0 \;\;\; \forall \, (\psi, \varphi) \in \: H_0^1 (\Omega)^2.\]
with suitable conditions at the boundary $\partial \Omega$ (for example Dirichlet conditions, as in \cite{BMRa23}). 

\medskip

The $\Gamma$-convergence provides little help to study the convergence of critical points. The extension of the fundamental theorem of $\Gamma-$convergence to the convergence of critical points has been studied in various settings, see \cite{HT00}, \cite{T02}, \cite{T05} for the Allen-Cahn functional and \cite{BBH94}, \cite{BBO01}, \cite{LR01}, \cite{SS07} for the Ginzburg-Landau functional. For the Ambrosio-Tortorelli functional, the convergence has been established in all dimensions in \cite{BMRa23}, with the additional hypothesis of the convergence of the energy $AT_\varepsilon (u_\varepsilon, v_\varepsilon)$ to $MS(u,1)$. In dimension one, this hypothesis is no longer required, and it is proved in \cite{FLS09} and \cite{BMRb23} that a sequence of critical points for the one-dimensional Ambrosio-Tortorelli converges to critical points of $MS$. A similar result is proved in \cite{L10} for the one-dimensional Mumford-Shah functional adapted to image segmentation. (See also \cite{BI23} in cohesive fracture).

\medskip

In \cite{FL03}, a model for the quasi-static evolution of a brittle fracture using the Mumford-Shah functional is proposed. The model starts with a discretized version of the evolution of the crack, where the irreversibility through time is translated into the condition that the crack $K (t_i)$ is included in the crack $K (t_{i+1})$, and the Mumford-Shah functional is then minimized at time $t_{i+1}$ over all the cracks $K$ containing $K(t_i)$, and with a variable Dirichlet condition $g_i$ at the boundary $\partial \Omega$. In \cite{G05}, it is proved that a quasi-static evolution of $AT_\varepsilon$ approximates a quasi-static evolution for $MS$ using minimizers of $AT_\varepsilon$. The irreversibility condition translates into the minimization of $AT_\varepsilon$ at time $t_{i+1}$ among all couples $(u,v)$ such that $v\leqslant v_{t_i, \varepsilon}$, where $(u_{t_i,\varepsilon}, v_{t_i,\varepsilon})$ is the minimum of $AT_\varepsilon$ at time $t_i$. 

\smallskip

A natural question is whether the result of \cite{G05} can be extended to critical points. Since in high dimension, results on the convergence of critical points of $AT_\varepsilon$ require stronger hypothesis, such an extension can be investigated only in dimension one for the moment. This paper is the first step in that direction, namely, we study the convergence of the critical points of the functional $AT_\varepsilon$ defined on the pairs $(u,v)$ in $H^1([0,L])$ satisfying $v\leqslant v_{0,\varepsilon}$, where $(u_{0,\varepsilon}, v_{0,\varepsilon})$ is a critical point of $AT_\varepsilon$. We use the same setting as \cite{BMRb23} (Dirichlet conditions on the variables $(u,v)$).

\smallskip
 
In Section 2, we define the one-dimensional functional $\overline{AT}_\varepsilon$ at study, which corresponds to the functional $AT_\varepsilon$ restricted to the pairs $(u, v)$ satisfying $v\leqslant v_{0,\varepsilon}$, we recall the results established by \cite{BMRb23} and state the main result in \autoref{Theorem 1} : the critical points of $\overline{AT}_\varepsilon$ with the irreversibility condition $v\leqslant v_{0,\varepsilon}$ converge to critical point of $MS$. \autoref{Theorem 2}, proved in the appendix, deals with the $\Gamma-$convergence of the sequence $\overline{AT}_\varepsilon$. \autoref{Theorem 1} can then be seen as an extension of \autoref{Theorem 2} for critical points. 

\smallskip

When we study the convergence of a functional with a phase field term, a natural question is the validity of an equipartition of the energy principle, i.e whether the terms of the phase field energy share the same limit. This principle, as well as the convergence of the Dirichlet approximation term is also stated in \autoref{Theorem 1}. 

\smallskip

In Section 3, we establish the properties concerning the critical points of $\overline{AT_\varepsilon}$, namely the regularity and the variations of $v_\varepsilon$. In Section 4, we study the asymptotic behaviour of the sequence $(u_\varepsilon , v_\varepsilon)$ and prove \autoref{Theorem 1}. 

\smallskip

Similarly to \cite{FLS09}, \cite{BMRb23} and \cite{L10}, the key point to the main estimates is the study of the quantity called \textit{discrepancy} (\autoref{discrepancy}) associated to $\overline{AT}_\varepsilon$ (which can be interpreted as the Hamiltonian of the functional). In \cite{FLS09}, \cite{BMRb23}, this discrepancy is constant, which is no longer the case in our situation. However, we are able to control the variations of the discrepancy in \autoref{estimatediscrepancy} to obtain analogous estimates and results as in \cite{BMRb23}. 

\section{Notations and statement of the results}

\subsection{Notations and preliminary results}

Throughout the paper, $L$ is a positive number. Letters like $C,m$ and $M$ will denote constants that do not depend on any other parameter and that may change from line to line. We also denote by $I_r(x)$ the interval $(x-r,x+r)$.  

For any open subset $\Omega \subset (0,L)$, we denote by 

\[||f||_{\infty, \Omega} := \sup_{x\in \Omega} |f(x)| \; ,\]
\[[\,f\,]_{1/2, \Omega} :=\sup_{x\neq y \in \Omega} \dfrac{|f(x)-f(y)|}{|x-y|^{1/2}} \; ,\] 

\[||f||_{C^{0,1/2}, \Omega} := ||f||_{\infty, \Omega} + [\,f\,]_{1/2, \Omega}.\]
The one-dimensional Ambrosio-Tortorelli functional,  is defined as
\[AT_\varepsilon(u,v) :=  \int_0^L \left((\eta_\varepsilon + v^2)|u'|^2 + \frac{(1-v)^2}{\varepsilon} + \varepsilon |v'|^2 \right)dx, \;\;\;\;\; \forall (u,v) \in H^1(0,L)^2,\]
where $\eta_\varepsilon$ is a positive parameter such that $\eta_\varepsilon / \varepsilon \rightarrow 0$. Before defining the one-dimensional Mumford-Shah functional, we recall the definition of the space $SBV^2(0,L)$, which is the space of the functions $u\in L^1$ such that their distributional derivative is a Radon measure which  writes as 
\[Du = u' d\mathcal{L} + (u^+ - u^-) \nu_u \mathcal{H}^{0}_{|J_u} \;\; \text{with} \;\; \int_0^L |u'|^2 < \infty, \; \mathcal{H}^0 (J_u) < \infty.\]
The one-dimensional Mumford-Shah functional $MS$ defined on $SBV^2(0,L)\times L^2(0,L)$ reads as 
\[MS(u,v) := \left\{\begin{array}{lll}
\displaystyle \int_0^L |u'|^2 + 2 \mathcal{H}^0 (J_u) & \text{if } v \equiv 1 \; ,\\
+\infty & \text{otherwise}.\\
\end{array}\right.
.\]
We define a critical point of $AT_\varepsilon$ as follows 

\begin{definition}
Let $a_0 \in \mathbb{R}$.
A critical point of $AT_\varepsilon$ is a pair $(u,v)\in H^1(0,L)^2$ that satisfies for all $(\psi, \varphi)\in \mathcal{C}_c^\infty (0,L)$,

\[\left\{\begin{array}{l}
dAT_\varepsilon(u,v) [(\psi, \varphi)] := \displaystyle\lim\limits_{t\rightarrow 0} \frac{AT_\varepsilon (u+t\psi, v+t\varphi) - AT_\varepsilon (u,v)}{t}=0 \; ,\\
\\
u_{0,\varepsilon}(0)=0, \; u_{0,\varepsilon}(L)=a_0 \; ,\\
\\
v_{0,\varepsilon}(0)=v_{0,\varepsilon}(L)=1 . \\
\end{array}\right.\]
\end{definition}
A critical point $(u_{0,\varepsilon}, v_{0,\varepsilon})$ of $AT_\varepsilon$ satisfies the following equations on $(0,L)$
\begin{equation}\label{criticalpoint0}
\left\{\begin{array}{ll}
-\varepsilon v_{0,\varepsilon}'' + \frac{v_{0,\varepsilon}-1}{\varepsilon} + v_{0,\varepsilon} |u_{0,\varepsilon}'|^2=0 & \text{in } (0,L)\; ,\\
&\\
\left[u_{0,\varepsilon}' (\eta_\varepsilon + v_{0,\varepsilon}^2) \right]' = 0 & \text{in } (0,L) \; ,\\
&\\
u_{0,\varepsilon}(0)=0, \; u_{0,\varepsilon}(L)=a_0 \; , &\\
&\\
v_{0,\varepsilon}(0)=v_{0,\varepsilon}(L)=1. &\\
\end{array}\right.
\end{equation}
In the sequel, we consider a sequence $(u_{0,\varepsilon}, v_{0,\varepsilon})$ of critical points of $AT_\varepsilon$ satisfying the following energy bound  
\begin{equation}\label{energybound}
\limsup\limits_{\varepsilon >0} AT_\varepsilon (u_{0,\varepsilon}, v_{0,\varepsilon}) < C.
\end{equation}
In \cite[Theorem 1.1]{BMRb23}, the following has been established :

\begin{theorem}\label{theorempourtemps0}
Up to a subsequence, $(u_{0,\varepsilon} , v_{0,\varepsilon}) \rightarrow (u_0, 1)$ in $L^2(0,L)^2$, where either $u_0=u_{jump} \equiv a_0 1_{[L/2,L]}$ or $u_0(x)=u_{aff}(x) = a_0x/L$.
\end{theorem}
When integrating the second equation of \eqref{criticalpoint0}, we obtain a constant $c_{0,\varepsilon}$ such that 
\[u_{0,\varepsilon}' (\eta_\varepsilon + v_{0,\varepsilon}^2) = c_{0,\varepsilon}.\]
\cite[Lemma 2.4]{BMRb23} states that 
\begin{lemma}
Up to a subsequence, the sequence 
\begin{equation}\label{defc0}
c_{0,\varepsilon} \rightarrow c_0;
\end{equation}
with $c_0$ a real number satisfying $c_0 \in \{0, a_0/L\}$.
\end{lemma}
\autoref{theorempourtemps0} enlightens a selection phenomenon : the accumulation points of the sequence $(u_{0,\varepsilon}, v_{0,\varepsilon})$ are specific critical points of $MS$. Indeed, critical points of $MS$ are either the affine function $u_{aff}$ or the piecewise constant functions. It is also proved in \cite[Theorem 1.2]{BMRb23} that $u_{jump}$ can be reached by a sequence of critical points of $AT_\varepsilon$. The crack of the material at time $0$ then corresponds to $J_{u_0}$, which is, according to \cite[Theorem 1.1]{BMRb23}

\begin{equation}\label{crack}
\Gamma_0 = J_{u_0} = \emptyset \;\;\;\;\; \text{ or } \;\;\;\;\; \Gamma_0= J_{u_0} = \left\{\frac{L}{2}\right\}.
\end{equation}

\medskip

The Ambrosio-Tortorelli functional with obstacle is defined as 
\begin{equation}\label{ATobstacle}
\overline{AT_\varepsilon}(u, v) := \left\{\begin{array}{ll}
AT_\varepsilon(u, v) & \text{ if } v \leqslant v_{0,\varepsilon} \; ,\\
+\infty & \text{ else,}
\end{array}\right. \;\;\; \forall\: (u,v) \in H^1(0,L).
\end{equation}
In order to account for the constraint on the phase field variable, we introduce the following definition for constrained critical points.   
\begin{definition}\label{criticobstacle}
A critical point of $\overline{AT}_\varepsilon$ is a pair $(u,v)\in H^1(0,L)^2$ that satisfies $v\leqslant v_{0,\varepsilon}$ and, for all $(\psi, \varphi)\in \mathcal{C}_c^\infty (0,L)$ with $\varphi\geqslant 0$,
\[\left\{\begin{array}{l}
d\overline{AT}_\varepsilon(u,v) [(\psi, \varphi)] := \displaystyle\lim\limits_{t\rightarrow 0^+} \frac{\overline{AT}_\varepsilon (u,v) - \overline{AT}_\varepsilon (u-t\psi, v-t\varphi)}{t}=0 \; ,\\
\\
u(0)=0, \; u(L)=a_1 \; ,\\
\\
v(0)=v(L)=1. \\
\end{array}\right.\]
\end{definition}
Note that the Dirichlet condition at $L$ is different and carry the evolution in time of the material. 

In the sequel, we denote 
\begin{equation}\label{Ovarepsilon and Fvarepsilon}
O_\varepsilon = \{v_\varepsilon < v_{0,\varepsilon}\} \;\;\;\; \text{and} \;\;\;\; F_\varepsilon = \{v_\varepsilon = v_{0,\varepsilon}\}.
\end{equation}
Since $v_\varepsilon , v_{0,\varepsilon}\in H^1(0,L)$, there difference is continuous by Sobolev embedding so $O_\varepsilon$ and $F_\varepsilon$ are respectively open and closed in $[0,L]$.

\subsection{Euler-Lagrange equations}

In this subsection, we establish the Euler-Lagrange equations for $\overline{AT}_\varepsilon$. 
Let $(u_\varepsilon , v_\varepsilon)$ be a critical point of $\overline{AT}_\varepsilon$. 
Take $\varphi =0$ in \autoref{criticobstacle}.  Similarly to \cite[Introduction]{FLS09}, the equation satisfied by $u_\varepsilon$ reads as  
\begin{equation}
\label{eq : eq1}
\left[u_\varepsilon ' (\eta_\varepsilon + v_\varepsilon ^2) \right]' =0.
\end{equation}
This means that there exists a constant $c_\varepsilon$ such that 
\begin{equation}
\label{eq : eq2}
u_\varepsilon ' = \frac{c_\varepsilon}{(\eta_\varepsilon + v_\varepsilon ^2)} \;\;\;\;\; \text{a.e in } (0,L).
\end{equation}
As $v_\varepsilon$ is continuous by Sobolev embedding, the left hand side is also continuous. Similarly to \cite{FLS09}, we obtain that $u_\varepsilon '$ has a constant sign, corresponding to the sign of $c_\varepsilon$, meaning that $u_\varepsilon$ is monotone on $[0,L]$. 

\bigskip

We next move on to the equation for $v_\varepsilon$. Taking $\psi =0$ in \autoref{criticobstacle},
we obtain the following inequality, where $v_\varepsilon ''$ denotes the distributional derivative of $v_\varepsilon'$, which exists as $v_\varepsilon \in H^1(0,L)$,  
\begin{equation}
\label{eq : eq3}
\left\langle\, \varphi\, , \,  v_\varepsilon |u_\varepsilon '|^2 - \varepsilon v_\varepsilon '' +\frac{v_\varepsilon -1}{\varepsilon} \,\right\rangle \leqslant 0
\end{equation}
which we rewrite in the following way : there exists a non positive Radon measure distribution $\mu_\varepsilon$ such that 
\begin{equation}
\label{eq : eq4}
-\varepsilon v_\varepsilon '' + \frac{v_\varepsilon -1}{\varepsilon} + v_\varepsilon |u_\varepsilon '|^2 = \mu_\varepsilon \leqslant 0.
\end{equation}
Taking $\varphi$ compactly supported on $\{v_\varepsilon < v_{0,\varepsilon}\}$ gives the reversed inequality of \eqref{eq : eq3}. This means that the measure $\mu_\varepsilon$ is supported on the closed set $F_\varepsilon$. 

\smallskip

Using \eqref{eq : eq4}, we see that $v_\varepsilon ''$ is also a Radon measure, so that $v_\varepsilon '\in BV(0,L)$, which is embedded in $L^\infty (0,L)$, and that $v_\varepsilon \in C^{0,1}(0,L)$. In section $3.1$, we provide a better regularity result on $v_\varepsilon$. This allows us to establish an explicit formula for $\mu_\varepsilon$ at the end of section $3.1$ (\autoref{exprmu}).

\smallskip

In summary, the equations satisfied by $(u_\varepsilon, v_\varepsilon)$ are    

\begin{equation}\label{criticalpoint1}
\left\{\begin{array}{ll}
-\varepsilon v_\varepsilon '' + \frac{v_\varepsilon-1}{\varepsilon} + v_\varepsilon |u_\varepsilon '|^2= \mu_\varepsilon \leqslant 0 & \text{in } (0,L)\\
&\\
\left[u_\varepsilon' (\eta_\varepsilon + v_\varepsilon^2) \right]' = 0 & \text{in } (0,L)\\
&\\
v_\varepsilon\leqslant v_{0,\varepsilon} &\text{in } (0,L)\; ,\\
&\\
u_\varepsilon(0)=0, \; u_\varepsilon(L)=a_0 \; ,&\\
&\\
v_\varepsilon(0)=v_\varepsilon(L)=1.&\\
\end{array}\right.
\end{equation}

\subsection{Statement of the theorems}

We consider a sequence $(u_\varepsilon, v_\varepsilon)$ of critical points of $\overline{AT}_\varepsilon$. 
We suppose that the sequence $(u_\varepsilon, v_\varepsilon)$ satisfies the following energy bound
\begin{equation}
\label{energyboundtemps1}
\limsup\limits_{\varepsilon \rightarrow 0} \overline{AT}_\varepsilon (u_\varepsilon, v_\varepsilon)\leqslant \mathcal{C}.
\end{equation}

Our first result concerns the regularity and is proved in section 3.

\begin{theorem}\label{Theorem regularity}
Let $(u_\varepsilon, v_\varepsilon)$ be a critical point of $\overline{AT}_\varepsilon$ in $H^1(0,L) \times H^1(0,L)$. Then $v_\varepsilon \in C^{1,1}([0,L])$.
\end{theorem}

We now state the main result of this work concerning the limit points of $(u_\varepsilon, v_\varepsilon)$.


\begin{theorem}[Convergence of the critical points]\label{Theorem 1}
Let $(u_\varepsilon , v_\varepsilon)$ be a family of critical points of $\overline{AT}_\varepsilon$ satisfying \eqref{energyboundtemps1}. Then there exists a real number $x_{min}\in \, (0,L)$ and a subsequence such that 

\begin{itemize}
\item[$(i)$] $(u_\varepsilon , v_\varepsilon) \rightarrow (u_1 , 1)$ in $L^2(0,L)$, with either $u_1(x)=u_{jump}(x)= a_1 1_{[x_{min}, L]}(x)$ or $u_1(x) = u_{aff} (x) = \dfrac{a_1x}{L}$. Moreover, if $|F_\varepsilon | \neq 0$ and $|a_1|\geqslant a_0$, then $u_1 = u_{jump}$.

\item[$(ii)$] $\left(\frac{(1-v_\varepsilon)^2}{\varepsilon} + \varepsilon (v_\varepsilon')^2\right) dx \overset{\star}{\rightharpoonup} \alpha \delta_{x_{min}}$ in $\mathcal{M}([0,L])$ with $\alpha = 0$ or $\alpha = 2$. Moreover, if $u_1= u_{jump}$ then $\alpha =2$. 

\item[$(iii)$] $(\eta_\varepsilon + v_\varepsilon ^2) |u_\varepsilon ' |^2 dx \overset{\star}{\rightharpoonup} |u_1'|^2 dx$ in $\mathcal{M}([0,L])$. 

\item[$(iv)$] (Equipartition of the energy principle) $\left|\dfrac{(1-v_\varepsilon )^2}{\varepsilon} - \varepsilon |v_\varepsilon '|^2\right| \rightarrow 0$ in $L^1 (0,L)$. 

\item[$(v)$]  $\varepsilon |v_\varepsilon '|^2 dx \overset{\star}{\rightharpoonup} \dfrac{\alpha}{2} \delta_{x_{min}},$ and $\dfrac{(1-v_\varepsilon)^2}{\varepsilon} dx  \overset{\star}{\rightharpoonup} \dfrac{\alpha}{2} \delta_{x_{min}}$ in $\mathcal{M}([0,L])$. 
\end{itemize}
\end{theorem}

\smallskip

\begin{remark} \autoref{Theorem 1} extends the following 
$\Gamma-$ convergence result, which we prove in the appendix :

\begin{theorem}[$\Gamma$ convergence of $\overline{AT}_\varepsilon$]\label{Theorem 2}
The functional $\overline{AT_\varepsilon}$ $\Gamma$-converges to the functional $\overline{MS}$,  defined for every $(u,v) \in L^2(0,L)^2$ as

\[\overline{MS}(u, v) := \left\{\begin{array}{ll}
\displaystyle \int_0^L |u'|^2 + 2\mathcal{H}^0(J_u \cup \Gamma_0) & \text{ if } v= 1 \; ,\\
+\infty & \text{ else} \; .
\end{array}\right.\]
where $\Gamma_0$ is defined by \eqref{crack}. 
\end{theorem}

While the $\Gamma$-convergence implies that a converging sequence of minimizers of $\overline{AT}_\varepsilon$ converges to a minimizer of $\overline{MS}(u)$, \autoref{Theorem 1} acts as a selection criteria : the limit of critical points of $\overline{AT}_\varepsilon$ are critical points of $\overline{MS}$ that have at most one discontinuity point. Observe that we recover a similar selection phenomenon as in \cite[Theorem 1.2]{BMRb23}.

Considering \eqref{Theorem 2}, it is natural to define the crack at time $1$ by 

\begin{equation}\label{crack1}
\Gamma_1 = \Gamma_0 \cup J_{u_1}
\end{equation}
The different possible values for $\Gamma_1$ according to \autoref{Theorem 1} are represented in 
\autoref{figure}.

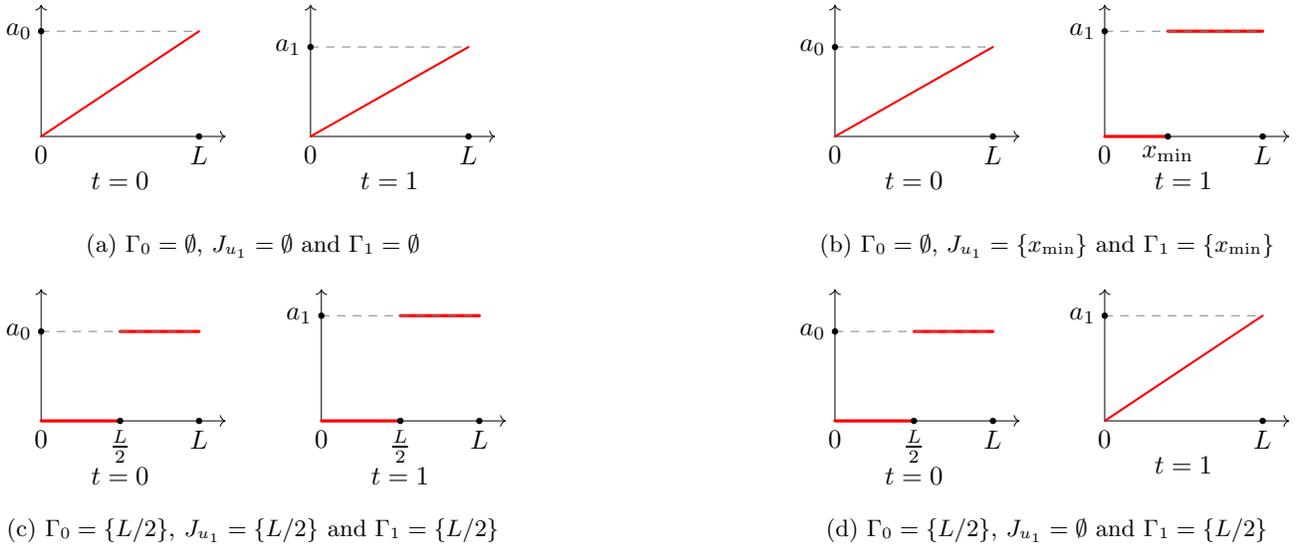
\begin{figure}[!h]
\begin{subfigure}[!h]{0.4\linewidth}
\begin{tikzpicture}
[scale=0.7]
\tkzInit[ymin=-1.5,ymax=3,xmin=-1,xmax=3.5]
\tkzClip
\tkzDefPoint(0,0){B}
\tkzDefPoint(3,0){B1}
\tkzDefPoint(3.5,0){C}
\tkzDefPoint(0,2){D}
\tkzDefPoint(0,2.5){E}
\tkzDefPoint(3,2){F}
\tkzDefPoint(0,2){H}

\tikzset{segment/.style={->}}
\tkzDrawSegment[segment](B,C)
\tkzDrawSegment[segment](B,E)
\tkzDrawSegment[color=red,thick](B,F)
\tkzDrawSegment[dashed,color=black!50](H,F)
\tkzDrawPoints(H,B1)

\tkzLabelPoint(B){$0$}
\tkzLabelPoint[below](B1){$L$}
\tkzLabelPoint[left](D){$a_0$}
\node[below] at (1.5,-0.5){$t=0$};
\end{tikzpicture}
\hspace{2mm}
\begin{tikzpicture}
[scale=0.7]
\tkzInit[ymin=-1.5,ymax=3,xmin=-1,xmax=3.5]
\tkzClip
\tkzDefPoint(0,0){B}
\tkzDefPoint(3,0){B1}
\tkzDefPoint(3.5,0){C}
\tkzDefPoint(0,1.7){D}
\tkzDefPoint(0,2.5){E}
\tkzDefPoint(3,1.7){F}
\tkzDefPoint(0,1.7){H}

\tikzset{segment/.style={->}}
\tkzDrawSegment[segment](B,C)
\tkzDrawSegment[segment](B,E)
\tkzDrawSegment[color=red,thick](B,F)
\tkzDrawSegment[dashed,color=black!50](H,F)
\tkzDrawPoints(H,B1)

\tkzLabelPoint(B){$0$}
\tkzLabelPoint[below](B1){$L$}
\tkzLabelPoint[left](D){$a_1$}
\node[below] at (1.5,-0.5){$t=1$};
\end{tikzpicture}
\caption{$\Gamma_0= \emptyset$, $J_{u_1}=\emptyset$ and $\Gamma_1 = \emptyset$}
\end{subfigure}
\hfill
\begin{subfigure}[!h]{0.4\linewidth}
\begin{tikzpicture}
[scale=0.7]
\tkzInit[ymin=-1.5,ymax=3,xmin=-1,xmax=3.5]
\tkzClip
\tkzDefPoint(0,0){B}
\tkzDefPoint(3,0){B1}
\tkzDefPoint(3.5,0){C}
\tkzDefPoint(0,1.7){D}
\tkzDefPoint(0,2.5){E}
\tkzDefPoint(3,1.7){F}
\tkzDefPoint(0,1.7){H}

\tikzset{segment/.style={->}}
\tkzDrawSegment[segment](B,C)
\tkzDrawSegment[segment](B,E)
\tkzDrawSegment[color=red,thick](B,F)
\tkzDrawSegment[dashed,color=black!50](H,F)
\tkzDrawPoints(H,B1)

\tkzLabelPoint(B){$0$}
\tkzLabelPoint[below](B1){$L$}
\tkzLabelPoint[left](D){$a_0$}
\node[below] at (1.5,-0.5){$t=0$};
\end{tikzpicture}
\hspace{2mm}
\begin{tikzpicture}
[scale=0.7]
\tkzInit[ymin=-1.5,ymax=3,xmin=-1,xmax=3.5]
\tkzClip
\tkzDefPoint(0,0){B}
\tkzDefPoint(3,0){B1}
\tkzDefPoint(3.5,0){C}
\tkzDefPoint(0,2){D}
\tkzDefPoint(0,2.5){E}
\tkzDefPoint(3,2){F}
\tkzDefPoint(1.2,0){G}
\tkzDefPoint(1.2,2){H}

\tikzset{segment/.style={->}}
\tkzDrawSegment[segment](B,C)
\tkzDrawSegment[segment](B,E)
\tkzDrawSegment[color=red, very thick](B,G)
\tkzDrawSegment[color=red,very thick](H,F)
\tkzDrawSegment[dashed,color=black!50](D,F)
\tkzDrawPoints(D,G,B1)

\tkzLabelPoint(B){$0$}
\tkzLabelPoint[below](G){$x_{\min}$}
\tkzLabelPoint[below](B1){$L$}
\tkzLabelPoint[left](D){$a_1$}
\node[below] at (1.5,-0.5){$t=1$};
\end{tikzpicture}
\caption{$\Gamma_0= \emptyset$, $J_{u_1}=\{x_{\min}\}$ and $\Gamma_1 = \{x_{\min}\}$}
\end{subfigure}
\hfill
\begin{subfigure}[!h]{0.4\linewidth}
\begin{tikzpicture}
[scale=0.7]
\tkzInit[ymin=-1.5,ymax=3,xmin=-1,xmax=3.5]
\tkzClip
\tkzDefPoint(0,0){B}
\tkzDefPoint(3,0){B1}
\tkzDefPoint(3.5,0){C}
\tkzDefPoint(0,1.7){D}
\tkzDefPoint(0,2.5){E}
\tkzDefPoint(3,1.7){F}
\tkzDefPoint(1.5,0){G}
\tkzDefPoint(1.5,1.7){H}

\tikzset{segment/.style={->}}
\tkzDrawSegment[segment](B,C)
\tkzDrawSegment[segment](B,E)
\tkzDrawSegment[color=red, very thick](B,G)
\tkzDrawSegment[color=red,very thick](H,F)
\tkzDrawSegment[dashed,color=black!50](D,F)
\tkzDrawPoints(D,G,B1)

\tkzLabelPoint(B){$0$}
\tkzLabelPoint[below](G){$\frac{L}{2}$}
\tkzLabelPoint[below](B1){$L$}
\tkzLabelPoint[left](D){$a_0$}
\node[below] at (1.5,-0.7){$t=0$};
\end{tikzpicture}
\hspace{2mm}
\begin{tikzpicture}
[scale=0.7]
\tkzInit[ymin=-1.5,ymax=3,xmin=-1.2,xmax=3.5]
\tkzClip
\tkzDefPoint(0,0){B}
\tkzDefPoint(3,0){B1}
\tkzDefPoint(3.5,0){C}
\tkzDefPoint(0,2){D}
\tkzDefPoint(0,2.5){E}
\tkzDefPoint(3,2){F}
\tkzDefPoint(1.5,0){G}
\tkzDefPoint(1.5,2){H}

\tikzset{segment/.style={->}}
\tkzDrawSegment[segment](B,C)
\tkzDrawSegment[segment](B,E)
\tkzDrawSegment[color=red, very thick](B,G)
\tkzDrawSegment[color=red,very thick](H,F)
\tkzDrawSegment[dashed,color=black!50](D,F)
\tkzDrawPoints(D,G,B1)

\tkzLabelPoint(B){$0$}
\tkzLabelPoint[below](G){$\frac{L}{2}$}
\tkzLabelPoint[below](B1){$L$}
\tkzLabelPoint[left](D){$a_1$}
\node[below] at (1.5,-0.7){$t=1$};
\end{tikzpicture}
\caption{$\Gamma_0= \{L/2\}$, $J_{u_1}=\{L/2\}$ and $\Gamma_1 = \{L/2\}$}
\end{subfigure}
\hfill
\begin{subfigure}[!h]{0.4\linewidth}
\begin{tikzpicture}
[scale=0.7]
\tkzInit[ymin=-1.5,ymax=3,xmin=-1,xmax=3.5]
\tkzClip
\tkzDefPoint(0,0){B}
\tkzDefPoint(3,0){B1}
\tkzDefPoint(3.5,0){C}
\tkzDefPoint(0,1.7){D}
\tkzDefPoint(0,2.5){E}
\tkzDefPoint(3,1.7){F}
\tkzDefPoint(1.5,0){G}
\tkzDefPoint(1.5,1.7){H}

\tikzset{segment/.style={->}}
\tkzDrawSegment[segment](B,C)
\tkzDrawSegment[segment](B,E)
\tkzDrawSegment[color=red, very thick](B,G)
\tkzDrawSegment[color=red,very thick](H,F)
\tkzDrawSegment[dashed,color=black!50](D,F)
\tkzDrawPoints(D,G,B1)

\tkzLabelPoint(B){$0$}
\tkzLabelPoint[below](G){$\frac{L}{2}$}
\tkzLabelPoint[below](B1){$L$}
\tkzLabelPoint[left](D){$a_0$}
\node[below] at (1.5,-0.7){$t=0$};
\end{tikzpicture}
\hspace{2mm}
\begin{tikzpicture}
[scale=0.7]
\tkzInit[ymin=-1.5,ymax=3,xmin=-1,xmax=3.5]
\tkzClip
\tkzDefPoint(0,0){B}
\tkzDefPoint(3,0){B1}
\tkzDefPoint(3.5,0){C}
\tkzDefPoint(0,2){D}
\tkzDefPoint(0,2.5){E}
\tkzDefPoint(3,2){F}
\tkzDefPoint(0,2){H}
\tkzDefPoint(1.5,1){X}

\tikzset{segment/.style={->}}
\tkzDrawSegment[segment](B,C)
\tkzDrawSegment[segment](B,E)
\tkzDrawSegment[color=red,thick](B,F)
\tkzDrawSegment[dashed,color=black!50](H,F)
\tkzDrawPoints(H,B1)

\tkzLabelPoint(B){$0$}
\tkzLabelPoint[below](B1){$L$}
\tkzLabelPoint[left](D){$a_1$}
\node[below] at (1.5,-0.5){$t=1$};
\end{tikzpicture}
\caption{$\Gamma_0=\{L/2\}$, $J_{u_1}= \emptyset$ and $\Gamma_1 = \{L/2\}$}
\end{subfigure}
\caption{Possible values of $\Gamma_0$, $J_{u_1}$ and $\Gamma_1$.}\label{figure}
\end{figure}
\end{remark}

\begin{remark}
The item $(iii)$ of \autoref{Theorem 1} is called the \textit{equipartition of energy principle} and is to be compared with the equality case of the arithmetico-geometric inequality 
\[\dfrac{(1-v)^2}{\varepsilon} + \varepsilon |v'|^2 \geqslant 2 |1-v| |v'|\]
used in the proof of Gamma-convergence of $AT_\varepsilon$ to $MS$.
\end{remark}

\begin{remark}
Let us stress that the implication in $(ii)$ is not an equivalence, and that the converse is still an open question. However, the fact that $\alpha >0$ is equivalent to the existence of a so-called $v-$jump for the sequence $(v_\varepsilon)$ defined as follows : 

\begin{definition}\label{v-jump}
A point $x \in [0,L]$ is a $v-$jump if there exists a sequence $(x_\varepsilon)$ such that $x_\varepsilon \rightarrow x$ and $v_\varepsilon (x_\varepsilon) \rightarrow 0$. 
\end{definition}

It turns out that a $v-$jump does not necessarily lead to a discontinuity point in the limit. 
The variations and the symmetry of the function $v_{0,\varepsilon}$ around $L/2$ imply that the only possible $v-$jump for $(v_{0,\varepsilon})$ is $L/2$. This symmetry disappears at time $1$, as $x_{min}$ is not necessarily equal to $L/2$.  
\end{remark}

\begin{remark}
A similar result holds for a sequence $(u_\varepsilon, v_\varepsilon)$ of critical points of $\overline{AT}_\varepsilon$ with Neumann conditions on $v_\varepsilon$, which is the framework of \cite[Theorem 2.2]{FLS09}. More precisely, if $\Gamma_0$ and $\Gamma_1$ are defined as in \eqref{crack} and \eqref{crack1}, \cite[Theorem 2.2]{FLS09}	states that $\Gamma_0 = \emptyset$ or $\Gamma_0$ is a regular subdivision of $(-L,L)$. It can be proved that 

\begin{itemize}
\item[$\bullet$] in the former case, $\Gamma_1= \{x_0, \ldots , x_{n+1}\}$, where $x_0=0, x_{n+1} = L$ and $x_i \in \left [\dfrac{(2i-1)L}{2n} , \dfrac{(2i+1)L}{2n}\right ]$ for $1\leqslant i \leqslant n$ or $\{x_1, \ldots , x_n\}$ where $x_i \in  \left [\dfrac{(i-1)L}{n} , \dfrac{iL}{n}\right ]$ for $1 \leqslant i \leqslant n$. Observe that, as for \autoref{Theorem 1}, the symmetry of the ``perfect staircase'' observed in \cite[Remark 2.3]{FLS09} is lost at time $1$. 
\item[$\bullet$] in the latter case, $\Gamma_1= \Gamma_0$.  
\end{itemize} 
\end{remark}


We would like to close this section by explaining what happens if we intend to iterate \autoref{Theorem 1} for $t=2$. Denote by $(u_{1,\varepsilon}, v_{1,\varepsilon})$ a sequence of critical points at time $1$, i.e satisfying $v_{1,\varepsilon} \leqslant v_{0,\varepsilon}$, $\mu_{1, \varepsilon}$ the right-hand side of the first equation of \eqref{criticalpoint1} and $(u_{2,\varepsilon},v_{2,\varepsilon})$ a sequence of critical points at time $2$, i.e satisfying $v_{2,\varepsilon} \leqslant v_{1,\varepsilon}$. The right hand-side in the first equation of \eqref{criticalpoint1} corresponding to $v_{2,\varepsilon}$ is a non positive measure $\mu_{2,\varepsilon}$. The expression of $\mu_{2,\varepsilon}$ now depends both on $\{v_{0,\varepsilon} = v_{1,\varepsilon}\}$ and $\{v_{1,\varepsilon} = v_{2,\varepsilon}\}$, so that we do not always have a similar result as \autoref{exprmu} in the case $|\{v_{1,\varepsilon} = v_{2,\varepsilon}\}| \neq 0$. As a consequence, the behavior of $(u_{2,\varepsilon},v_{2,\varepsilon})$ is less predictable, in the sense that we do not have an analogous of item (i) of \autoref{Theorem 1}. A possible reason to that phenomenon is that, as we consider critical points and not only minimizers, we do not have an energy balance at our disposal, which is crucial in the iteration argument proposed in \cite[Theorem 2.2]{G05}.

\section{Critical points of $\overline{AT}_\varepsilon$}

\subsection{$C^{1,1}$ regularity of the critical points} 

In this subsection, we investigate the regularity of $v_\varepsilon$. More precisely, we establish the following proposition : 

\begin{proposition}\label{regularityveps}
Let $(u_\varepsilon, v_\varepsilon)$ be a critical point of $\overline{AT}_\varepsilon$ in $H^1(0,L) \times H^1(0,L)$. Then $v_\varepsilon \in C^{1,1}([0,L])$. 
\end{proposition}

The methods of the proof are similar to those of the classical obstacle problem, see e.g. \cite[Theorem 5.5]{FO22}.
Set $w_\varepsilon =v_{0,\varepsilon}-v_\varepsilon$, which is a non-negative function on $[0,L]$. Note that by Sobolev embedding, $v_\varepsilon \in C^{0,1/2}([0,L])$, so that, using the equation \eqref{eq : eq2}, $u_\varepsilon \in C^{1,1/2}([0,L])$. In particular, $u_\varepsilon '$ is bounded on $[0,L]$. On $(0,L)$, $w_\varepsilon$ satisfies the following equation  

\[\begin{array}{lll}
\displaystyle\varepsilon w_\varepsilon '' &=& \displaystyle\varepsilon (v_{0,\varepsilon}''-v_\varepsilon'') = \frac{v_{0,\varepsilon}-1 - (v_\varepsilon-1)}{\varepsilon} + v_{0,\varepsilon}|u_{0,\varepsilon}'|^2 - v_\varepsilon |u_\varepsilon '|^2 +\mu_\varepsilon\\
&&\\
&=& \displaystyle\frac{w_\varepsilon}{\varepsilon} + (v_{0,\varepsilon}-v_\varepsilon)|u_{0,\varepsilon}'|^2 + v_\varepsilon (|u_{0,\varepsilon}'|^2 - |u_\varepsilon'|^2) +\mu_\varepsilon \\
&&\\
&=& \displaystyle\frac{w_\varepsilon}{\varepsilon} + w_\varepsilon |u_{0,\varepsilon}'|^2 + v_\varepsilon\left(\frac{c_{0,\varepsilon}^2}{(\eta_\varepsilon + v_{0,\varepsilon}^2)^2} - \frac{c_\varepsilon^2}{(\eta_\varepsilon + v_\varepsilon^2)^2}\right) + \mu_\varepsilon \\
\end{array}\]
Since $\mu_\varepsilon \leqslant 0$, we obtain the following expressions  
\begin{equation}
\label{eq : eq5}
\left\{\begin{array}{rlll}
\varepsilon w_\varepsilon'' &\leqslant & \displaystyle w_\varepsilon \left(\frac{1}{\varepsilon} +|u_{0,\varepsilon}'|^2\right) + v_\varepsilon\left(\frac{c_{0,\varepsilon}^2}{(\eta_\varepsilon + v_{0,\varepsilon}^2)^2} - \frac{c_\varepsilon^2}{(\eta_\varepsilon + v_\varepsilon^2)^2}\right) & \text{on } [0,L], \\
\varepsilon w_\varepsilon '' &=& \displaystyle w_\varepsilon \left(\frac{1}{\varepsilon} +|u_{0,\varepsilon}'|^2\right) + v_\varepsilon\left(\frac{c_{0,\varepsilon}^2}{(\eta_\varepsilon + v_{0,\varepsilon}^2)^2} - \frac{c_\varepsilon^2}{(\eta_\varepsilon + v_\varepsilon^2)^2}\right) & \text{on } O_\varepsilon .\\
\end{array}\right.
\end{equation}
Since $v_{0,\varepsilon}$ is smooth, it is sufficient to prove that $w_\varepsilon$ is in $C^{1,1}([0,L])$. We first need to prove that $w_\varepsilon$ grows at most quadratically around $\partial O_\varepsilon$ (\autoref{quadratic growth}). Then, we can use Schauder estimates on \eqref{eq : eq5} to ensure that $w_\varepsilon '$ is continuous at $\partial O_\varepsilon$.  

\begin{lemma}\label{quadratic growth}
There exists a constant $C >0$ only depending on $ \|u_{0,\varepsilon}' \|_{\infty, [0,L]},  \|u_\varepsilon ' \|_{\infty, [0,L]}$ and $\varepsilon$ such that for any $x$ such that $d(x, \partial O_\varepsilon) \leqslant \dfrac{1}{3} d(x, \{0,L\})$, one has 

\[0\leqslant w_\varepsilon(x) \leqslant C d(x,\partial O_\varepsilon)^2. \]
\end{lemma}

\begin{proof}
Since $0\leqslant v_\varepsilon \leqslant v_{0,\varepsilon}\leqslant 1$, $w_\varepsilon$ is bounded. As $u_\varepsilon '$ and $u_{0,\varepsilon}'$ are also bounded on $[0,L]$, the first inequality of the system \eqref{eq : eq5} gives
\begin{equation}
\label{boundw''}
\varepsilon w_\varepsilon '' \leqslant \frac{1}{\varepsilon} w_\varepsilon  + |u_{0,\varepsilon} '|^2 w_\varepsilon + v_\varepsilon (|u_{0,\varepsilon} '|^2 - |u_\varepsilon '|^2) \leqslant C,
\end{equation} 
where this inequality is written in the distributional sense. 

\smallskip

Fix any $z\in (0,L)$. The mean value formula applied to the superharmonic function $\varepsilon w_\varepsilon - C\dfrac{|z-\cdot |^2}{2}$ yields, for any $r$ such that $I_r(z) \subset [0,L]$, 
\begin{equation}
\label{eq : eq6}
\varepsilon w_\varepsilon (z) \geqslant \frac{1}{2r} \int_{I_r(z)} \left(\varepsilon w_\varepsilon (y) -  C \frac{|z-y|^2}{2}\right) dy\geqslant \frac{\varepsilon}{2r} \int_{I_r(z)} w_\varepsilon (y) dy -  C r^2
\end{equation}
Secondly, for any $y\in O_\varepsilon$, the second equation of \eqref{eq : eq5} yields 
\[\varepsilon w_\varepsilon ''(y) \geqslant - v_\varepsilon (y) \frac{c_\varepsilon ^2}{(\eta_\varepsilon + |v_\varepsilon |^2)^2} \geqslant - \frac{c_\varepsilon ^2 }{\eta_\varepsilon ^2} = - C'.\]
Fix then $z\in O_\varepsilon$. The mean value formula applied to the subharmonic function $\varepsilon w_\varepsilon + C' \dfrac{|z-\cdot |^2}{2}$ yields for any $r$ such $I_r(z) \subset O_\varepsilon$, 
\begin{equation}
\label{eq : eq7}
\varepsilon w_\varepsilon (z) \leqslant \frac{1}{2r} \int_{I_r(z)} \left(\varepsilon w_\varepsilon(y) + C' \frac{|z-y|^2}{2} \right) dy  \leqslant \frac{\varepsilon}{2r} \int_{I_r(z)} w_\varepsilon (y) dy + C' r^2.
\end{equation}
Now fix $x\in O_\varepsilon$ such that $r:= d(x, \partial O_\varepsilon)$ satisfies $r\leqslant \dfrac{1}{3} d(x, \{0, L\})$ and take any $x'\in \overline{I_r(x)} \cap \partial O_\varepsilon$. Then $I_{2r}(x') \subset I_{3r}(x)\subset (0,L)$ and $x'\in F_\varepsilon$ so, using \eqref{eq : eq6},  
\begin{equation}
\label{eq : eq8}
0 = \varepsilon w_\varepsilon (x') \geqslant \frac{1}{4r} \int_{I_{2r}(x')} \varepsilon w_\varepsilon(y) dy - Cr^2.
\end{equation}
Then, since $w_\varepsilon \geqslant 0$ and $I_r (x) \subset I_{2r}(x')$, we get 
\[\varepsilon w_\varepsilon(x) - C' r^2 \overset{\eqref{eq : eq7}}{\leqslant} \frac{\varepsilon}{2r} \int_{I_r(x)} w_\varepsilon(y) dy \leqslant \frac{\varepsilon}{2r} \int_{I_{2r}(x')} w_\varepsilon(y) dy \overset{\eqref{eq : eq8}}{\leqslant} 2Cr^2.\]
which is the desired inequality. 
\end{proof}

We are now ready to prove \autoref{regularityveps}.

\begin{proof}[Proof of \autoref{regularityveps}]
First, we focus on the regularity on $(0,L)$. 
Let $x_1 \in O_\varepsilon$ and take $x_0 \in \partial O_\varepsilon$ such that $r:=|x_0 -x_1| = d(x_1, \partial O_\varepsilon)$ and suppose that $r < \dfrac{1}{3} d(x_1 , \{0,L\})$. Notice that the right-hand side of the second equation of \eqref{eq : eq5} belongs to $C^{0,1/2}([0,L])$. Applying Schauder estimates (\cite[Theorem 4.6]{GT01}) on $I_r(x_1) \subset O_\varepsilon$, there exists a constant $C$ independent of $x_1, r$ and $w_\varepsilon$ such that
\begin{multline}
\label{schauder}
r  \|w_\varepsilon ' \|_{\infty, I_{r/2}(x_1)} + r^2  \|w_\varepsilon '' \|_{\infty, I_{r/2}(x_1)} \\
\leqslant C\left(\underbrace{ \|w_\varepsilon  \|_{\infty, I_r(x_1)}}_{:= I} + r^2 \underbrace{\left|\left|w_\varepsilon\left(\frac{1}{\varepsilon} + |u_\varepsilon '|^2\right)\right|\right|_{C^{0,1/2},I_r(x_1)}}_{:=II} + r^2\underbrace{ \|v_\varepsilon (|u_{0,\varepsilon} '|^2 - |u_\varepsilon '|^2) \|_{C^{0,1/2},I_r(x_1)}}_{:= III} \right).
\end{multline}
We now estimate the terms in the right hand side of \eqref{schauder}.

\medskip

\underline{$I$ :} Using \autoref{quadratic growth}, one has $ \|w_\varepsilon  \|_{\infty, I_r(x_1)} \leqslant Cd(x_1,\partial O_\varepsilon )^2 = C r^2$.

\medskip

\underline{$II$ :} We split the norm $ \|\, \cdot \, \|_{C^{0,1/2}}$ into the norm $ \|\, \cdot \,  \|_\infty$ and the semi-norm $[\,\cdot \,]_{1/2}$. Concerning the norm $ \|\, \cdot \, \|_\infty$, one has 

\[ \left|\left|w_\varepsilon\left(\frac{1}{\varepsilon} + |u_\varepsilon '|^2\right)\right|\right|_{\infty, I_r(x_1)} \leqslant 2\left(\frac{1}{\varepsilon} +  \|u_{\varepsilon} ' \|_{\infty, [0,L]}^2\right) = C,\]
where $C >0$ is a constant which depends on $ \|u_\varepsilon ' \|_{\infty, [0,L]}$. Concerning the semi-norm $[\, \cdot \,]$, we use the inequality $[|u_\varepsilon '|^2]_{1/2} \leqslant 2 \|u_\varepsilon ' \|_\infty [u_\varepsilon ']_{1/2}$. This yields  
\[\begin{array}{lll}
\displaystyle\left[w_\varepsilon\left(\frac{1}{\varepsilon} + |u_\varepsilon '|^2\right)\right]_{1/2, I_r(x_1)} &\leqslant & \displaystyle  \|w_\varepsilon \|_{\infty, I_r(x_1)} \left[\frac{1}{\varepsilon} + |u_\varepsilon '|^2\right]_{1/2, I_r(x_1)} +  [w_\varepsilon]_{1/2,I_r(x_1)}\left|\left|\frac{1}{\varepsilon} + |u_\varepsilon '|^2\right|\right|_{\infty, I_r(x_1)} \\
&&\\
&\leqslant & \displaystyle 2 \|u_\varepsilon ' \|_{\infty, [0,L]} [u_\varepsilon ']_{1/2, [0,L]} + [w_\varepsilon]_{1/2,[0,L]}\left|\left|\frac{1}{\varepsilon} + |u_\varepsilon '|^2\right|\right|_{\infty, [0,L]} \\
&&\\
&\leqslant & \displaystyle C\\
\end{array}\]
where $C>0$ is a constant depending on $ \|u_\varepsilon ' \|_{C^{0,1/2}, [0,L]}$ and $ \|w_\varepsilon  \|_{C^{0,1/2}, [0,L]}$.

\medskip
 
\underline{$III$ :} Similarly, one has 
\[ \|v_\varepsilon (|u_{0,\varepsilon} '|^2 - |u_\varepsilon '|^2) \|_{\infty, I_r(x_1)} \leqslant  \|u_{0,\varepsilon} ' \|_{\infty, [0,L]} ^2 +  \|u_\varepsilon ' \|_{\infty, [0,L]} ^2,\]
and 
\[\begin{array}{ll}
&[v_\varepsilon (|u_{0,\varepsilon} '|^2 - |u_\varepsilon '|^2)]_{1/2, I_r(x_1)}\\
&\\
\leqslant & \displaystyle  \|v_\varepsilon \|_{\infty, I_r(x_1)} ([|u_{0,\varepsilon} '|^2]_{1/2, I_r(x_1)} + |u_\varepsilon '|^2_{1/2, I_r(x_1)}] + [v_\varepsilon]_{1/2, I_r(x_1)}( \|u_{0,\varepsilon} ' \|^2_{\infty, I_r(x_1)} + \|u_\varepsilon ' \|^2_{\infty, I_r(x_1)}) \\
&\\
\leqslant & \displaystyle 2 \|u_{0,\varepsilon}' \|_{\infty,[0,L]} [u_{0,\varepsilon} ']_{1/2, [0,L]} + 2 \|u_{\varepsilon}' \|_{\infty,[0,L]} [u_{\varepsilon} ']_{1/2, [0,L]} + [v_\varepsilon]_{1/2, [0,L]}( \|u_{0,\varepsilon} ' \|^2_{\infty, [0,L]} + \|u_\varepsilon ' \|^2_{\infty, [0,L]}) \\
&\\
\leqslant & \displaystyle C\\
\end{array},\]
where $C>0$ is a constant depending only on $ \|u_{0,\varepsilon} ' \|_{C^{0,1/2}, [0,L]},  \|u_{\varepsilon} ' \|_{C^{0,1/2}, [0,L]}$ and $ \|v_{\varepsilon} ' \|_{C^{0,1/2}, [,L]}$. 
It follows from \eqref{schauder} that
\[r  \|w_\varepsilon ' \|_{\infty, I_{r/2}(x_1)} + r^2  \|w_\varepsilon '' \|_{\infty, I_{r/2}(x_1)} \\
\leqslant C r^2.\]
In particular, $|w_\varepsilon '(x_1)| \leqslant  \|w_\varepsilon ' \|_{\infty, I_{r/2} (x_1)} \leqslant C d(x_1, \partial O_\varepsilon)$. This means that $w_\varepsilon ' (x_1)$ tends to $0$ as $d(x_1,\partial O_\varepsilon) \rightarrow 0$. 
Moreover, as $w_\varepsilon$ vanishes on $F_\varepsilon$, $w_\varepsilon '=0$ a.e on $F_\varepsilon$ (see \cite[Lemma 7.7]{GT01}). 
So $w_\varepsilon '$ admits a representative on $(0,L)$ that is zero on $F_\varepsilon$ and is continuous on $(0,L)$, so that $w_\varepsilon \in C^1(0,L)$.  

\smallskip

From \eqref{boundw''}, we also have
\[ \|w_\varepsilon '' \|_{\infty,(0,L)} \leqslant C_2,\]
where $C_2 >0$ is a constant depending on $ \|u_{0,\varepsilon}' \|_{\infty, [0,L]}$, from which we deduce that $w_\varepsilon '$ is Lipschitz on $(0,L)$. 

\smallskip

To recover the regularity at the boundary points $0$ and $L$, we see that $w_\varepsilon '$ is uniformly continuous on $(0,L)$(because it is Lipschitz) so it admits a continuous extension on $[0,L]$. This implies that $w_\varepsilon$ is differentiable at $0$ and $L$ with $w_\varepsilon '(0) = \lim\limits_{x\rightarrow 0^+} w_\varepsilon '(x)$ and  $w_\varepsilon '(L) = \lim\limits_{x\rightarrow L^-} w_\varepsilon '(x)$. We deduce that $w_\varepsilon \in \, C^1([0,L])$. Then, as $w_\varepsilon '$ is Lipschitz on $(0,L)$, its extension to $[0,L]$ is Lipschitz as well so that $w_\varepsilon \in C^{1,1}([0,L])$. Hence $v_\varepsilon \in C^{1,1}([0,L])$. 
%
%
%
%
%
\end{proof}
 
The following result gives an explicit pointwise formula for $\mu_\varepsilon$.

\begin{corollary}\label{exprmu}
Let $\mu_\varepsilon$ be the right-hand side in \eqref{eq : eq4}. We have the following expression for $\mu_\varepsilon$,
\begin{equation}
\label{eq : mu_e}
\mu_\varepsilon = \frac{v_{0,\varepsilon}}{[(v_{0,\varepsilon})^2+\eta_\varepsilon]^2} (( c_\varepsilon)^2 - (c_{0,\varepsilon})^2) 1_{F_\varepsilon} = v_{0,\varepsilon}(|u_{\varepsilon}'|^2 - |u_{0,\varepsilon}'|^2)1_{F_\varepsilon}
\end{equation} 
In particular, if $|F_\varepsilon|\neq 0$, then $|c_{\varepsilon}| \leqslant c_{0,\varepsilon}$. 
\end{corollary}

\begin{proof}
As $v_\varepsilon \in C^{1,1}([0,1])$, $\mu_\varepsilon$ is absolutely continuous with respect to the Lebesgue measure. Moreover, as $v_\varepsilon = v_{0,\varepsilon}$ on $F_\varepsilon$, we have $v_\varepsilon '= (v_{0,\varepsilon})'$ and $v_\varepsilon ''= (v_{0,\varepsilon})''$ a.e on $F_\varepsilon$ (see \cite[Lemma 7.7]{GT01}). Hence 

\[\begin{array}{lllll}
\displaystyle\mu_\varepsilon &=& \displaystyle \left(-\varepsilon v_\varepsilon'' +\frac{v_\varepsilon-1}{\varepsilon} + v_\varepsilon |u_\varepsilon'|^2\right)1_{F_\varepsilon} &=& \displaystyle\left(-\varepsilon v_{0,\varepsilon}'' + \frac{v_{0,\varepsilon} -1}{\varepsilon} + v_{0,\varepsilon} |u_\varepsilon'|^2\right) 1_{F_\varepsilon} \\
&&\\
&=& \displaystyle\left( - v_{0,\varepsilon} (u_{0,\varepsilon}')^2 + v_{0,\varepsilon} |u_\varepsilon'|^2\right) 1_{F_\varepsilon} &=& \displaystyle v_{0,\varepsilon}(|u_\varepsilon'|^2 - |u_{0,\varepsilon}'|^2) 1_{F_\varepsilon} \\
&&\\
&=&\displaystyle v_{0,\varepsilon} \left( \frac{c_{\varepsilon}^2}{(v_\varepsilon)^2 + \eta_\varepsilon} - \frac{c_{0,\varepsilon}^2}{(v_{0,\varepsilon})^2 + \eta_\varepsilon} \right) 1_{F_\varepsilon}
&=& \displaystyle \frac{v_{0,\varepsilon}}{(v_{0,\varepsilon}^2+\eta_\varepsilon)^2} (c_{\varepsilon}^2 - c_{0,\varepsilon}^2) 1_{F_\varepsilon} \\
\end{array} .\]
As $\mu_\varepsilon \leqslant 0$, provided that $|F_\varepsilon| \neq 0$, one has $c_\varepsilon ^2 \leqslant c_{0,\varepsilon} ^2$, and $|c_\varepsilon| \leqslant c_{0,\varepsilon}$. 
\end{proof}

\subsection{Shape of $v_\varepsilon$}

The aim of this section is to take advantage of the ODE \eqref{eq : eq4} satisfied by $v_\varepsilon$ to determine its variations. We prove that $v_\varepsilon$ admits a unique global minimum, denoted by $x_\varepsilon \in [L/4, 3L/4]$ and that $v_\varepsilon$ is non-increasing on $[0,x_\varepsilon]$ and non-decreasing on $[x_\varepsilon, L]$. We start with the following lemma on the possible positions of the critical points :  

\begin{lemma}\label{pointscritiquesF}
Let $x_\varepsilon$ be a critical point of $v_\varepsilon$. If $x_\varepsilon\in F_\varepsilon$, then $x_\varepsilon=L/2$. 
\end{lemma}

\begin{proof}
Note that statement is non empty, by Rolle's Theorem, as $v_\varepsilon (0)= v_\varepsilon (L)$.
Suppose $x_\varepsilon \in F_\varepsilon\setminus \{0,L\}$. Then

\[0= v_\varepsilon '(x_\varepsilon) = \lim\limits_{h\rightarrow 0^+} \frac{v_\varepsilon(x_\varepsilon+h)-v_\varepsilon(x_\varepsilon)}{h} = \lim\limits_{h\rightarrow 0^+} \frac{v_\varepsilon(x_\varepsilon+h)-v_{0,\varepsilon}(x_\varepsilon)}{h} \leqslant \lim\limits_{h\rightarrow 0^+} \frac{v_{0,\varepsilon}(x_\varepsilon+h)-v_{0,\varepsilon}(x_\varepsilon)}{h} = v_{0,\varepsilon} '(x)\]
and 
\[0= v_\varepsilon'(x_\varepsilon) = \lim\limits_{h\rightarrow 0^+} \frac{v_\varepsilon(x_\varepsilon)-v_\varepsilon(x_\varepsilon-h)}{h} = \lim\limits_{h\rightarrow 0^+} \frac{v_{0,\varepsilon}(x_\varepsilon)-v_\varepsilon(x_\varepsilon-h)}{h} \geqslant \lim \limits_{h\rightarrow 0^+}\frac{v_{0,\varepsilon}(x_\varepsilon)-v_{0,\varepsilon}(x_\varepsilon-h)}{h} = v_{0,\varepsilon}'(x)\]
so that $v_{0,\varepsilon}'(x_\varepsilon)=0$, which implies $x_\varepsilon=L/2$ by \cite[Proposition 2.1]{BMRb23}. If $x_\varepsilon=0$, then, using again \cite[Proposition 2.1]{BMRb23}, 

\[0> v_{0\varepsilon}'(0) = \lim\limits_{h\rightarrow 0^+} \frac{v_{0,\varepsilon}(h)-v_{0,\varepsilon}(0)}{h} \geqslant \lim\limits_{h\rightarrow 0^+} \frac{v_\varepsilon(h) - 1}{h} = v_\varepsilon'(0)\]

and if $x_\varepsilon= L$,  

\[0 < v_{0,\varepsilon}'(L) = \lim\limits_{h\rightarrow 0^+} \frac{v_{0,\varepsilon}(L)-v_{0,\varepsilon}(L-h)}{h} \leqslant \lim\limits_{h\rightarrow 0^+} \frac{1 - v_\varepsilon(L-h)}{h} = v_\varepsilon'(L) \]

so that $0$ and $L$ are not critical points of $v_\varepsilon$. 
\end{proof}

The following proposition establishes the variations of $v_\varepsilon$ on $[0,L]$.

\begin{proposition}\label{variations of v}
The function $v_\varepsilon$ admits a unique local minimum at $x_\varepsilon$ on $[0,L]$. This local minimum is also a global minimum and is either located at $L/2$ or it belongs to the connected component of $O_\varepsilon$ containing $L/2$. 
\end{proposition}

Observe that while $v_{0,\varepsilon}$ is symmetric with respect to $L/2$, the previous proposition states that it is not necessarily the case for $v_\varepsilon$. 

\begin{proof}
According to \autoref{pointscritiquesF}, a local minimum of $v_\varepsilon$ is either $L/2$ or in $O_\varepsilon$. 
The proof is divided in two steps. In the first step, we prove that any connected component of $O_\varepsilon$ contains at most one critical point $y$. Using this first step, we prove in the second step that the only connected component of $O_\varepsilon$ that contains a local minimum is (if it exists) the connected component containing $L/2$. 

\smallskip

\underline{Step $1$ :} Any connected component of $O_\varepsilon$ contains at most one critical point. 

\smallskip

Take $(\beta, \gamma)$ a non-empty connected component of $O_\varepsilon$ and suppose that it contains a critical point of $v_\varepsilon$. Assume that $v_\varepsilon (\beta) \leqslant v_\varepsilon (\gamma)$ (the case $v_\varepsilon (\beta) \geqslant v_\varepsilon(\gamma)$ can be treated similarly). As $\beta, \gamma \in F_\varepsilon$, we also have $v_{0,\varepsilon} (\beta) \leqslant v_{0,\varepsilon} (\gamma)$. As $v_{0,\varepsilon}$ is decreasing on $[0,L/2]$ and increasing on $[L/2,L]$ (\cite[Proposition 2.1]{BMRb23}), we have in particular that $\gamma > L/2$ and $v_{0,\varepsilon}' (\gamma) >0$.

Assume that $\{v_\varepsilon ' =0\} \cap (\beta, \gamma)$ is non-empty. First, $\gamma$ cannot be critical points of $v_\varepsilon$, as 

\[v_\varepsilon' (\gamma) = \lim\limits_{h\rightarrow 0} \frac{v_\varepsilon (\gamma) - v_\varepsilon (\gamma-h)}{h} \geqslant \lim\limits_{h\rightarrow 0} \frac{v_{0,\varepsilon} (\gamma) - v_{0,\varepsilon} (\gamma-h)}{h} = v_{0,\varepsilon} '(\gamma) >0.\]

Hence, the supremum $y$ of $\{v_\varepsilon ' =0\} \cap (\beta,\gamma )$ belongs to $(\beta,\gamma )$. First, suppose that $y > \frac{\beta+\gamma}{2}$. The Cauchy problem 
\[\left\{\begin{array}{rlll}
\displaystyle -\varepsilon v_\varepsilon '' + \frac{v_\varepsilon -1}{\varepsilon} + v_\varepsilon |u_\varepsilon '|^2 &=& 0 & \text{on } (\beta, \gamma),\\
\displaystyle v_\varepsilon (y) &=& v_\varepsilon (y), &\\
\displaystyle v_\varepsilon '(y) &=& 0, &\\
\end{array}\right.\]
admits a unique solution on the maximal interval $(\beta,\gamma)$ (by Cauchy-Lipschitz Theorem). But the function $v_\varepsilon(2y-\cdot )$, defined on $(2y-\gamma, \gamma)$, is also a solution of the Cauchy problem on $(2y-\gamma, \gamma) \subset (\beta,\gamma)$, so that by uniqueness, $v_\varepsilon$ is symmetric around $y$ on $(2y-\gamma,\gamma)$ (see \autoref{figure2a}). By continuity, $v_\varepsilon$ is symmetric around $y$ on $[2y-\gamma,\gamma]$ and
\[v_{0,\varepsilon} (2y-\gamma) \geqslant v_\varepsilon (2y-\gamma) = v_\varepsilon (\gamma) = v_{0,\varepsilon} (\gamma).\]
Since $2y-\gamma \in (\beta, \gamma)$, it follows from the variations of $v_{0,\varepsilon}$ (see \cite[Proposition 2.1]{BMRb23}) that $v_{0,\varepsilon}(2y-\gamma) < \max(v_{0,\varepsilon}(\beta), v_{0,\varepsilon} (\gamma))= v_{0,\varepsilon} (\gamma)$.  
This contradiction shows that 
\begin{equation}
\label{ineq min global}
y\leqslant \dfrac{\beta+\gamma}{2}.
\end{equation}
Now, suppose that there exists another critical point $z$ of $v_\varepsilon$ on $(\beta,\gamma)$. As $y$ is supposed to be the largest critical point, $z<y$. Similarly to the previous argument of symmetry, $v_\varepsilon (2y - \cdot)$ is a solution of the Cauchy problem on $(\beta, 2y-\beta) \subset (\beta, \gamma)$ so $v_\varepsilon$ is symmetric around $y$ on $[\beta, 2y-\beta]$. It follows that $2y-z$ is also critical point of $v_\varepsilon$ on $(\beta,\gamma)$, which implies $2y-z=y$, hence $y=z$. 

\smallskip

Finally, $y$ is the unique critical point of $v_\varepsilon$ on $(\beta,\gamma)$. 

\medskip

\underline{Step $2$ :} $(\beta, \gamma)$ contains a local minimum if and only if $L/2 \in (\beta, \gamma)$. 

\smallskip

We distinguish two cases. 

\underline{Case 1 :} Suppose that $\dfrac{L}{2}\notin (\beta,\gamma )$ and that $y$ is a local minimum of $v_\varepsilon$. As $v_\varepsilon '$ does not vanish on $(\beta, \gamma) \setminus \{y\}$, $v_\varepsilon$ is decreasing on $(\beta, y)$ and increasing on $(y, \gamma)$. 

On one hand, the symmetry of $v_\varepsilon$ around $y$ yields : 
\[v'_\varepsilon (\beta) = \lim\limits_{x\rightarrow \beta^+} v_\varepsilon '(x) = -\lim\limits_{x\rightarrow (2y - \beta)^-} v_\varepsilon '(x) = -v_\varepsilon '(2y-\beta) < 0.\]
But on the other hand : 
\[v_\varepsilon '(\beta) = \lim\limits_{h\rightarrow 0} \frac{v_\varepsilon (\beta) - v_\varepsilon (\beta -h)}{h} \geqslant \lim\limits_{h\rightarrow 0} \frac{v_{0,\varepsilon} (\beta) - v_{0,\varepsilon} (\beta -h)}{h} = v_{0,\varepsilon} '(\beta) >0.\]
Thus, it follows that $(\beta, \gamma)$ does not contain any local minimum. 

\medskip

\underline{Case 2 :} Suppose that $L/2 \in (\beta, \gamma)$. Since
\[v_{\varepsilon}(L/2) \leqslant v_{0,\varepsilon} (L/2) < \min (v_{0,\varepsilon} (\beta), v_{0,\varepsilon} (\gamma)) = \min (v_\varepsilon (\beta), v_\varepsilon (\gamma)),\] 
$v_\varepsilon$ admits at least one local minimum on $(\beta, \gamma )$.
By uniqueness, this local minimum is $y$.

\smallskip

This concludes the proof. 

\end{proof}

\begin{figure}[!h]
\begin{subfigure}[!h]{0.42\linewidth}
\scalebox{1.15}{
\begin{tikzpicture}
\draw[->] (-2.5, 1) -- (2.5, 1);
\draw[->] (-2.5, 1) -- (-2.5, 5.5) ;
\draw [very thick, domain=0.52 :1.49, dashed,color=blue!80] plot (\x , {2.3*cos(5*pow(cos(\x/2.7 r +0.35 r),3.0) r)+3.95}) node[above] {\scriptsize $\textcolor{blue!80}{v_\varepsilon}$};

\draw [domain=-0.45 :0.52, dashed,color=red] plot (\x , {2.3*cos(5*pow(cos((1.04 -\x)/2.7 r +0.35 r),3.0) r)+3.95});
\node at (-0.45,5) {\scriptsize $\textcolor{red}{v_\varepsilon(2y-\cdot)}$};

\draw [very thick, domain=-0.56 :0.52, dashed,color=blue!80] plot (\x , {2.3*cos(5*cos(\x/2.7 r +0.7 r) r)+3.95});

\draw [very thick, domain=-2.5:2.5, samples=100] plot (\x,{5.0*cos(pow(cos(\x/2 r),4.0) r)}) node[right] {\footnotesize $v_{0,\varepsilon}$};

\tkzDefPoint(-0.65,3.44){B}
\tkzDefPoint(-0.45,4.79){C'}
\tkzDefPoint(-0.45,3.09){C''}
\tkzDefPoint(1.49,4.79){C}
\tkzDefPoint(0.52,1.65){D}
\tkzDefPoint(0.52,5){E}

\tkzDefPoint(-0.65,1){Xb}
\tkzDefPoint(-0.45,0.6){Xc'}
\tkzDefPoint(1.49,0.6){Xc}
\tkzDefPoint(0.52,1){Xd}
\tkzDefPoint(-2.5,4.79){Yc}
\tkzDefPoint(-2.5,3.09){Yc''}

\tkzDrawPoints(B,C,C',C'',D)
\tkzDrawSegment[dashed,color=red](Xd,E)
\tkzDrawSegment[dotted](B,Xb)
\tkzDrawSegment[dotted](C,Yc)
\tkzDrawSegment[dotted](C'',Yc'')
\tkzDrawSegment[dotted](C',Xc')
\tkzDrawSegment[dotted](C,Xc)
\tkzLabelPoint(Xb){\scriptsize $\beta$}
\tkzLabelPoint(Xc){\scriptsize $\gamma$}
\tkzLabelPoint(Xc'){\scriptsize $2y-\gamma$}
\tkzLabelPoint(Xd){\scriptsize $y$}
\tkzLabelPoint[left](Yc){\scriptsize $v_{0,\varepsilon}(\gamma)$}
\tkzLabelPoint[left](Yc''){\scriptsize $v_{0,\varepsilon}(2y-\gamma)$}
\end{tikzpicture}}
\caption{$y\leqslant \dfrac{\gamma +\beta}{2}$}\label{figure2a}
\end{subfigure}
\hfill
\begin{subfigure}[!h]{0.5\linewidth}
\scalebox{1.15}{
\begin{tikzpicture}
\draw[->] (-2.5, 1) -- (2.5, 1);
\draw[->] (-2.5, 1) -- (-2.5, 5.5) ;

\draw [very thick, domain=-0.75 :1.3, dashed,color=blue!80] plot (\x , {2.3*cos(5*pow(cos(\x/2.7 r +0.45 r),3.0) r)+3.5}) node[above] {\footnotesize $\textcolor{blue!80}{v_\varepsilon}$};

\draw [very thick, domain=-2.5:2.5, samples=100] plot (\x,{5.0*cos(pow(cos(\x/2 r),4.0) r)}) node[right] {\footnotesize $v_{0,\varepsilon}$};

\tkzDefPoint(-0.79,3.74){B}
\tkzDefPoint(1.29,4.59){C}
\tkzDefPoint(0.24,1.2){D}
\tkzDefPoint(1.06,3.74){B'}
\tkzDefPoint(0.79,3.74){B''}
\tkzDefPoint(1.06,4.26){G}
\tkzDefPoint(0.24,5){E}
\tkzDefPoint(0,2){H}
\tkzDefPoint(0,5){I}

\tkzDefPoint(-0.79,0.6){Xb}
\tkzDefPoint(0.79,0.6){Xb''}
\tkzDefPoint(1.06,1){Xb'}
\tkzDefPoint(1.29,1){Xc}
\tkzDefPoint(0.24,1){Xd}
\tkzDefPoint(-2.5,4.26){Yg}
\tkzDefPoint(-2.5,3.74){Yb}
\tkzDefPoint(-2.5,4.59){Yc}

\tkzDrawPoints(B,B',B'',G)
\tkzDrawSegment[dashed,color=red](Xd,E)
\tkzDrawSegment[dotted](B,Xb)
\tkzDrawSegment[dotted](B'',Xb'')
\tkzDrawSegment[dotted](B',Yb)
\tkzDrawSegment[dotted](G,Yg)
\tkzDrawSegment[dotted](G,Xb')

\tkzLabelPoint(Xb){\scriptsize $\beta$}
\tkzLabelPoint(Xd){\scriptsize $x_\varepsilon$}
\tkzLabelPoint[below right=0.5pt and -6pt](Xb'){\scriptsize $2x_\varepsilon-\beta$}
\tkzLabelPoint(Xb''){\scriptsize $L-\beta$}
\tkzLabelPoint[left](Yb){\scriptsize $v_{0,\varepsilon}(L-\beta)$}
\tkzLabelPoint[left](Yg){\scriptsize $v_{0,\varepsilon}(2x_\varepsilon-\beta)$}
\end{tikzpicture}}
\caption{$x_\varepsilon \leqslant \dfrac{3}{4}L$}\label{figure2b}
\end{subfigure}
\caption{Critical points of $v_\varepsilon$}\label{figure2}
\end{figure}
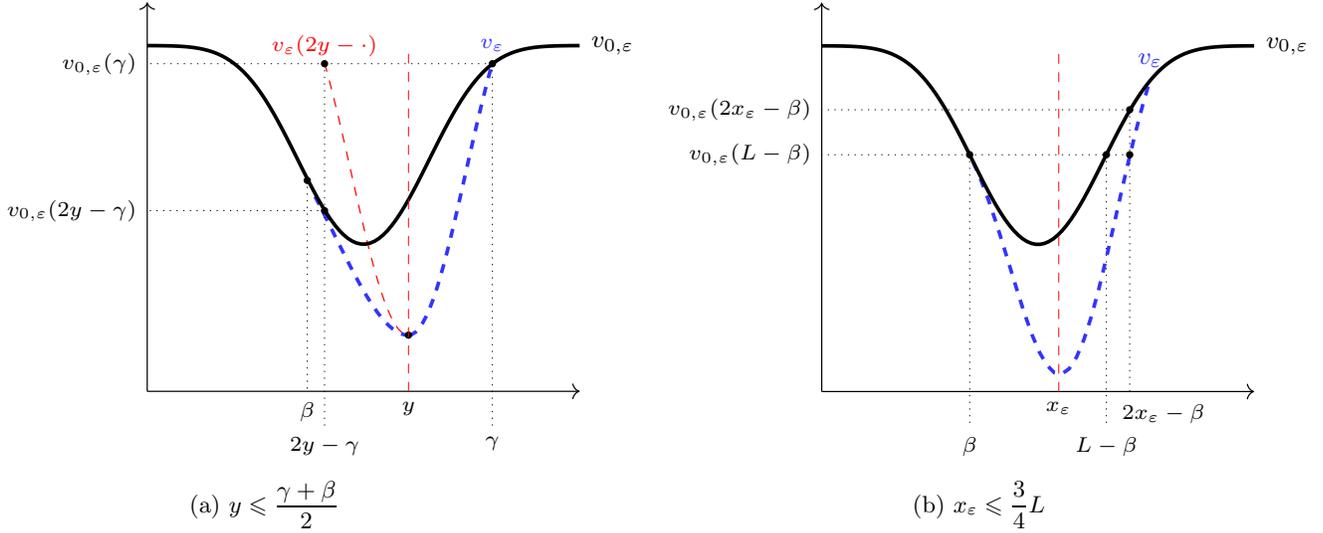

\begin{corollary}\label{variationveps}
The function $v_\varepsilon$ is non-increasing on $[0,x_\varepsilon]$ and non-decreasing on $[x_\varepsilon,L]$. 
\end{corollary}

\begin{proof}
Let us prove that $v_\varepsilon$ is non-increasing on $[0, x_\varepsilon]$. Let $a,b\in [0,x_\varepsilon]$ such that $a< b$. If $v_\varepsilon (a) < v_\varepsilon (b)$, then $v_\varepsilon (a)<\min (v_\varepsilon (0), v_\varepsilon (b))$, so $v_\varepsilon$ admits a local minimum on $(0,b)$, which contradicts \autoref{variations of v}. So $v_\varepsilon (a) \geqslant v_\varepsilon (b)$ and $v_\varepsilon$ is non increasing on $[0,x_\varepsilon]$. Similarly, $v_\varepsilon$ is non decreasing on $[x_\varepsilon ,1]$. 
\end{proof}

Notice that the sequence $(x_\varepsilon)$ of global minimum of $v_\varepsilon$ is bounded so, up to an extraction, we can suppose that
\begin{equation}\label{xmin}
x_\varepsilon \longrightarrow x_{min} \in [0,L].
\end{equation}
The previous proof implies the following on $x_{min}$ : 

\begin{corollary}\label{corollaire position xmin}
Let $x_\varepsilon$ denote the global minimum of $v_\varepsilon$ and 
let $x_{min}$ be the limit of a converging subsequence of $(x_\varepsilon)$. Then $x_{min} \in [L/4, 3L/4]$. 
\end{corollary}

\begin{proof}
Suppose that $x_\varepsilon \in \,(\beta, \gamma)$ the connected component of $O_\varepsilon$ containing $L/2$. Suppose also that $v_\varepsilon (\beta)\leqslant v_\varepsilon (\gamma)$. According to \eqref{ineq min global}, $x_\varepsilon \leqslant \dfrac{\gamma+ \beta}{2} \leqslant \dfrac{3}{4}L < L$. On the other hand, the symmetries of $v_{0,\varepsilon}$ around $L/2$ imply 
\[v_{0,\varepsilon} (2x_\varepsilon -\beta) \geqslant v_\varepsilon (2x_\varepsilon -\beta) = v_\varepsilon(\beta) =v_{0,\varepsilon}(\beta)= v_{0,\varepsilon} (L-\beta).\] 
From the variations of $v_{0,\varepsilon}$ (see \cite[Proposition 2.1]{BMRb23}) and the inequality $v_{0,\varepsilon} (2x_\varepsilon -\beta) \geqslant v_{0,\varepsilon}(\beta)$, we deduce that $2x_\varepsilon \geqslant L/2$. Then, from the the variations of $v_{0,\varepsilon}$ on $[L/2,L]$ and the inequality $v_{0,\varepsilon} (2x_\varepsilon -\beta) \geqslant v_{0,\varepsilon} (L-\beta)$, we deduce that $2x_\varepsilon -\beta \geqslant L-\beta$, so $x_\varepsilon \geqslant L/2$ (see \autoref{figure2b}). The case where $v_\varepsilon (\beta) \geqslant v_\varepsilon (\gamma)$ can be treated similarly. 
\end{proof}

In particular, $x_{min}$ is not localized at the boundaries of $[0,L]$. 

\section{Asymptotic analysis}

In this section, we prove \autoref{Theorem 1}. The proofs of the several items of \autoref{Theorem 1} are disseminated within the subsections : $(i)$ is proved in subsection 4.3, $(ii)$ in subsection 4.5 and subsection 4.7, $(iii)$ in subsection 4.4 and $(iv)$ and $(v)$ in subsection 4.6. 

\subsection{Preliminary estimates}

The following estimates are classical and can be found in \cite{FLS09}, \cite{BMRb23} and \cite{L10}. 

\smallskip

First, notice that the energy bound \eqref{energyboundtemps1} provides the following inequality on $c_\varepsilon$ : 

\[\mathcal{C} \geqslant \overline{AT}_\varepsilon (u_\varepsilon , v_\varepsilon) \geqslant \int_0^L (\eta_\varepsilon + v_\varepsilon ^2) |u_\varepsilon '|^2 = \int_0^L |c_\varepsilon u_\varepsilon '| = |c_\varepsilon| |u_\varepsilon (L) - u_\varepsilon (0)| = |c_\varepsilon a_1|.\]
where we used that $u_\varepsilon '$ is of constant sign (see \eqref{eq : eq2}). It follows that $|c_\varepsilon| \leqslant C/|a_1|$. The sequence $(c_\varepsilon)$ is bounded, so, up to an extraction, we can suppose that it converges to a constant $c_1$. The following proposition establishes the connection between $c_1$ and the limit function of $(u_\varepsilon)$. 

\begin{proposition}\label{convergenceuv}
One has $v_\varepsilon \rightarrow 1$ in $L^2(0,L)$ and, up to an extraction, there exists $u_1\in BV(0,L)$ such that $u_\varepsilon \rightarrow u_1$ in $L^1(0,L)$. 
Moreover, $u_\varepsilon ' \rightarrow c_1$ a.e in $[0,L]$ with $|Du_1|(0,L) \leqslant |a_1|$. We also have that $c_1 \in [0, a_1/L]$ if $a_1 \geqslant 0$ and $c_1 \in [a_1/L, 0]$ if $a_1 \leqslant 0$.   
\end{proposition}

\begin{proof}
The convergence of $v_\varepsilon$ comes from the energy bound \eqref{energyboundtemps1} : 
\[\int_0^L (1-v_\varepsilon)^2 \leqslant \varepsilon \overline{AT}_\varepsilon (u_\varepsilon , v_\varepsilon ) \leqslant \mathcal{C}\varepsilon  \rightarrow 0.\]
As $u_\varepsilon$ is monotone, $u_\varepsilon (0) =0$ and $u_\varepsilon (L) =a_1$, it is of constant sign and $\int_0^L |u_\varepsilon| = |\int_0^L u_\varepsilon | \leqslant |a_1| L$. So $(u_\varepsilon)$ is bounded in $L^1(0,L)$. Secondly, as $u_\varepsilon '$ is of constant sign, 
\[|Du_\varepsilon | (0,L) = \int_0^L |u_\varepsilon '| = |\int_0^L u_\varepsilon '| = |a_1|.\]
So the total variation of $u_\varepsilon$ is also bounded. Then, the sequence $(u_\varepsilon )$ is bounded in $BV(0,L)$ and we can extract a converging subsequence in $L^1(0,L)$ to a function $u_1\in BV(0,L)$. The lower semi-continuity of the total variation implies 
\[|Du_1|(0,L) \leqslant \liminf\limits_{\varepsilon \rightarrow 0} |Du_\varepsilon | (0,L) \leqslant |a_1|.\]
Using equation \eqref{eq : eq2}, as one can extract a converging subsequence of $(v_\varepsilon)$ that converges to $1$ almost everywhere, one gets $u_\varepsilon ' \rightarrow c_1$ almost everywhere. If $u_\varepsilon \geqslant 0$, then by Fatou's lemma : 
\[c_1 L = \int_0^L \lim\limits_{\varepsilon \rightarrow 0}  u_\varepsilon ' \leqslant \liminf \limits_{\varepsilon \rightarrow 0} \int_0^L u_\varepsilon ' = a_1.\]
If $u_\varepsilon \leqslant 0$, then similarly 
\[-c_1L =  \int_0^L  \lim\limits_{\varepsilon \rightarrow 0}(-u_\varepsilon ') \leqslant \liminf \limits_{\varepsilon \rightarrow 0} \int_0^L (-u_\varepsilon ') = -a_1.\]
We obtain that if $a_1\geqslant 0$, then $c_1 \in [0, a_1/L]$ and if $a_1 \leqslant 0$, then $c_1 \in [a_1/L, 0]$.   
\end{proof}

\subsection{Estimates on $v_\varepsilon$}

In this section we establish the two main estimates on $v_\varepsilon$, namely : 

\begin{proposition}\label{ellip}
Let $x\in [0,L]$. Suppose that $v_\varepsilon \geqslant \alpha >0$ over $I_r(x)$ for all $\varepsilon$ small enough. Then there exists $\varepsilon_0 >0$ and $b,M>0$ depending on $\alpha$ such that on $I_{r/2}(x)$ 
\begin{equation}\label{1647}
|\mu_\varepsilon | \leqslant M , \;\;\;\;\;\; |u_\varepsilon '| \leqslant M , \;\;\;\;\;\; v_\varepsilon \geqslant 1-b\varepsilon, \;\;\; \forall \varepsilon < \varepsilon_0.
\end{equation}
\end{proposition}

\begin{corollary}\label{ellip2}
For any compact set $K \in [0, L] \setminus \{x_{min}\}$, there exists $\varepsilon_0 >0$ and $b,M>0$ depending on $K$ such that \eqref{1647} holds.
\end{corollary}

\begin{proposition}\label{O(1/eps)}
If $c_0 >0$, then, 

\[ \|v_\varepsilon ' \|_\infty \leqslant \frac{C}{\varepsilon}.\]

If $|c_1| > 0$, then, 

\[ \|u_\varepsilon ' \|_\infty \leqslant \frac{C}{\varepsilon}, \;\;\;\;\; \underset{(0,L)}{\inf} v_\varepsilon  \geqslant m\sqrt{\varepsilon}.\]
\end{proposition}

Note that the first estimate relies on a condition concerning $c_0$ and not $c_1$. The proof of \autoref{ellip} and its corollary is inspired by the proof of the Proposition 3.12 of \cite{L10} and the proof of Proposition 3.2 of \cite{HT00}.

\begin{proof}[Proof of \autoref{ellip}]
Suppose that the assertion is false for every $b$. We will reach a contradiction for $b$ large enough. To this end, fix $b>0$ and suppose on the contrary that, for infinitely many $\varepsilon$, $\inf_{I_{r/2}(x)} v_\varepsilon < 1-b\varepsilon$. 

First we can already bound $u_\varepsilon '$ and $\mu_\varepsilon$ uniformly in $\varepsilon$ on $I_r(x)$, namely : 
\begin{equation}
\label{inequ'}
|u'_\varepsilon| \leqslant \frac{| c_\varepsilon|}{\alpha^2 + \eta_\varepsilon} \leqslant M
\end{equation}
Concerning $\mu_\varepsilon$, since $v_{0,\varepsilon}= v_\varepsilon$ on $F_\varepsilon$, one has by \autoref{exprmu} on $I_r(x)$ : 
\begin{equation}
\label{ineqmu}
-\mu_\varepsilon = |\mu_\varepsilon| = \frac{v_{0,\varepsilon}}{(v_{0,\varepsilon}^2+\eta_\varepsilon)^2} | c_\varepsilon^2-c_{0,\varepsilon}^2| 1_{F_\varepsilon} \leqslant \frac{1}{(\alpha^2 +\eta_\varepsilon)^2} M 1_{F_\varepsilon} \leqslant M.
\end{equation}
Let $\zeta_\varepsilon$ be a smooth function such that $\zeta_\varepsilon(x-r)=\zeta_\varepsilon(x+r)=\alpha$, $\zeta_\varepsilon = 1-b/2 \varepsilon$ on $I_{r/2}(x)$, $\zeta_\varepsilon \leqslant 1$ on $I_r(x)$ and $ \|\zeta_\varepsilon'' \| = O(1)$.  

\smallskip

Set  $g_\varepsilon= v_\varepsilon - \zeta_\varepsilon$. Then $g_\varepsilon(x+r),g_\varepsilon (x-r) \geqslant 0$, while $\displaystyle \underset{I_{r/2}(x)}{\inf} g_\varepsilon \leqslant (1-b\varepsilon) -(1-(b\varepsilon)/2) = -(b\varepsilon )/2 <0$ since $\underset{I_{r/2}(x)}{\inf} v_\varepsilon \leqslant 1-b\varepsilon$. In particular, $g_\varepsilon$ reaches its minimum at a point $x_1$ of $(x-r, x+r)$ and this minimum satisfies $g_\varepsilon (x_1) \leqslant -(b\varepsilon )/2$. Then : 
\[\begin{array}{lll}
0 &\leqslant & \displaystyle\varepsilon g_\varepsilon''(x_1) \, =\, \varepsilon v_\varepsilon ''(x_1)  - \varepsilon \zeta_\varepsilon ''(x_1) \, =\, \displaystyle \frac{v_\varepsilon (x_1) -1}{\varepsilon} + v_\varepsilon (x_1) | u_\varepsilon '(x_1)| ^2 - \mu_\varepsilon (x_1) - \varepsilon \zeta_\varepsilon ''(x_1)\\
&&\\
&=& \displaystyle g_\varepsilon(x_1)\left( \frac{1}{\varepsilon} + |u_\varepsilon '(x_1)|^2 \right) + \zeta_\varepsilon (x_1) \left( \frac{1}{\varepsilon} + |u_\varepsilon '(x_1)|^2 \right) - \frac{1}{\varepsilon} - \varepsilon \zeta_\varepsilon ''(x_1) -\mu_\varepsilon (x_1)\\
\end{array}.\]
We deduce, using \eqref{inequ'} and \eqref{ineqmu}, 
\[-b/2 \geqslant -\frac{b\varepsilon}{2} \left( \frac{1}{\varepsilon} + |u_\varepsilon '(x_1)|^2 \right) \geqslant \frac{1-\zeta_\varepsilon(x_1)}{\varepsilon}+ \varepsilon \zeta_\varepsilon ''(x_1) - \zeta_\varepsilon (x_1) |u_\varepsilon '(x_1)|^2 + \mu_\varepsilon \geqslant - \varepsilon \|\zeta_\varepsilon'' \|_\infty - M^2-M \geqslant -2M^2\]
which is absurd for $b$ large enough.  
\end{proof}

\begin{proof}[Proof of \autoref{ellip2}]
Let $K$ be a compact set of $[0,L]\setminus \{x_{min}\}$ and take $\delta >0$ such that $(x_{min} - \delta, x_{min} + \delta ) \subset [0,L] \setminus K$. Then, for $\varepsilon >0$ small enough, $x_\varepsilon \in (x_{min} - \delta/2, x_{min} + \delta/2)$ so using \autoref{variationveps}, $v_\varepsilon$ is non-increasing on $(x_{min}-\delta, x_{min} -\delta/2)$. The $L^2$ convergence of $v_\varepsilon$ then implies  
\begin{equation}
\label{eq : eq12}
\frac{\delta}{2} (1-v_\varepsilon (x_{min} - \delta))^2 \leqslant \int_{x_{min} -\delta}^{x_{min} - \delta/2} (1-v_\varepsilon )^2 \rightarrow 0.
\end{equation}
It follows that for $\varepsilon$ small enough, $v_\varepsilon(x_{min} - \delta) \geqslant 1/2$ and then $v_\varepsilon \geqslant 1/2$ on $[0, x_{min} -\delta]$. Similarly, $v_\varepsilon \geqslant 1/2$ on $[x_{min} + \delta , L]$. Hence $v_\varepsilon \geqslant 1/2$ on $K$. 

\smallskip

Let $x\in K$ and $r_x$ such that $I_{r_x}(x) \subset (0, x_{min} -\delta ) \cup (x_{min} + \delta , L)$. Then $v_\varepsilon \geqslant 1/2$ on $I_{r_x}(x)$. According to \autoref{ellip}, there exists $\varepsilon_x >0$ and constants $b_x, M_x$ such that 
\[v_\varepsilon \geqslant 1-b_x\varepsilon \;\;\; , \;\;\; |u_\varepsilon '| \leqslant M_x \;\;\; , \;\;\; |\mu_\varepsilon | \leqslant M_x  \;\;\; \forall \varepsilon < \varepsilon_x.\]
As $K$ is compact, there exists a finite covering of $K$ of the form $K\subset \bigcup_{i=1}^{n} I_{r_{x_i}/2}(x_i)$. Taking $\varepsilon_0 = \underset{1\leqslant i\leqslant n}{\min} \varepsilon_{x_i}, b=\underset{1\leqslant i\leqslant n}{\max} b_{x_i}$ and $M=\underset{1\leqslant i\leqslant n}{\min} M_{x_i}$ yields \autoref{ellip2}. 
\end{proof}

Before proving \autoref{O(1/eps)}, we introduce the so-called discrepancy of $\overline{AT}_\varepsilon$. 

\begin{definition}
We denote by $d_\varepsilon$, also named as the \textit{discrepancy} of $\overline{AT}_\varepsilon$ the following quantity : 
\begin{equation}
\label{discrepancy}
d_\varepsilon := \dfrac{(1-v_\varepsilon)^2}{\varepsilon} - \varepsilon |v_\varepsilon '|^2 - (\eta_\varepsilon + v_\varepsilon^2)|u_\varepsilon '|^2.
\end{equation}
\end{definition}
Remark that 
\[\int_0^L |d_\varepsilon| = \int_0^L \left| \dfrac{(1-v_\varepsilon)^2}{\varepsilon} - \varepsilon (v_\varepsilon ')^2 - (\eta_\varepsilon + v_\varepsilon^2)|u_\varepsilon'|^2\right| \leqslant AT_\varepsilon (v_\varepsilon, u_\varepsilon) <C.\]
As in \cite{FLS09}, \cite{BMRb23} and \cite{L10}, we use the variations of the discrepancy to infer estimates on $u_\varepsilon$ and $v_\varepsilon$. The following proposition gives the expression of the derivative of the discrepancy : 

\begin{proposition}\label{dprime} There holds
\begin{equation}
\label{derivative discrepancy}
d_\varepsilon'= 2 v_{0,\varepsilon}' \mu_\varepsilon.
\end{equation}
Consequently, $d_\varepsilon$ is non-decreasing on $[0, L/2]$ and non-increasing on $[L/2, L]$. 
\end{proposition}
\begin{proof}
Note that $d_\varepsilon \in C^{0,1}([0,L])$, so that we can define its derivative almost everywhere but here in the sense of $\mathcal{D}'$, which reads as :

\[\begin{array}{lll}
\left\{\dfrac{(1-v_\varepsilon)^2}{\varepsilon} - \varepsilon (v_\varepsilon ')^2 - (\eta_\varepsilon + v_\varepsilon^2)|u_\varepsilon'|^2 \right\}'
&=&\displaystyle 2v_\varepsilon '[ \frac{v_\varepsilon - 1}{\varepsilon} - \varepsilon v_\varepsilon '' ] - 2 (v_\varepsilon^2 +\eta_\varepsilon) u_\varepsilon'' u_\varepsilon ' - 2v_\varepsilon v_\varepsilon ' |u_\varepsilon'|^2 \\ 
&&\\
&=&\displaystyle 2v_\varepsilon '[\mu_\varepsilon - v_\varepsilon |u_\varepsilon'|^2]  - 2 (v_\varepsilon^2 +\eta_\varepsilon) u_\varepsilon'' u_\varepsilon ' - 2v_\varepsilon v_\varepsilon ' |u_\varepsilon'|^2  \\
&&\\
&=& \displaystyle 2 v_\varepsilon ' \mu_\varepsilon - 2 u_\varepsilon '\underbrace{[(\eta_\varepsilon + v_\varepsilon^2) u_\varepsilon ']'}_{=0}
= \displaystyle 2 v_\varepsilon ' \mu_\varepsilon
= \displaystyle 2 v_{0,\varepsilon}' \mu_\varepsilon\\ 
\end{array}\]
where the last equality uses the fact that $v_\varepsilon = v_{0,\varepsilon}$ on the support of $\mu_\varepsilon$. Since $\mu_\varepsilon \leqslant 0$ and $v_{0,\varepsilon}'\leqslant 0$ on $[0,L/2]$ while $v_{0,\varepsilon}' \geqslant 0$ on $[L/2,L]$, we deduce that $d_\varepsilon$ is non-decreasing on $[0,L/2]$ and non-increasing on $[L/2,L]$. 
\end{proof}

We use this expression to prove the following estimate on the discrepancy. 
 
\begin{lemma}\label{estimatediscrepancy}
Let $d_\varepsilon$ be the discrepancy associated to $(u_\varepsilon, v_\varepsilon)$, defined by \eqref{discrepancy}. If $c_0 >0$, then  :

\[ \|d_\varepsilon \|_\infty \leqslant \frac{C}{\varepsilon}\] 
\end{lemma}

The hypothesis on $c_0$ comes from the necessity, in the proof, to bound from below $v_{0,\varepsilon}(L/2)$. As stated in Lemma 2.2 of \cite{BMRb23}, if $c_0 >0$, then $v_{0,\varepsilon} (L/2) \geqslant C\sqrt{\varepsilon}$.

\begin{proof}
The proof of \autoref{estimatediscrepancy} consists in estimating $d_\varepsilon$ at its local extremum $L/2$ and controlling the integral of the derivative $d_\varepsilon '$ to prove that the difference $|d_\varepsilon(x) - d_\varepsilon(L/2)|$ is bounded by $C/\varepsilon$. 

\smallskip

We start by estimating $d_\varepsilon (L/2)$. We distinguish two cases. 
Suppose first that $d_\varepsilon(L/2)\geqslant 0$. Then we have 

\[0\leqslant d_\varepsilon (L/2) \leqslant \frac{(1-v_\varepsilon(L/2))^2}{\varepsilon} \leqslant \frac{1}{\varepsilon}.\]
Suppose now that $d_\varepsilon(L/2) \leqslant 0$. Recall that $ \|d_\varepsilon \|_1$ is bounded by a constant $C$ independent from $\varepsilon$. Then, because $L/2$ is a global maximum of $d_\varepsilon$, $d_\varepsilon$ is non-positive and 

\[C \geqslant \int_0^{L} |d_\varepsilon| = -\int_0^L d_\varepsilon \geqslant L (-d_\varepsilon (L/2)) \]
so that $|d_\varepsilon (L/2)| \leqslant C$. Both cases imply that $|d_\varepsilon (L/2)|\leqslant C/\varepsilon$.  

\smallskip 
 
Now we prove that the total variation of $d_\varepsilon$ is bounded by $C/\varepsilon$. By \autoref{exprmu} :

\[\begin{array}{lll}
\displaystyle\int_0^L |d_\varepsilon '| &=& \displaystyle 2\int_0^L \left|v_{0,\varepsilon}' \mu_\varepsilon \right| = \displaystyle 2\int_0^L \left|v_{0,\varepsilon}' \frac{v_{0,\varepsilon}}{[(v_{0,\varepsilon})^2+\eta_\varepsilon]^2} (( c_\varepsilon)^2 - (c_{0,\varepsilon})^2) 1_{F_\varepsilon} \right|\\
&&\\
&=& \displaystyle 2 ((c_{0,\varepsilon})^2 - ( c_\varepsilon)^2)\int_0^L |v_{0,\varepsilon}'|\frac{v_{0,\varepsilon}}{[(v_{0,\varepsilon})^2+\eta_\varepsilon]^2} \leqslant  \displaystyle 2 C \int_0^L |v_{0,\varepsilon}'|\frac{v_{0,\varepsilon}}{[(v_{0,\varepsilon})^2+\eta_\varepsilon]^2}. \\
\end{array} \]
Using that $v_{0,\varepsilon}' \leqslant 0$ on $[0,L/2]$ and $v_{0,\varepsilon}' \geqslant 0$ on $[L/2,L]$, along with the estimate $v_{0,\varepsilon} (L/2) \geqslant C\sqrt{\varepsilon}$, we have : 
\[\int_0^{L/2} |v_{0,\varepsilon}'|\frac{v_{0,\varepsilon}}{[(v_{0,\varepsilon})^2+\eta_\varepsilon]^2} = - \int_0^{L/2} v_{0,\varepsilon}'\frac{v_{0,\varepsilon}}{[(v_{0,\varepsilon})^2+\eta_\varepsilon]^2}  = \frac{1}{2}\left[\frac{1}{(v_{0,\varepsilon})^2+\eta_\varepsilon}\right]_0^{L/2} = \frac{1}{2(v_{0,\varepsilon}^2 (L/2) + \eta_\varepsilon)} - \frac{1}{2(1+\eta_\varepsilon)} \leqslant \frac{C}{\varepsilon}.\]
Similarly, we obtain that $\displaystyle\int_{L/2}^L |v_{0,\varepsilon}'|\dfrac{v_{0,\varepsilon}}{[(v_{0,\varepsilon})^2+\eta_\varepsilon]^2}\leqslant \dfrac{C}{\varepsilon}$, so that $\displaystyle\int_0^L |d_\varepsilon'| \leqslant C/\varepsilon$. 

\smallskip

We deduce that for a.e $x$ in $[0,L]$, one has 

\[|d_\varepsilon(x)| \leqslant |d_\varepsilon(L/2)| + \left|\int_{L/2}^x d_\varepsilon' \right| \leqslant |d_\varepsilon(L/2)| + \int_0^L |d_\varepsilon'| \leqslant C/\varepsilon,\]
which gives the desired result. 
\end{proof}

We are now in position to complete the proof of \autoref{O(1/eps)} : 

\begin{proof}[Proof of \autoref{O(1/eps)}]
First we focus on the estimate on $v_\varepsilon$. For a.e $x$ in $[0,L]$, one has  

\[0\leqslant \varepsilon (v_\varepsilon ')^2 = \frac{(1-v_\varepsilon)^2}{\varepsilon} - d_\varepsilon - (\eta_\varepsilon + v_\varepsilon^2) |u_\varepsilon'|^2 \leqslant \frac{1}{\varepsilon} - d_\varepsilon \leqslant \frac{C}{\varepsilon},\]
hence the result. Now we focus on the estimate on $u_\varepsilon '$. One has 
\[| c_\varepsilon u_\varepsilon'| = (\eta_\varepsilon + v_\varepsilon^2) |u_\varepsilon '|^2 = \frac{(1-v_\varepsilon)^2}{\varepsilon} - \varepsilon |v_\varepsilon'|^2 - d_\varepsilon \leqslant \frac{C}{\varepsilon}.\]
Dividing by $|c_\varepsilon|$, we obtain the desired result. Concerning $v_\varepsilon$, we have 
\[\eta_\varepsilon + v_\varepsilon^2 = \frac{ c_\varepsilon}{u_\varepsilon '} \geqslant C \varepsilon .\]
Since $\eta_\varepsilon << \varepsilon$, we conclude that $v_\varepsilon \geqslant m\varepsilon$.
\end{proof}

In \cite{FLS09} and \cite{BMRb23}, the first estimate of \autoref{O(1/eps)} is true without any assumption on the slope $c_0$. This comes from the fact that the discrepancy $d_\varepsilon$ is constant in their case, so an estimate as in \autoref{estimatediscrepancy} follows. However, the hypothesis $c_0 >0$ is not that restrictive in our case, since $c_0=0$ implies that $c_1 =0$ as well, as long as $|F_\varepsilon| \neq 0$ (see \autoref{slopenull}). 

\subsection{Possible values for the limit slope}

The goal of this subsection is to establish the following property :

\begin{proposition}\label{slope}
$c_1 \in \{0, \frac{a_1}{L} \}$. 
\end{proposition}

\begin{corollary}\label{slopenull}
If $|a_1| >a_0$ and $|F_\varepsilon| \neq 0$ for infinitely many values of $\varepsilon$, then $c_1 = 0$. 
\end{corollary}

\begin{proof}[Proof of \autoref{slopenull}]
If $|F_\varepsilon| \neq 0$, then by \autoref{exprmu} $|c_{\varepsilon}|\leqslant |c_{0,\varepsilon}|$. Hence, by \cite[Lemma 2.1]{BMRb23} $|c_1| \leqslant c_0 \leqslant a_0/L < |a_1|/L$, so \autoref{slope} yields $c_1 =0$. 
\end{proof}

The proof of \autoref{slope} is similar to that of \cite[Lemma 4.4]{FLS09}.

\begin{proof}[Proof of \autoref{slope}]
We already know by \autoref{convergenceuv} that $c_1$ is within the interval $[0, a_1/L]$ (or $[a_1/L, 0]$). 

\medskip 

Suppose now that $|c_1| >0$. The proof is divided into two steps. In the first step, we integrate \eqref{eq : eq4} over sets of the form \[K_{\varepsilon, A} := \{v_\varepsilon \leqslant A \sqrt{\varepsilon}\}\]
to obtain an estimate on $|K_{\varepsilon, A}|$. In the second step, we split the integral $\int_0^L u_\varepsilon'$ into three terms which we study using the previous estimate and \autoref{O(1/eps)}.  

\medskip

\underline{Step 1 :} We estimate $|K_{\varepsilon, A}|$. Fix $A>0$. We integrate the inequality 

\[-\varepsilon v_\varepsilon '' + \frac{v_\varepsilon  c_\varepsilon^2}{(\eta_\varepsilon + v_\varepsilon ^2)^2} + \frac{v_\varepsilon -1}{\varepsilon} \leqslant 0\]
over the set $\{v_\varepsilon \leqslant A\sqrt{\varepsilon}\}$. We obtain 
\begin{equation}\label{intègreineq}
\int_{K_{\varepsilon, A}} \varepsilon v_\varepsilon'' \geqslant \int_{K_{\varepsilon, A}} \frac{v_\varepsilon  c_\varepsilon^2}{(\eta_\varepsilon +v_\varepsilon ^2)^2} + \int_{K_{\varepsilon, A}} \frac{v_\varepsilon -1}{\varepsilon}.
\end{equation} 
According to \autoref{variations of v}, the set $K_{A,\varepsilon}$ is an interval that we denote by $[a_\varepsilon , b_\varepsilon]$. 
Then we have, using \autoref{O(1/eps)}, 
\[|\int_{a_\varepsilon}^{b_\varepsilon} \varepsilon v_\varepsilon'' | = \varepsilon |v_\varepsilon'(b_\varepsilon) - v_\varepsilon '(a_\varepsilon)| \leqslant C.\]
Next, recall by \autoref{O(1/eps)} that $v_\varepsilon^2 \geqslant m\varepsilon >> \eta_\varepsilon$. We can then bound from below the right-hand side of \eqref{intègreineq} by 
\[\int_{K_{\varepsilon, A}} \frac{v_\varepsilon  c_\varepsilon^2}{(\eta_\varepsilon +v_\varepsilon ^2)^2} + \int_{K_{\varepsilon, A}} \frac{v_\varepsilon -1}{\varepsilon} \geqslant \int_{K_{\varepsilon, A}} \frac{ v_\varepsilon c_1^2/2}{2 v_\varepsilon^4} - \frac{|K_{\varepsilon, A}|}{\varepsilon} \geqslant C\frac{|K_{\varepsilon, A}|}{A^3 \varepsilon^{3/2}} - \frac{|K_{\varepsilon, A}|}{\varepsilon} \geqslant C \frac{|K_{\varepsilon, A}|}{A^3 \varepsilon^{3/2}}. \] 
Combining the two bounds, we obtain $|K_{\varepsilon, A}| \leqslant C A^3 \varepsilon^{3/2}$. This ends Step 1. 

\medskip

\underline{Step 2 :} In this step, we investigate the weight of $u_\varepsilon '$ in each of the three regimes $\{v_\varepsilon \leqslant A\sqrt{\varepsilon}\}, \{A\sqrt{\varepsilon} \leqslant v_\varepsilon \leqslant 1/2\}$ and $\{v_\varepsilon \geqslant 1/2\}$. Indeed, one has : 
\[a_1 = \int_0^L u_\varepsilon ' = \underbrace{\int_{K_{\varepsilon, A}} u_\varepsilon'}_{:=I_1} + \underbrace{\int_{\{A\sqrt{\varepsilon} \leqslant v_\varepsilon \leqslant 1/2\}} u_\varepsilon'}_{:=I_2} + \underbrace{\int_{ \{v_\varepsilon \geqslant 1/2\}} u_\varepsilon '}_{:=I_3} . \]
First, thanks to the previous estimate on $K_{\varepsilon,A}$ and \autoref{O(1/eps)}, 
\[I_1 \leqslant \int_{K_{\varepsilon,A}} |u_\varepsilon'| \leqslant \frac{C}{\varepsilon} |K_{\varepsilon, A}| \leqslant C A^3 \varepsilon^{1/2}. \]
As $u_\varepsilon' \rightarrow c_1$ and $v_\varepsilon \rightarrow 1$ a.e, we have $u_\varepsilon' 1_{v_\varepsilon \geqslant 1/2} \rightarrow c_1$ a.e. As
\[\left|\frac{ c_\varepsilon}{\eta_\varepsilon + v_\varepsilon^2} 1_{\{v_\varepsilon \geqslant 1/2\}}\right| \leqslant 2|c_1| \times 4, \]
we can apply the dominated convergence theorem to obtain 
\[I_3 \rightarrow \int_0^L c_1 = c_1 L.\]
For $I_2$, the energy bound \eqref{energybound} gives :
\[C \geqslant \int_0^L \frac{(1-v_\varepsilon)^2}{\varepsilon} \geqslant \int_{\{v_\varepsilon \leqslant 1/2\}} \frac{(1-v_\varepsilon)^2}{\varepsilon} \geqslant \frac{1}{4\varepsilon} |\{v_\varepsilon \leqslant 1/2\}|\] 
On $\{A\sqrt{\varepsilon}\leqslant v_\varepsilon \leqslant 1/2\}$, we recover 
\[ |u_\varepsilon ' |= \frac{| c_\varepsilon|}{\eta_\varepsilon +v_\varepsilon^2} \leqslant \frac{| c_\varepsilon|}{v_\varepsilon^2} \leqslant \frac{2|c_1|}{A^2\varepsilon}. \]
Thus 
\[|I_2| \leqslant |\{A\sqrt{\varepsilon} \leqslant v_\varepsilon \leqslant 1/2\}| \times \frac{2|c_1|}{A^2\varepsilon} \leqslant \frac{C}{A^2}\]
Combining the estimates on $I_1, I_2$ and $I_3$, we obtain 
\[|a_1| \leqslant \frac{C}{A^2} + |c_1|L \]
By letting $A$ tending to $\infty$, we obtain $|c_1| L \geqslant |a_1|$. Combining with $|c_1|\leqslant |a_1|/L$, we get the desired result. 
\end{proof}

We can now conclude the proof of $(i)$ in \autoref{Theorem 1}. 

\begin{proof}[Proof of $(i)$ in \autoref{Theorem 1}]
Let $\delta >0$. 
We have already seen with \autoref{ellip2} that $v_\varepsilon \rightarrow 1$ uniformly on $[0,L]\setminus I_{\delta}(x_{min})$ for every $\delta >0$. This combined with \eqref{eq : eq2} implies that $u_\varepsilon ' \rightarrow c_1$ uniformly on $[0,L]\setminus I_{\delta}(x_{min})$ for every $\delta >0$. Then we get that : 

\[u_\varepsilon (x) = \int_0^x u_\varepsilon '(t) dt \rightarrow c_1 x,\]
uniformly with respect to $x\in [0, x_{min} -\delta]$. Similarly, $u_\varepsilon (x) \rightarrow a_1 - c_1 (L-x)$ uniformly with respect to $x\in [x_{min} +\delta , L]$. Using that $c_1 = 0$ or $c_1 =a_1/L$, we see that the limit function $u_1$ of $(u_\varepsilon)$ (in the point-wise convergence sense) is either $u_1 = u_{aff}$ or $u_1 = u_{jump}$. Moreover, as $u_\varepsilon$ is monotone and $u_\varepsilon $ is uniformly bounded by $a_1$ on $[0,L]$, by dominated convergence, we get that $u_\varepsilon \rightarrow u_1$ in $L^2(0,L)$. The additionnal conclusion corresponds to \autoref{slopenull}.   
\end{proof}


\subsection{Convergence of the Dirichlet energy}

In this subsection, we prove the following, which proves part $(iii)$ of \autoref{Theorem 1}. The proof is similar to \cite[section 6]{FLS09} and \cite[section 3]{BMRb23}.

\begin{lemma}
$(\eta_\varepsilon + v_\varepsilon^2) |u_\varepsilon'|^2 dx\overset{\star}{\rightharpoonup} |u_1'|^2 dx$ weak$^\star$ in $\mathcal{M}([0,L])$.  
\end{lemma}

\begin{proof}
Let $\varphi$ be a test-function in $\mathcal{C}^\infty([0,L])$. According to \autoref{convergenceuv} :

\[\int_0^L (v_\varepsilon ^2 +\eta_\varepsilon )|u_\varepsilon'|^2 \varphi  = c_{\varepsilon} \int_0^L u_\varepsilon ' \varphi  = - c_\varepsilon \int_0^L u_\varepsilon \varphi' +  c_\varepsilon a_1 \varphi(L) \rightarrow -c_1 \int_0^L u_1 \varphi' + c_1 a_1 \varphi (L).\]
If $c_1=0$, then $u_1=u_{jump}$ and $u_1'=0$ so 

\[-c_1 \int_0^L u_1 \varphi' + c_1 a_1 \varphi (L) = 0 = \int_0^L |u_{jump}'|^2 \varphi .\]
If $c_1 = a_1/L$, then $u_1 = u_{aff}$ and $u_1' =a_1/L$, so by integration by part,

\[-c_1 \int_0^L u_1 \varphi' + c_1 a_1 \varphi (L) = c_1 \int_0^L u_{aff}' \varphi = \int_0^L |u_{aff} '|^2 \varphi.\]
\end{proof}

\subsection{Convergence of the phase-field term}

In this subsection, we prove the first part of $(ii)$ in \autoref{Theorem 1}, namely  

\begin{lemma}
$\left(\frac{(1-v_\varepsilon)^2}{\varepsilon} + \varepsilon |v_\varepsilon'|^2\right) dx \overset{\star}{\rightharpoonup} \alpha \delta_{x_{min}}$ weak$^\star$ in $\mathcal{M}([0,L])$ for some $\alpha\in \mathbb{R}$.  
\end{lemma}

The computation of the weight $\alpha$ is made in section \autoref{section 4.7}. 

\medskip

We start by noticing that the energy bound \eqref{energybound} implies that the measure $\left(\frac{(1-v_\varepsilon)^2}{\varepsilon} + \varepsilon |v_\varepsilon'|^2\right) dx$ is bounded, thus we can extract a converging subsequence such that

\[\left(\frac{(1-v_\varepsilon)^2}{\varepsilon} + \varepsilon |v_\varepsilon'|^2\right) dx \overset{\star}{\rightharpoonup} \nu \;\;\;\; \text{in } \mathcal{M}([0,L])\]
for some non negative $\nu \in \mathcal{M}([0,L])$. 
In \cite{BMRb23} and in \cite{FLS09}, the proof that $\nu$ is atomic relies on the following estimate (see \cite[Lemma 3.1]{BMRb23} and \cite[Lemma 6.1]{FLS09}): if $K$ is a compact subset of $[0,L] \setminus \{x_{min}\}$, then 

\[\int_K \left(\frac{(1-v_{0,\varepsilon})^2}{\varepsilon} + \varepsilon |v_{0,\varepsilon}'|^2\right) dx \leqslant C_K \varepsilon^{1/4}.\]
The proof of this estimate relies on an integration by part of the PDE satisfied by $v_{0,\varepsilon}$ on $K$, and strongly uses the estimate $ \|v_{0,\varepsilon} '  \|_\infty \leqslant C/\varepsilon$. In our case, such an estimate $ \|v_\varepsilon ' \|_\infty \leqslant C/\varepsilon$ is proved only in the case where $c_0 >0$.

\smallskip

We are going to show the following result, which provides a sharper estimate (an order $1$ on $\varepsilon$ instead of an order $1/4$) : 

\begin{proposition}\label{vitesse sur compact}
Let $K$ be a compact subset of $[0,L] \setminus \{x_{min}\}$. There exists a constant $C_{K}$ such that :

\[\int_{K} \left(\frac{(1-v_\varepsilon)^2}{\varepsilon}+\varepsilon |v_\varepsilon'|^2 \right) \leqslant C_K \varepsilon. \]
\end{proposition} 

The proof is divided in three steps. In the first one, we prove the result for a compact set $K$ of $(0, L) \setminus \{x_{min}\}$ : 

\begin{lemma}\label{vitesse sur compact sans bord}
Let $K$ be a compact subset of $(0,L)\, \setminus \{x_{min}\}$. 
For any positive test-function $\varphi$ supported in $K$, there exists a constant $C_{K, \varphi}$ such that :

\[\int_{K} \left(\frac{(1-v_\varepsilon)^2}{\varepsilon}+\varepsilon |v_\varepsilon'|^2 \right) \varphi \leqslant C_{K,\varphi} \varepsilon. \]
\end{lemma} 

In the second one, we use \autoref{vitesse sur compact sans bord} to obtain a sharper estimate on $ \|v_\varepsilon ' \|_\infty$ on $(0, L)\setminus \{x_{min}\}$. 

\begin{lemma}\label{grosse estimée sur vprime}
Let $K$ be a connected compact subset of $(0,L) \setminus \{x_{min}\}$. Then there exists a constant $C_K$ such that, 
\[\underset{K}{\sup} \;\varepsilon |v_\varepsilon '|^2 \leqslant C_K \varepsilon^{1/3}.\]
\end{lemma}

Finally, we use \autoref{grosse estimée sur vprime} to prove \autoref{vitesse sur compact}, in a similar way than \cite[Lemma 3.1]{BMRb23} and \cite[Lemma 6.1]{FLS09}. 

\begin{proof}[Proof of \autoref{vitesse sur compact sans bord}]
The argument is inspired from \cite[Proposition 4.1]{HT00} to prove a monotonicity formula for the Allen-Cahn functional.

Take $\varphi$ a smooth and positive test function supported on $K$. First, using \autoref{ellip2}, there exists $b, M>0$ such that 
\[ u'_\varepsilon \leqslant M, \;\;\;\; |\mu_\varepsilon| \leqslant M,  \;\;\;\; v_\varepsilon \geqslant 1-b\varepsilon \;\;\;\; \text{on } K.\]
Then
\begin{equation}\label{1707}
\int_{K} \frac{(1-v_\varepsilon)^2}{\varepsilon} \varphi \leqslant \int_{K} \frac{b^2 \varepsilon^2}{\varepsilon} \varphi \leqslant C_K\varepsilon.
\end{equation}
It remains to estimate $\int_{K} \varepsilon |v_\varepsilon'|^2 \varphi$. To do so, we compare $\varepsilon |v_\varepsilon'|^2$ with its second derivative (in the distributional sense). One has
\[[\varepsilon |v_\varepsilon '|^2]' = 2\varepsilon v_\varepsilon ''v_\varepsilon ' =  2v_\varepsilon '\left(\frac{v_\varepsilon -1}{\varepsilon} + v_\varepsilon |u_\varepsilon'|^2 - \mu_\varepsilon \right).\]
Note that $v_\varepsilon ' \mu_\varepsilon \in L^1(0,L)$, so it admits a derivative in the distributional sense and then
\[[\varepsilon |v_\varepsilon '|^2]'' = \underbrace{2 v_\varepsilon'' \left(\frac{v_\varepsilon -1}{\varepsilon} + v_\varepsilon |u_\varepsilon'|^2\right)}_{= A_\varepsilon} \underbrace{ - 2 (v_\varepsilon'\mu_\varepsilon)'}_{=B_\varepsilon} + 2 \frac{v_\varepsilon'^2}{\varepsilon} +  \underbrace{2 |v_\varepsilon'|^2 |u_\varepsilon'|^2}_{=C_\varepsilon} + \underbrace{4 v_\varepsilon'v_\varepsilon u_\varepsilon''u_\varepsilon'}_{=D_\varepsilon} .\]
It follows that 
\begin{equation}
\label{eq15}
2\int_K \varepsilon |v_\varepsilon'|^2 \varphi = \varepsilon ^2\int_K (\varepsilon (v_\varepsilon ')^2)'' \varphi - \varepsilon ^2 \left(\int_K A_\varepsilon \varphi + \left\langle \, B_\varepsilon \, ,\, \varphi\, \right \rangle + \int_K C_\varepsilon \varphi + \int_0^L D_\varepsilon \varphi\right).
\end{equation}
We now study each terms of the right hand side.  

\medskip

\underline{$\displaystyle\int_0^L (\varepsilon (v_\varepsilon ')^2)'' \varphi$ :} As $\varphi=0$ outside of $K$, then $\varphi$ and $\varphi'$ vanish at $0$ and $L$. Applying \eqref{1707} to $\varphi''$ yields  
\[\lim\limits_{\varepsilon\rightarrow 0} \int_0^L (\varepsilon |v_\varepsilon '|^2)'' \varphi = \lim\limits_{\varepsilon\rightarrow 0} \int_0^L \varepsilon |v_\varepsilon '|^2 \varphi'' =\lim\limits_{\varepsilon\rightarrow 0} \int_0^L \left(\varepsilon |v_\varepsilon '|^2 + \frac{(1-v_\varepsilon)^2}{\varepsilon} \right) \varphi'' \rightarrow \int_0^L \varphi'' d\nu.\]
Hence, 
\begin{equation}
\label{primeprime}
\displaystyle \varepsilon^2 \int_{K} (\varepsilon (v_\varepsilon ')^2)'' \varphi = \displaystyle \varepsilon^2 \int_0^L (\varepsilon (v_\varepsilon ')^2)''\varphi \leqslant C_{\varphi}\varepsilon ^2.
\end{equation} 

\underline{$\displaystyle\int_{K} A_\varepsilon \varphi$ :}

\[A_\varepsilon  = \frac{1}{\varepsilon} \underbrace{\left(\frac{v_\varepsilon -1}{\varepsilon} + v_\varepsilon |u_\varepsilon'|^2\right)^2}_{\geqslant 0} +\frac{1}{\varepsilon}\underbrace{(-\mu_\varepsilon)}_{\leqslant M} \left(\underbrace{\frac{v_\varepsilon -1}{\varepsilon}}_{\geqslant - b} +\underbrace{v_\varepsilon |u_\varepsilon '|^2}_{\geqslant 0}\right) \geqslant \frac{-bM}{\varepsilon} .\]

so that
\begin{equation}
\label{Aeps}
-\varepsilon^2 \int_{K} A_\varepsilon \varphi \leqslant C_{K,\varphi}\varepsilon.
\end{equation} 

\underline{$\displaystyle \langle \,B_\varepsilon  \, ,\, \varphi\, \rangle$ :} We use Cauchy-Schwarz inequality and the fact that $|\mu_\varepsilon | \leqslant M$ to get that
\begin{equation}\label{ineqBepsinter}
\left| 2\varepsilon ^2\langle \,(v_\varepsilon' \mu_\varepsilon)'\, ,\, \varphi\, \rangle\right| = 2\varepsilon ^2\left|\int_0^L v_\varepsilon ' \mu_\varepsilon \varphi'\right| \leqslant 2\varepsilon ^{3/2} \underbrace{ \|\mu_\varepsilon \|_{L^2(0,L)}}_{\leqslant C} \left|\left|\varepsilon^{1/2} |v_\varepsilon '| \varphi ' \right|\right|_{L^2(0,L)} .
\end{equation}
On the other hand, 
\[\left|\left|\varepsilon^{1/2} |v_\varepsilon '| \varphi ' \right|\right|^2_{L^2(0,L)} = \int_0^L \varepsilon |v_\varepsilon '|^2 |\varphi '| ^2 \leqslant \int_0^L \left(\frac{(1-v_\varepsilon)^2}{\varepsilon} + \varepsilon |v_\varepsilon '|^2\right) |\varphi'|^2 \rightarrow \int_0^L |\varphi '|^2 d\nu ,\]
From which we deduce that $ \|\varepsilon^{1/2} (v_\varepsilon ') \varphi '  \|^2_{L^2(0,L)} \leqslant C_{\varphi}$. \eqref{ineqBepsinter} then yields  
\begin{equation}
\label{Beps}
\left| \varepsilon^2\left\langle \, B_\varepsilon \, ,\, \varphi\, \right \rangle \right| \leqslant C_{\varphi} \varepsilon ^{3/2}.
\end{equation}

\underline{$\displaystyle \int_{K} C_\varepsilon \varphi$ :} $C_\varepsilon \geqslant 0$, so 
\begin{equation}
\label{Ceps}
-\varepsilon^2 \int_{K} C_\varepsilon \varphi \leqslant 0.
\end{equation}

\underline{$\displaystyle \int_{K} D_\varepsilon \varphi$ :} We have $\displaystyle u_\varepsilon '' = \left(\dfrac{c_{\varepsilon}}{v_\varepsilon^2+\eta_\varepsilon}\right)' = - 2c_{\varepsilon}\frac{ v_\varepsilon' v_\varepsilon}{(v_\varepsilon^2 +\eta_\varepsilon)^2}$, so that
\[D_\varepsilon = -8 v_\varepsilon' v_\varepsilon c_{\varepsilon}\frac{ v_\varepsilon' v_\varepsilon}{(v_\varepsilon^2 +\eta_\varepsilon)^2} \cdot \frac{c_{\varepsilon}}{v_\varepsilon^2 + \eta_\varepsilon} = -8 c_{\varepsilon}^2 |v_\varepsilon'|^2 \underbrace{\frac{v_\varepsilon^2}{(v_\varepsilon^2 + \eta_\varepsilon)^3}}_{\leqslant C}.\]
where we used that $1-b\varepsilon \leqslant v_\varepsilon \leqslant 1$ on $K$. As $(c_{\varepsilon})$ is bounded, we can write :
\[|D_\varepsilon| \leqslant M |v_\varepsilon'|^2 .\]
Thus
\begin{equation}
\label{Deps}
|\varepsilon^2 \int_{K} D_\varepsilon \varphi | \leqslant M\varepsilon \int_{K} \varepsilon |v_\varepsilon'|^2 \varphi.
\end{equation}
Injecting \eqref{Ceps} and \eqref{Deps} in \eqref{eq15} yields the following inequality, which is true for every $\varepsilon$ small enough and every positive smooth test function $\varphi$ :
\begin{equation}
\label{eq21}
(1-M\varepsilon) \int_K \varepsilon |v_\varepsilon '|^2 \varphi \leqslant \varepsilon^2 \int_K (\varepsilon |v_\varepsilon '|^2)'' \varphi - \varepsilon^2 \left(\int_K A_\varepsilon \varphi + \langle B_\varepsilon , \varphi\rangle \right) 
\end{equation} 
Now combining with \eqref{primeprime},\eqref{Aeps} and \eqref{Beps}, we obtain  
\begin{equation}
\label{terme d'énergie en O(eps)}
(1-M\varepsilon) \int_{K} \varepsilon |v_\varepsilon '|^2 \varphi \leqslant C\varepsilon^2+ C \varepsilon + C\varepsilon^{3/2} \leqslant C_{K,\varphi} \varepsilon
\end{equation} 
Which concludes the proof. 
\end{proof}

Note that a corollary of this property is that 
\begin{equation}
\int_K \varepsilon |v_\varepsilon '|^2 \leqslant C_K\varepsilon \;\;\;\;\; \text{ for any compact set } K \subset (0,L)\setminus \{x_{min}\}.
\end{equation}

\smallskip

We now move on to the proof of \autoref{grosse estimée sur vprime}. 

\begin{proof}[Proof of \autoref{grosse estimée sur vprime}]
Let $K$ be a connected compact subset of $(0,L) \setminus \{x_{min}\}$ and let $y \in K$.
The strategy is the following : we consider  a mollifier $\rho_{\varepsilon}$ (i.e $\rho$ is positive smooth function, compactly supported on $[-1,1]$ with integral $1$ and $\rho_{\varepsilon} = \dfrac{1}{\varepsilon^\alpha} \rho\left(\frac{\cdot}{\varepsilon^\alpha}\right)$, whith $\alpha>0$ to be determined). Then we take $\rho_\varepsilon (\cdot - y)$ as a test function in the inequality \eqref{eq21} and we prove the two estimates   
\begin{equation}
\label{énergie convolution}
\int_K \rho_\varepsilon(x-y) \varepsilon |v_\varepsilon '(x)|^2 dx \leqslant C_K \varepsilon^{1-\alpha}
\end{equation}
and
\begin{equation}
\label{différence convolution}
\left|\int_K \rho_\varepsilon (x-y) \varepsilon |v_\varepsilon'(x)|^2 dx -\varepsilon |v_\varepsilon'(y)|^2 \right| \leqslant C_K \varepsilon^{\alpha/2}.
\end{equation}
It follows that for $\alpha = 2/3$, one has $\varepsilon |v_\varepsilon'(y)|^2 \leqslant C_K \varepsilon^{1/3}$. 

\smallskip

We start by proving \eqref{énergie convolution}. One already has from inequalities \eqref{eq21}, \eqref{Aeps} and \eqref{ineqBepsinter} that  
\[(1-M\varepsilon) \int_K \rho_\varepsilon (x-y) \varepsilon |v_\varepsilon'|^2(x)dx \leqslant \underbrace{\varepsilon^2 \int_K (\varepsilon |v_\varepsilon '|^2)'' \rho_\varepsilon}_{:= I_1} + \underbrace{\varepsilon \int_K bM \rho_\varepsilon}_{:= I_2} -  \underbrace{\varepsilon^2\int_K (v_\varepsilon ' \mu_\varepsilon )' \rho_\varepsilon}_{:= I_3} .\]
\underline{$I_1$ :} Using \autoref{vitesse sur compact sans bord},
\begin{equation}
\left|\int_K \varepsilon |v_\varepsilon'|^2 \rho_\varepsilon'' \right|= \varepsilon^{-2\alpha} \left|\int_K \varepsilon |v_\varepsilon'|^2 \frac{1}{\varepsilon^\alpha} \rho''\left(\frac{x}{\varepsilon^\alpha}\right)\right| \leqslant \varepsilon^{-3\alpha}  \|\rho'' \|_\infty \int_K \varepsilon |v_\varepsilon'|^2 \leqslant C_K \varepsilon^{1-3\alpha}.
\end{equation}
Hence, 
\begin{equation}\label{I1}
I_1 \leqslant C_K \varepsilon^{3 -3\alpha}
\end{equation}  

\medskip

\underline{$I_2$ :}

\begin{equation}
\label{I2}
|I_2| = \left|\varepsilon bM \int_K \rho_\varepsilon\right| \leqslant bM |K| \varepsilon^{1-\alpha}  \|\rho \|_\infty \leqslant C_K \varepsilon^{1-\alpha}.
\end{equation}

\underline{$I_3$ :} Using the Cauchy-Schwarz inequality, \eqref{ineqmu} and \autoref{vitesse sur compact sans bord},
\begin{equation}
|I_3| \leqslant \left|\varepsilon^2 \int_K v_\varepsilon' \mu_\varepsilon \rho_\varepsilon '\right| \leqslant \varepsilon^2 \underbrace{\left(\int_K \mu_\varepsilon^2 \right)^{1/2}}_{\leqslant C_K} \varepsilon^{-1/2}\left(\int_K \varepsilon |v_\varepsilon'|^2 |\rho_\varepsilon '| ^2\right)^{1/2} \leqslant C_K \varepsilon^2 \varepsilon^{-2\alpha-1/2} \|\rho' \|_\infty \underbrace{\left(\int_K \varepsilon |v_\varepsilon'|^2 \right)^{1/2}}_{\leqslant C_K \varepsilon^{1/2}}.
\end{equation} 
And thus, 
\begin{equation}
\label{I3}
|I_3| \leqslant C_K \varepsilon^{2-2\alpha}.
\end{equation} 

\medskip

Combining \eqref{I1}, \eqref{I2} and \eqref{I3} yields \eqref{énergie convolution}. 

\bigskip

We now move on to the proof of \eqref{différence convolution}. First, using that $\int_K \rho_\varepsilon (x-y) =1$, we rewrite \eqref{différence convolution} as 

\[\displaystyle\int_K \rho_\varepsilon (x-y) (\varepsilon |v_\varepsilon'|^2(x)dx  -\varepsilon |v_\varepsilon'|^2(y))dx.\]
Then, for all $x$ in $K$, considering the fact that $|v_\varepsilon '|^2$ is $C^{0,1}([0,L])$ and using the Cauchy-Schwarz inequality,
\[|\varepsilon |v_\varepsilon'(x)|^2 -\varepsilon |v_\varepsilon'(y)|^2| = 2\varepsilon \left|\int_y^x v_\varepsilon''(t) v_\varepsilon '(t)dt\right| \leqslant 2 \varepsilon \varepsilon^{-1}\left|\int_y^x \varepsilon^2 |v_\varepsilon''(t)|^2 dt\right|^{1/2} \varepsilon^{-1/2}\underbrace{\left|\int_y^x \varepsilon |v_\varepsilon'(t)|^2 dt\right|^{1/2}}_{\leqslant C_K \varepsilon^{1/2}} .\]
By \autoref{ellip2}, 
\[\|\varepsilon v_\varepsilon ''\|_{L^\infty(K)} = \|\frac{v_\varepsilon -1}{\varepsilon} - \mu_\varepsilon +v_\varepsilon |u_\varepsilon'|^2 \|_{L^\infty (K)} \leqslant b + M +M^2 \leqslant C\]
so that 
\[\left|\int_y^x \varepsilon^2 v_\varepsilon''(t)^2 \right|^{1/2} \leqslant C|x-y|^{1/2}.\]
We deduce 
\[ \int_K \rho_\varepsilon (x-y) \left|\varepsilon |v_\varepsilon'(x)|^2 -\varepsilon |v_\varepsilon'(y)|^2\right|dx \leqslant C_K \int_K |x-y|^{1/2} \frac{1}{\varepsilon^\alpha} \rho_\varepsilon \left(\frac{x-y}{\varepsilon^\alpha}\right)dx.\]
Changing variable $s= \dfrac{x-y}{\varepsilon^\alpha}$, we rewrite the integral as 
\[\int_{(K-y)/\varepsilon^\alpha} \varepsilon^{\alpha/2} |s|^{1/2} \rho (s) ds \leqslant C \varepsilon^{\alpha/2}.\]
Finally, 
\[\int_K \rho_\varepsilon (x-y) \left|\varepsilon |v_\varepsilon'(x)|^2 -\varepsilon |v_\varepsilon'(y)|^2\right| dx \leqslant C_K\varepsilon^{\alpha/2}.\]
Combining \eqref{énergie convolution} and \eqref{différence convolution} finishes the proof. 
\end{proof}

We are now in position to prove \autoref{vitesse sur compact}. 

\begin{proof}[Proof of \autoref{vitesse sur compact}]
It is sufficient to prove the inequality for $K= [\, 0,x_{min} - \delta] \cup [x_{min} + \delta, L]$ for any $\delta >0$. In the sequel, we fix $\delta >0$.  
We adapt the argument used by \cite[Lemma 3.1]{BMRb23} and \cite[Lemma 6.3]{FLS09}. Multiplying both sides of the equation \eqref{eq : eq4} by $v_\varepsilon - 1$ and integrating by parts over $[0, x_{min} - \delta]$ yields : 
\[\int_0^{x_{min} -\delta} \left( \frac{(1-v_\varepsilon)^2}{\varepsilon} + \varepsilon |v_\varepsilon '|^2\right) = \varepsilon v_\varepsilon '(x_{min} - \delta) (v_\varepsilon (x_{min} - \delta) - 1) + \int_0^{x_{min}-\delta} \mu_\varepsilon (v_\varepsilon - 1) + \int_0^{x_{min} - \delta} c_\varepsilon^2 \frac{v_\varepsilon (1-v_\varepsilon) }{(\eta_\varepsilon + v_\varepsilon ^2)^2}.\]

We now bound each term on the right-hand side. Let $K' = K \cap [L/8, 7L/8]$ which is a compact subset of $[0,L]\setminus \{x_{min}\}$. According to \autoref{grosse estimée sur vprime}, $|v_\varepsilon '| \leqslant C_{K'}\varepsilon ^{-1/3}$ on $K'$. Using \autoref{ellip2}, there exists a constant $b>0$ such that $v_\varepsilon \geqslant 1-b\varepsilon$ on $K$. This implies that $\mu_\varepsilon$ is uniformly bounded on $K$ (see the inequality \eqref{ineqmu}). Consequently :   
\[\varepsilon\underbrace{| v_\varepsilon '(x_{min} - \delta)|}_{\leqslant C_{K'} \varepsilon^{-1/3}}\underbrace{|v_\varepsilon (x_{min} - \delta) - 1|}_{\leqslant b \varepsilon} \leqslant C_\delta \varepsilon^{5/3}.\]
We also have 
\[\left|\int_0^{x_{min}-\delta}\underbrace{\mu_\varepsilon}_{\leqslant M} \underbrace{(v_\varepsilon - 1)}_{\geqslant -b\varepsilon}\right| \leqslant C_\delta\varepsilon,\]
as well as
\[\left|\int_0^{x_{min} - \delta} c_\varepsilon^2 (1-v_\varepsilon) \frac{v_\varepsilon}{(\eta_\varepsilon + v_\varepsilon ^2)^2}\right| \leqslant c_\varepsilon ^2 \int_0^{x_{min} - \delta} \underbrace{|1-v_\varepsilon|}_{\leqslant b\varepsilon} \frac{1}{\underbrace{v_\varepsilon^3}_{\leqslant C}}\leqslant C_\delta \varepsilon.\]
We deduce that 
\[\int_0^{x_{min} -\delta} \left( \frac{(1-v_\varepsilon)^2}{\varepsilon} + \varepsilon |v_\varepsilon '|^2\right) \leqslant C_\delta \varepsilon.\]
We proceed similarly on $[x_{min} +\delta, L]$ to conclude. 
\end{proof}

Once we get \autoref{vitesse sur compact}, we can adapt the proof of \autoref{grosse estimée sur vprime} to obtain the following :

\begin{corollary}\label{grosse estimée sur v' avec bord}
Let $K$ be a connected compact set of $[\, 0,L\, ] \setminus \{x_{min}\}$. Then there exists a constant $C_K$ such that for all $y\in K$ : 

\[\underset{K}{\sup}\,  \varepsilon |v_\varepsilon '|^2 \leqslant C_K \varepsilon^{1/3}.\]
\end{corollary}

\begin{proof}
Fix $y\in K$. 
Take $\rho_{\varepsilon}$ a mollifier, with parameter $\varepsilon^\alpha$ as in the proof of \autoref{grosse estimée sur vprime} and $\alpha >0$ to be determined. The same argument used to establish \eqref{différence convolution} yields
\[\left|\int_K \rho_\varepsilon (x-y) \varepsilon |v_\varepsilon'(x)|^2 dx -\varepsilon |v_\varepsilon'(y)|^2 \right|\leqslant C_K \varepsilon^{\alpha/2}.\]
On the other hand, using \autoref{vitesse sur compact}, 
\[\begin{array}{lll}
\displaystyle\int_K \rho_\varepsilon(x-y) \varepsilon |v_\varepsilon '(x)|^2 dx &\leqslant &\displaystyle\int_K \rho_\varepsilon(x-y) \left(\frac{(1-v_\varepsilon)^2}{\varepsilon} + \varepsilon |v_\varepsilon '(x)|^2 \right)dx\\
&&\\
&\leqslant &\displaystyle\frac{\|\rho \|_\infty}{\varepsilon^\alpha}\int_K \left(\frac{(1-v_\varepsilon)^2}{\varepsilon} + \varepsilon |v_\varepsilon '(x)|^2 \right)dx\leqslant  C_K \varepsilon^{1-\alpha}.\\
\end{array}\]
Combining both estimates for $\alpha = 2/3$ yields the desired result. 
\end{proof}

The \autoref{vitesse sur compact} implies that $\nu$ is supported only on $\{x_{min}\}$ and thus $\nu = \alpha \delta_{x_{min}}$ for some $\alpha$. 
%
%
%

\subsection{Equipartition of the energy}

In this subsection, our goal is to prove the following result, which corresponds to the item $(iv)$ of \autoref{Theorem 1}, and which implies $(v)$. 

\begin{proposition}\label{equipartition}
We have 

\begin{equation}\label{integralequi}
\lim\limits_{\varepsilon \rightarrow 0} \int_0^L \left|\frac{(1-v_\varepsilon)^2}{\varepsilon} - \varepsilon (v_\varepsilon ')^2 \right| = 0.
\end{equation}
\end{proposition}

We recover here a classical result in the case of the minimizers of $AT_\varepsilon$, as it corresponds to the equality case of the arithmetico-geometric inequality used in the proof of the $\Gamma$-convergence of $AT_\varepsilon$ to $MS$. It is used in \cite{FLS09} and \cite{BMRb23} to compute the value of the weight $\alpha$ associated to the limit Dirac-measure of the phase field term of $AT_\varepsilon$. An equipartition of the energy result also exists for the Ambrosio-Tortorelli functional with an additional fidelity term, see \cite{L10}.

\smallskip

In \cite{FLS09} and \cite{L10}, the Neumann conditions on \eqref{eq : eq4} imply that the difference $\dfrac{(1-v_\varepsilon)^2}{\varepsilon} - \varepsilon |v_\varepsilon '|^2$ is positive. In \cite{BMRb23} and \cite{FLS09}, the proof of the equipartition of the energy hinges on the fact that $d_\varepsilon$ is constant and tends to $-c_0^2$. 

\medskip

In our case, recalling the equality :
\[\frac{(1-v_\varepsilon)^2}{\varepsilon} - \varepsilon |v_\varepsilon '|^2 = d_\varepsilon + c_\varepsilon u_\varepsilon ',\]
the first ingredient is to prove that $d_\varepsilon + c_\varepsilon u_\varepsilon ' \rightarrow 0$ at the boundary point $0$ and $L$ of $(0,L)$. It rests on \autoref{grosse estimée sur v' avec bord}. The second ingredient is to bound $d_\varepsilon '$ to control the variation of $v_\varepsilon$ on the full interval. To deal with the signs of each terms and get rid of the absolute values, we need to separate our integrals along $L/2$ and $x_\varepsilon$.

\begin{proof}[Proof of \autoref{equipartition}]

\autoref{grosse estimée sur v' avec bord} implies in particular that $\varepsilon |v_\varepsilon '|^2(0), \varepsilon |v_\varepsilon '|^2(L) \rightarrow 0$ as $\varepsilon \rightarrow 0$, by recalling that $x_{min} \in [L/4, 3L/4]$ (\autoref{corollaire position xmin}). Then we have 
\begin{equation}\label{discrepancy en 0}
d_\varepsilon (0) + c_\varepsilon u_\varepsilon '(0) = -\varepsilon |v_\varepsilon '|^2(0) \rightarrow 0
\end{equation}
and similarly, $d_\varepsilon (L) + c_\varepsilon u_\varepsilon '(L) \rightarrow 0$. 

By \autoref{dprime}, $d_\varepsilon \in C^{0,1}([0,L])$, so we can write $d_\varepsilon (x) = d_\varepsilon (0) + \int_0^x d_\varepsilon '$ on $[0,L/2]$ and $d_\varepsilon (x) = d_\varepsilon (L) + \int_L^x d_\varepsilon '$ on $[L/2, L]$, with $d_\varepsilon '\geqslant 0$ on $[0,L/2]$ and $d_\varepsilon ' \leqslant 0$ on $[L/2,L]$. However, the sign of $u_\varepsilon '(x) -u_\varepsilon '(0)$ depends on the position of $x$ with respect to $x_\varepsilon$. So we split the integral in the left-hand side of \eqref{integralequi} along $x_\varepsilon$ and later along $L/2$. 

\[\begin{array}{llll}
\displaystyle\int_0^L \left|\frac{(1-v_\varepsilon)^2}{\varepsilon} - \varepsilon |v_\varepsilon '|^2\right| &=&& \displaystyle \int_0^{x_\varepsilon} \left|d_\varepsilon (0) + c_\varepsilon u_\varepsilon '(0) + \int_0^x d_\varepsilon '(t)dt + c_\varepsilon (u_\varepsilon '(x) - u_\varepsilon '(0))\right|dx \\
&&\\
&&+&\displaystyle \int_{x_\varepsilon}^L \left|d_\varepsilon (L) +c_\varepsilon u_\varepsilon '(L) + \int_L^x d_\varepsilon '(t) dt + c_\varepsilon (u_\varepsilon '(x) - u_\varepsilon '(L) \right| dx \\
&&\\
&\leqslant&& \displaystyle \underbrace{x_\varepsilon |d_\varepsilon (0) + c_\varepsilon u_\varepsilon '(0)|}_{:=I} + \underbrace{\displaystyle\int_0^{x_\varepsilon} \left| \int_0^x d_\varepsilon '\right|}_{:=II} + \underbrace{\int_0^{x_\varepsilon} |c_\varepsilon (u_\varepsilon '(x) - u_\varepsilon '(0))| dx}_{:=III}\\
&&\\
&& \displaystyle + & \displaystyle \underbrace{(L-x_\varepsilon) |d_\varepsilon (L) + c_\varepsilon u_\varepsilon '(L)|}_{:=IV} + \displaystyle\underbrace{\int_{x_\varepsilon}^L \left| \int_x^L d_\varepsilon '\right|}_{:=V} + \underbrace{\int_{x_\varepsilon}^L |c_\varepsilon (u_\varepsilon '(x) - u_\varepsilon '(L))| dx}_{:=VI}\\
\end{array}\]

We deal separately with $I+IV$, $II+V$ and $III+ VI$. 

\medskip

\underline{$I+IV$ :} This tends to $0$ according to \eqref{discrepancy en 0}. 

\medskip

\underline{$II +V$ :} We prove that 

\[II+V \leqslant \int_0^{L/2} \left(\int_0^x d_\varepsilon '\right) dx - \int_{L/2}^L \left(\int_x^L d_\varepsilon '\right) dx + \left| \frac{L}{2} -x_\varepsilon \right| \left|\int_0^L d_\varepsilon '\right|\]

Suppose that $x_\varepsilon \geqslant L/2$ (the case $x_\varepsilon \leqslant L/2$ can be treated in a similar way). Recall that $d_\varepsilon ' \geqslant 0$ on $[0,L/2]$ and $d_\varepsilon ' \leqslant 0$ on $[L/2, L]$. Then $II+V$ can be rewritten as 

\[\begin{array}{lll}
\displaystyle\int_0^{x_\varepsilon} \left|\int_0^x d_\varepsilon '\right|dx + \int_{x_\varepsilon}^L \left|\int_x^Ld_\varepsilon '\right|dx &=& \displaystyle\int_0^{L/2} \left|\int_0^x \underbrace{d_\varepsilon '}_{\geqslant 0}\right|dx + \int_{L/2}^{x_\varepsilon} \left|\int_0^L d_\varepsilon' - \int_x^L \underbrace{d_\varepsilon '}_{\leqslant 0}\right|dx + \int_{x_\varepsilon}^L  \left|\int_x^L\underbrace{ d_\varepsilon ' }_{\leqslant 0} \right| dx \\
&&\\
&\leqslant& \displaystyle\int_0^{L/2} \left(\int_0^x d_\varepsilon '\right) dx - \int_{L/2}^{x_\varepsilon} \left(\int_x^L d_\varepsilon '\right) dx + \int_{L/2}^{x_\varepsilon} \left|\int_0^L d_\varepsilon'\right| -\int_{x_\varepsilon}^L \left(\int_x^L d_\varepsilon '\right) dx \\
&&\\
&=& \displaystyle \int_0^{L/2} \left(\int_0^x d_\varepsilon '\right) dx - \int_{L/2}^L \left(\int_x^L d_\varepsilon '\right) dx + \left| \frac{L}{2} -x_\varepsilon \right| \left|\int_0^L d_\varepsilon '\right|.\\
\end{array}\]

Now, as $u_\varepsilon' (0) =u_\varepsilon '(L) = \dfrac{c_\varepsilon}{\eta_\varepsilon + 1}$, by \eqref{discrepancy en 0} $d_\varepsilon (0)$ and $d_\varepsilon (L)$ both converge to $-c_1^2$. Then 
\[\left| \frac{L}{2} -x_\varepsilon \right| \left|\int_0^L d_\varepsilon '\right| = \left| \frac{L}{2} -x_\varepsilon \right| |d_\varepsilon (L) - d_\varepsilon (0)| \rightarrow 0.\] 
We study the remaining terms separately, recalling the expressions \eqref{derivative discrepancy} and \eqref{eq : mu_e}.  

\[\begin{array}{lllll}
\displaystyle \int_{0}^{L/2} \int_{0}^x d_\varepsilon ' dx &\leqslant & \displaystyle 2(c_{0,\varepsilon}^2 -c_\varepsilon^2)\int_{0}^{L/2} \int_{0}^x -\frac{v_{0,\varepsilon}' v_{0,\varepsilon}}{(v_{0,\varepsilon}^2 +\eta_\varepsilon)^2}  dx &=& \displaystyle (c_{0,\varepsilon}^2 -c_\varepsilon^2)\int_{0}^{L/2} \left[\frac{1}{v_{0,\varepsilon}^2 +\eta_\varepsilon}\right]_{0}^x dx \\
&&\\
&\leqslant& \displaystyle c_{0,\varepsilon}^2 \int_{0}^{L/2} \left(\frac{1}{v_{0,\varepsilon}^2(x) +\eta_\varepsilon} - \frac{1}{v_{0,\varepsilon}^2(0) +\eta_\varepsilon}\right) dx &=&\displaystyle \int_{0}^{L/2} \frac{c_{0,\varepsilon}^2}{v_{0,\varepsilon}^2(x)+ \eta_\varepsilon} dx -  \frac{c_{0,\varepsilon}^2}{v_{0,\varepsilon}^2(0) +\eta_\varepsilon} \cdot \frac{L}{2}\\
&&\\
&=& \displaystyle \int_{0}^{L/2} c_{0,\varepsilon} u_{0,\varepsilon} '(x) dx -  \frac{c_{0,\varepsilon}^2}{1+\eta_\varepsilon} \cdot \frac{L}{2} && \\   
\end{array}\]
Similarly, one has 
\[ - \int_{L/2}^L \left(\int_x^L d_\varepsilon '\right) dx \leqslant  \int_{L/2}^{L} c_{0,\varepsilon} u_{0,\varepsilon} '(x) dx - c_{0,\varepsilon}^2 \frac{1}{1+\eta_\varepsilon} \cdot \frac{L}{2}.\]
Summing both inequalities yields
\[\int_0^{L/2} \left(\int_0^x d_\varepsilon '\right) dx - \int_{L/2}^L \left(\int_x^L d_\varepsilon '\right) dx \leqslant  c_{0,\varepsilon} \underbrace{\int_0^L u_{0,\varepsilon} '(x) dx}_{u_{0,\varepsilon} (L) - u_{0,\varepsilon} (0)} - Lc_{0,\varepsilon}^2 \frac{1}{1+\eta_\varepsilon} = c_{0,\varepsilon} \left(a_0 - L \frac{c_{0,\varepsilon}}{1+\eta_\varepsilon}\right) \rightarrow c_0 (a_0 - Lc_0) =0,\]
where we used that $c_0\in  \{0, a_0/L\}$ (see \cite[Lemma 2.4]{BMRb23}).
We conclude that $II+V \rightarrow 0$. 

\medskip

\underline{$III+ VI$ :} Recall that $v_\varepsilon$ is non-increasing on $[0,x_\varepsilon]$ and non-decreasing on $[x_\varepsilon , L]$. Then, $c_\varepsilon u_\varepsilon '$ is itself non-decreasing on $[0,x_\varepsilon]$ and non-increasing on $[x_\varepsilon , L]$. This allows us to get rid of the absolute values in the expressions, so that $III+VI$ rewrites as  : 
\[\begin{array}{lll}
III+VI &=& \displaystyle \int_0^{x_\varepsilon} c_\varepsilon (u_\varepsilon '(x) - u_\varepsilon '(0)) dx + \int_{x_\varepsilon}^L c_\varepsilon (u_\varepsilon '(x) - u_\varepsilon '(L)) dx\\
&&\\
&=& \displaystyle \int_0^L c_\varepsilon u_\varepsilon '(x)dx - c_\varepsilon (x_\varepsilon u_\varepsilon '(0) + (L-x_\varepsilon) u_\varepsilon '(L)) \; =\;  c_\varepsilon (u_\varepsilon (L) - u_\varepsilon (0)) - L c_\varepsilon \frac{c_\varepsilon}{1+\eta_\varepsilon}\\
&&\\
&\rightarrow& c_1 (a_1 - c_1 L) =0\\
\end{array}\]
where we used \autoref{slope}. In conclusion, we get $III+VI \rightarrow 0$. 

\medskip

This completes the proof of the equipartition of the energy.
\end{proof}  

\subsection{Computation of the weight of the limit measure}\label{section 4.7}

In this section, we compute the weight $\alpha$ of the Dirac mass $\delta_{x_{min}}$, which completes the proof of the item $(ii)$ in \autoref{Theorem 1}. 

To do so, we use the equipartition principle, as it is done in \cite[section 6]{FLS09} and \cite[section 3]{BMRb23}. 

\begin{proof}[Proof of $(ii)$]
The equipartition principle ensures that : 
\[\left|\int_0^L \left( \frac{(1-v_\varepsilon)^2}{\varepsilon} + \varepsilon |v_\varepsilon '|^2\right) - 2|v_\varepsilon '| (1-v_\varepsilon) \right| = \int_0^L \left |\frac{1-v_\varepsilon}{\sqrt{\varepsilon}}-\sqrt{\varepsilon} v_\varepsilon ' \right|^2 \leqslant \int_0^L \left |\frac{(1-v_\varepsilon)^2}{\varepsilon}-\varepsilon |v_\varepsilon '|^2 \right| \rightarrow 0,\]
so that it is enough to evaluate $\displaystyle\int_0^L 2|v_\varepsilon '| (1-v_\varepsilon)$. Recall that $v_\varepsilon$ is non-increasing on $[0,x_\varepsilon]$ and non-decreasing on $[x_\varepsilon , L]$, where $x_\varepsilon$ denotes the global minimum of $v_\varepsilon$. We have : 
\[
\displaystyle\int_0^L 2|v_\varepsilon '| (1-v_\varepsilon) = \displaystyle  \int_0^{x_\varepsilon} -2v_\varepsilon ' (1-v_\varepsilon) +  \int_{x_\varepsilon}^L 2v_\varepsilon ' (1-v_\varepsilon) = \displaystyle \left[(1-v_\varepsilon)^2\right]_0^{x_\varepsilon} + \left[-(1-v_\varepsilon)^2\right]_{x_\varepsilon}^L = 2(1-v_\varepsilon(x_\varepsilon))^2.\]
So the value of $\alpha$ depends on the limit of $v_\varepsilon (x_\varepsilon)$. 

\medskip

First, note that $\alpha \geqslant 0$ as the measures $\left(\frac{(1-v_\varepsilon)^2}{\varepsilon} + \varepsilon |v_\varepsilon'|^2\right) dx$ are non negative. 
Suppose that $\alpha >0$. We prove that $v_\varepsilon(x_\varepsilon) \rightarrow 0$. Suppose that it is not the case and that there exists a subsequence such that $v_\varepsilon (x_\varepsilon) \geqslant \beta >0$. Then $v_\varepsilon \geqslant \beta$ on $[0,L]$. According to \autoref{ellip}, it implies that $v_\varepsilon (x_\varepsilon ) \geqslant 1-b\varepsilon$. But then 
\[\left[(1-v_\varepsilon)^2\right]_0^{x_\varepsilon} + \left[-(1-v_\varepsilon)^2\right]_{x_\varepsilon}^L \rightarrow 0\]
which contradicts that $\alpha >0$. 
Hence, $v_\varepsilon (x_\varepsilon) \rightarrow 0$ and 
\[\left[(1-v_\varepsilon)^2\right]_0^{x_\varepsilon} + \left[-(1-v_\varepsilon)^2\right]_{x_\varepsilon}^L \rightarrow 2\]
We conclude that if $\alpha \neq 0$, then $\alpha = 2$. 
\end{proof}

As a last remark, we explicit the link between the convergence of $\varepsilon |v_{0,\varepsilon} '|^2$ and the value of $x_{min}$. According to \cite[Theorem 1.1]{BMRb23}, there are to cases : 

\begin{itemize}
\item[$\bullet$] If $\varepsilon |v_{0,\varepsilon} '|^2 \overset{\star}{\rightharpoonup} \delta_{L/2}$, then $L/2$ is a $v-$jump for $v_{0,\varepsilon}$. The obstacle condition implies that $L/2$ is also a $v-$jump for $v_\varepsilon$, which means that $x_{min} =L/2$. Note that $\varepsilon |v_{0,\varepsilon} '|^2 \overset{\star}{\rightharpoonup} \delta_{L/2}$ can happen both when $u_0= u_{aff}$ and $u_0= u_{jump}$. 

\item[$\bullet$] If $\varepsilon |v_{0,\varepsilon} '|^2 \overset{\star}{\rightharpoonup} 0$, then $v_{0,\varepsilon}$ admits no $v-$jump and $x_{min}$ is not always equal to $L/2$ (but still belongs to $[L/4, 3L/4]$, see \autoref{corollaire position xmin}). Note that in this case, we always have $u_0= u_{aff}$. 
\end{itemize}


\section{Appendix : Proof of \autoref{Theorem 2}}

The proof of the $\Gamma-$convergence of $\overline{AT}_\varepsilon$ is an adaptation the proof of the classical $\Gamma$-convergence result of $AT_\varepsilon$ to $MS$ (see \cite{AT92}). The main difficulty is to take into account the obstacle $\Gamma_0$ in the proof of the $\Gamma$-limsup. The argument is inspired by \cite[Lemma 6.2]{G05}. 

\subsection{Proof of the $\Gamma$-liminf}

\begin{lemma}
Let $(u_\varepsilon, v_\varepsilon)$ be a sequence in $L^2(0,L)^2$ converging to $(u,v)$ in $L^2(0,L)^2$. Then 

\[\overline{MS}(u,v) \leqslant \liminf\limits_{\varepsilon \rightarrow +\infty} \overline{AT_\varepsilon}(u_\varepsilon, v_\varepsilon).\]
\end{lemma}

\begin{proof}
We suppose that $\liminf\limits_{\varepsilon \rightarrow \infty} \overline{AT}_\varepsilon(u_\varepsilon, v_\varepsilon) < \infty$. We extract a subsequence for which the energy converges to the liminf. From now on, we keep the notations $(u_\varepsilon, v_\varepsilon)$ for simplicity but we only work with this subsequence.

Since $u_\varepsilon \in H^1(0,L)\subset L^{\infty}(0,L)$, we can write $ \|u_\varepsilon \|_\infty \leqslant M$. We also have by a maximum principle (i.e replacing $v_\varepsilon$ by a better competitor) that $0\leqslant v_\varepsilon \leqslant 1$. Now, applying the arithmetico-geometric inequality to the second member of the energy yields 
\begin{equation}\label{IAG}
\mathcal{C} \geqslant \overline{AT}_\varepsilon (u_\varepsilon, v_\varepsilon) \geqslant \int_0^L (\eta_\varepsilon + v_\varepsilon ^2) |u_\varepsilon '|^2 + 2\int_0^L |v_\varepsilon ' |(1-v_\varepsilon).
\end{equation}
We first focus on the second term of the right-hand-side. Fix $0 < a \leqslant b < 1$. The coarea formula allows us to compute 
\begin{equation}\label{coarea}
\int_0^L |v_\varepsilon '|(1-v_\varepsilon) = \int_0^1 (1-t) \mathcal{H}^0(\partial ^\star \{v_\varepsilon < t\}) dt \geqslant \int_a^b (1-t) \mathcal{H}^0(\partial ^\star \{v_\varepsilon < t\}) dt .
\end{equation}
Set $f(t)= t-t^2/2$. Then there exists $t_\varepsilon \in (a,b)$ such that
\begin{equation}\label{moyenne}
\int_a^b f'(t) \mathcal{H}^0(\partial ^\star \{v_\varepsilon < t\}) dt \geqslant (f(b)-f(a)) \mathcal{H}^0(\partial ^\star \{v_\varepsilon < t_\varepsilon\}).
\end{equation}
Set $X_\varepsilon = \{v_\varepsilon < t_\varepsilon\}$ and $Y_\varepsilon = \{v_{0,\varepsilon} < t_\varepsilon\}$. Note that since $v_\varepsilon \leqslant v_{0,\varepsilon}$, one has $Y_\varepsilon \subset X_\varepsilon$. 
Combining \eqref{IAG}, \eqref{coarea} and \eqref{moyenne} yields 
\begin{equation}\label{border}
\mathcal{H}^0(\partial^\star X_\varepsilon) \leqslant C.
\end{equation}
Using the variations of $v_{0,\varepsilon}$ (see \cite[Proposition 2.1]{BMRb23}, $\mathcal{H}^0(\partial^\star Y_\varepsilon) \leqslant 2$ and $Y_\varepsilon$ is an open interval $(c^\varepsilon, d^\varepsilon )$. Moreover, since $t_\varepsilon < b$, one has 
\begin{equation}\label{measureto0}
|Y_\varepsilon| \subset |X_\varepsilon| \leqslant |\{1-v_\varepsilon > 1-b\}| \leqslant \frac{1}{(1-b)^2} \int_0^L (1-v_\varepsilon)^2 \leqslant \frac{C}{(1-b)^2} \varepsilon \rightarrow 0.
\end{equation}
Because $v_\varepsilon \in H^1(0,L)$, $v_\varepsilon$ is continuous and $X_\varepsilon$ is open set and can be written as the union of a countable number of open intervals and $\partial^\star X_\varepsilon = \partial X_\varepsilon$. By \eqref{border}, up to an extraction, there exists an integer $m$ such that for all $\varepsilon$, 
\[X_\varepsilon = \bigcup_{i=1}^m (a_i^\varepsilon, b_i^\varepsilon).\]
There exists an index $i_0$ such that $(c^\varepsilon, d^\varepsilon) \subset (a_{i_0}^\varepsilon, b_{i_0}^\varepsilon)$. Since $|X_\varepsilon|, |Y_\varepsilon| \rightarrow 0$, one has $b_i^\varepsilon - a_i^\varepsilon \rightarrow 0$ and since $[0,L]$ is compact, up to another extraction, the sequences $(a_i^\varepsilon), (b_i^\varepsilon)$ converge to the same limit $z_i$. In particular, $c^\varepsilon , d^\varepsilon \rightarrow z_{i_0}$. 

\smallskip

We will prove that $\Gamma_0 \cup J_u \subset \cup_i \{z_i\}$, with $\Gamma_0$ being empty or $\{L/2\}$. We first prove that $J_u\subset \cup_i \{z_i\}$, which is given if we prove that $u\in H^1(z_i, z_{i+1})$. Take an index $i\in \{1,\ldots , m\}$ and take $z_i < \alpha < \beta < z_{i+1}$. Let us show that $u_\varepsilon \rightharpoonup u$ in $H^1(\alpha, \beta)$. 

First, $ \|u_\varepsilon \|_2$ is bounded. Second, thanks to the energy bound \eqref{energybound}, and since for $\varepsilon$ large enough, one has $[\alpha, \beta] \subset [0,L]\setminus X_\varepsilon\subset [0,L]\setminus \{v_\varepsilon < b\}$ : 
\[C \geqslant \int_\alpha^\beta (\eta_\varepsilon + v_\varepsilon^2) |u_\varepsilon '|^2 \geqslant b \int_\alpha^\beta |u_\varepsilon'|^2 = b \|u'_\varepsilon \|_2^2.\]
So $(u_\varepsilon)$ is bounded in $H^1(0,L)$. Thus, up to an extraction, $u_\varepsilon \rightharpoonup \hat{u}$ in $H^1(\alpha, \beta)$. But according to Rellich embedding, up to a subsequence, $u_\varepsilon \rightarrow \hat{u}$ in $L^2(\alpha, \beta)$. Thus $\hat{u}=u$ and
\begin{equation}
\int_\alpha^\beta |u'|^2 \leqslant \liminf\limits_{\varepsilon \rightarrow 0} \int_\alpha^\beta |u_\varepsilon'|^2 \leqslant \liminf\limits_{\varepsilon \rightarrow 0} \int_{b_i^\varepsilon}^{a_{i+1}^\varepsilon} |u_\varepsilon'|^2.
\end{equation}
Letting $\alpha$ and $\beta$ tend to $z_i$ and $z_{i+1}$ respectively shows that $u\in H^1(z_i, z_{i+1})$ with 
\begin{equation}\label{dirichleti}
\int_{z_i}^{z_{i+1}} |u'|^2 \leqslant \liminf\limits_{\varepsilon \rightarrow 0} \int_{b_i^\varepsilon}^{a_{i+1}^\varepsilon} |u_\varepsilon'|^2.
\end{equation}
Similarly, we obtain
\begin{equation}\label{dirichletbord}
\int_{0}^{z_1} |u'|^2 \leqslant \liminf\limits_{\varepsilon \rightarrow 0} \int_{0}^{a_1^\varepsilon} |u_\varepsilon'|^2 \;\;\;\;\; \text{ and } \;\;\;\;\; \int_{z_m}^{L} |u'|^2 \leqslant \liminf\limits_{\varepsilon \rightarrow 0} \int_{b_m^\varepsilon}^{L} |u_\varepsilon'|^2.
\end{equation}
Combining \eqref{dirichletbord} with \eqref{dirichleti} for $i=1, \ldots , m$, we deduce that $u\in SBV^2(0,L)$ with $J_u\subset \cup_{i=1}^m \{z_i\}$ and with 
\begin{equation}\label{u'}
\int_0^L |u'|^2 \leqslant \liminf\limits_{\varepsilon\rightarrow 0} \int_{[0,L] \setminus X_\varepsilon} |u_\varepsilon '|^2.
\end{equation}
The same argument applied to the sequences $(u_{0,\varepsilon}, v_{0,\varepsilon})$ proves that $J_{u_0} \subset \{z_{i_0}\}$. Recalling \eqref{crack}, this proves that $\Gamma_0 \subset \cup_i \{z_i\}$. We then have 
\[2 \mathcal{H}^0(J_u \cup \Gamma_0) \leqslant 2m = \mathcal{H}^0(\partial^\star \{v_\varepsilon < t_\varepsilon\})\]
Combining this inequality with \eqref{coarea} and \eqref{moyenne} yields  
\begin{equation}\label{HJgamma}
\liminf\limits_{\varepsilon \rightarrow 0} 2\int_0^L |v_\varepsilon |(1-v_\varepsilon) \geqslant 2(f(b)- f(a)) \lim\limits_{\varepsilon \rightarrow 0} \mathcal{H}^0(\partial^\star X_\varepsilon) = 4(f(b)-f(a)) \mathcal{H}^0(J_u \cup \Gamma_0).
\end{equation}
Combining \eqref{HJgamma} with \eqref{IAG} and letting $a$ and $b$ tending to $0$ and $1$ respectivelygives 
\begin{equation}\label{ModicavsMS}
\liminf\limits_{\varepsilon \rightarrow 0}\int_0^L\left(\frac{(1-v_\varepsilon)^2}{\varepsilon} + \varepsilon |v_\varepsilon'|^2\right) \geqslant 2 \mathcal{H}^0(J_u \cup \Gamma_0).
\end{equation}
Now focus on the first term of $\overline{AT}_\varepsilon$. Indeed,
\[\int_0^L (\eta_\varepsilon + v_\varepsilon^2) |u_\varepsilon'|^2 \geqslant \int_{[0,L]\setminus X_\varepsilon} v_\varepsilon^2 |u_\varepsilon'|^2 \geqslant a^2 \int_{[0,L]\setminus X_\varepsilon} |u_\varepsilon'|^2.\]
and using \eqref{u'}, 
\[\liminf\limits_{\varepsilon \rightarrow 0} \int_0^L (\eta_\varepsilon + v_\varepsilon^2) |u_\varepsilon'|^2 \geqslant a^2\liminf\limits_{\varepsilon \rightarrow 0} \int_{[0,L]\setminus X_\varepsilon} |u_\varepsilon'|^2 \geqslant a^2 \int_0^L |u'|^2.\]
Since this inequality is true for every $a<1$, letting $a\rightarrow 1$ leads to
\begin{equation}\label{liminfdirichlet}
\liminf\limits_{\varepsilon \rightarrow 0} \int_0^L (\eta_\varepsilon + v_\varepsilon^2) |u_\varepsilon'|^2 \geqslant \int_0^L |u'|^2.
\end{equation}
Summing \eqref{ModicavsMS} and \eqref{liminfdirichlet}, gives $\liminf\limits_{\varepsilon \rightarrow 0} \overline{AT_\varepsilon} (u_\varepsilon, v_\varepsilon)\geqslant \overline{MS}(u).$
\end{proof}

\subsection{Proof of the $\Gamma$-limsup}

\begin{lemma}\label{gamma-limsup}
Let $u \in L^2((0,L))$. Then there exists a sequence $(u_\varepsilon, v_\varepsilon)$ converging in $L^2 \times L^2$ to $(u,1)$ such that 

\[ \limsup\limits_{\varepsilon \rightarrow 0} \overline{AT_\varepsilon} (u_\varepsilon, v_\varepsilon) \leqslant \overline{MS}(u).\]
\end{lemma}

In the sequel, we denote by $MM_\varepsilon^V (v) := \displaystyle \int_V \frac{(1-v)^2}{\varepsilon} + \varepsilon |v'|^2$ and $MM_\varepsilon (v) := MM_\varepsilon ^{[0,L]} (v)$. 

\medskip

We suppose that $u$ is in $SBV^2(0,L)$, so that $\overline{MS}(u) < \infty$. The proof goes the following way : we approach $J_u\setminus \Gamma_0$ with the classical recovery sequence $(w_\varepsilon, z_\varepsilon)$ used the proof of the $\Gamma-$limsup of $AT_\varepsilon$ to $MS$ (\autoref{recoverysequence}). This sequence might not satisfy the obstacle condition $w_\varepsilon \leqslant v_{0,\varepsilon}$ so we apply the modification $w_{1,\varepsilon} = \min (w_\varepsilon, v_{0,\varepsilon})$ (\autoref{w1eps}). In a neighbourhood $W$ of $\Gamma_0$, we might not have $\limsup\limits_{\varepsilon \rightarrow 0} MM_\varepsilon^V (w_{1,\varepsilon}) \leqslant 2\mathcal{H}^0(\Gamma_0)$. For this reason we construct a sequence $w_{2,\varepsilon}$ and a cutoff function $\varphi_\varepsilon$ that both vanish around $\Gamma_0$(\autoref{w2eps}). Finally, setting $v_\varepsilon = \min(w_{1,\varepsilon}, w_{2,\varepsilon})$ and $u_\varepsilon= \varphi_\varepsilon z_\varepsilon$, we prove that $(u_\varepsilon, v_\varepsilon)$ satisfies the requirements of \autoref{gamma-limsup}. 

\begin{proof}
Let us start by recalling the construction of a recovery sequence $(w_\varepsilon, z_\varepsilon)$ for $J_u\setminus \Gamma_0$. This construction can be found in \cite{AT92}.

\begin{proposition}\label{recoverysequence}
Let $d(x)$ denote the distance from $x$ to the set $J_u\setminus \Gamma_0$ and $\alpha_\varepsilon = \sqrt{\eta_\varepsilon \varepsilon}$. We set  
\[h_\varepsilon(t) = 1-e^{-\frac{t-\alpha_\varepsilon}{\varepsilon}},\]
Then the pair $(w_\varepsilon , z_\varepsilon)$  given by
\begin{equation}\label{weps}
w_\varepsilon = \left\{ \begin{array}{ll}
h_\varepsilon(d(x)) & \text{if } d(x) > \alpha_\varepsilon\\
0 & \text{if } 0\leqslant d(x) \leqslant \alpha_\varepsilon\\
\end{array}\right.
\end{equation}
and 
\begin{equation}\label{zeps}
z_\varepsilon = \left\{ \begin{array}{ll}
u(x) & \text{if } d(x) > \alpha_\varepsilon\\
\left(\frac{2d(x)}{\alpha_\varepsilon} -1\right) u(x) & \text{if } \frac{\alpha_\varepsilon}{2} < d(x) \leqslant \alpha_\varepsilon \\
0 & \text{if } d(x) \leqslant \frac{\alpha_\varepsilon}{2} \\ 
\end{array}\right.
\end{equation}
satisfies 
\begin{equation}\label{wzright1}
(w_\varepsilon, z_\varepsilon) \overset{L^2(0,L)^2}{\longrightarrow} (1,u)
\end{equation}
and, for all neighbourhood $W$ of $J_u\setminus \Gamma_0$
\begin{equation}\label{modicarecovery}
\limsup\limits_{\varepsilon \rightarrow 0} MM_\varepsilon^{W} (w_\varepsilon) \leqslant 2 \mathcal{H}^0 (J_u\setminus \Gamma_0)
\end{equation}
and, 
\begin{equation}\label{MSrecovery}
\limsup\limits_{\varepsilon \rightarrow 0} \int_{W} (\eta_\varepsilon + w_\varepsilon ^2) |z_\varepsilon '|^2 \leqslant \int_{W} |u'|^2.
\end{equation}
\end{proposition}

Now we adapt this recovery sequence to satisfy the obstacle condition : 

\begin{lemma}\label{w1eps}
Let 
\begin{equation}\label{w1epsilon}
w_{1,\varepsilon} = \min (w_\varepsilon , v_{0,\varepsilon})
\end{equation}
Then $w_{1,\varepsilon} \leqslant v_{0,\varepsilon}$, 
\begin{equation}\label{w1right1}
w_{1,\varepsilon}\overset{L^2(0,L)}{\longrightarrow} 1,
\end{equation}
and for any neighbourhood $W_1$ of $J_u \setminus \Gamma_0$ not containing $\Gamma_0$, 
\begin{equation}\label{modicaw1}
\limsup\limits_{\varepsilon \rightarrow 0} MM_\varepsilon^{W_1} (w_{1,\varepsilon}) \leqslant 2 \mathcal{H}^0 (J_u\setminus \Gamma_0)
\end{equation}
and, 
\begin{equation}\label{MSw1}
\limsup\limits_{\varepsilon \rightarrow 0} \int_{W_1} (\eta_\varepsilon + w_{1,\varepsilon}^2) |z_\varepsilon '|^2 \leqslant \int_{W_1} |u'|^2.
\end{equation}
\end{lemma}

\begin{proof}[Proof of \autoref{w1eps}]
First, by definition of \eqref{w1epsilon}, $w_{1,\varepsilon} \leqslant v_{0,\varepsilon}$.

\smallskip

Second, according to \eqref{wzright1}, $w_\varepsilon \rightarrow 1$ in $L^2(0,L)$ and according to \cite[Theorem 1.1 (i)]{BMRb23}, $v_{0,\varepsilon} \rightarrow 1$ in $L^2(0,L)$, so that $w_{1,\varepsilon} \rightarrow 1$ in $L^2(0,L)$ as well. 

\smallskip

Third, take $W_1$ a neighbourhood of $J_u\setminus \Gamma_0$ not containing $\Gamma_0$. Using \cite[Theorem 1.1 (iv)]{BMRb23}, $\limsup\limits_{\varepsilon \rightarrow 0} MM_\varepsilon^{W_1} (v_{0,\varepsilon}) \rightarrow 0$, so using \eqref{modicarecovery} and \eqref{MSrecovery} :

\[\limsup\limits_{\varepsilon \rightarrow 0} MM_\varepsilon^{W_1} (w_{1,\varepsilon}) \leqslant \limsup\limits_{\varepsilon \rightarrow 0} \left( \underbrace{MM_\varepsilon^{W_1} (v_{0,\varepsilon})}_{\rightarrow 0} + MM_\varepsilon ^{W_1} (w_\varepsilon) \right)\leqslant  2 \mathcal{H}^0 (J_u \setminus \Gamma_0)\]
and 
\[\limsup\limits_{\varepsilon \rightarrow 0} \int_{W_1} (\eta_\varepsilon + w_{1,\varepsilon}^2) |z_\varepsilon '|^2 \leqslant \limsup\limits_{\varepsilon \rightarrow 0}\int_{W_1} (\eta_\varepsilon + w_\varepsilon^2) |z_\varepsilon '|^2 \leqslant \int_{W_1} |u'|^2. \]
\end{proof}

Now we deal with $\Gamma_0$. Let $\mathcal{C} >0$ be such that 

\begin{equation}\label{energyboundv0}
\limsup MM_\varepsilon (v_{0,\varepsilon}) \leqslant \mathcal{C}.
\end{equation}

We construct a function $w_{2,\varepsilon}$ and a cut-off function $\varphi_\varepsilon$ that vanish around $\Gamma_0$. The construction is given in the following lemma (which is the analogue of \cite[Lemma 6.2]{G05}) :

\begin{lemma}\label{w2eps}
Set $r_\varepsilon = \sqrt[4]{\eta_\varepsilon/\varepsilon}, s_\varepsilon = 2 r_\varepsilon, t_\varepsilon = \sqrt{s_\varepsilon}$ and $k_\varepsilon= t_\varepsilon/s_\varepsilon$. 

There exists a cut-off function $\varphi_\varepsilon$ such that

\begin{equation}\label{phieps}
\varphi_\varepsilon =  \left\{ \begin{array}{lll}
1 & \text{if } & s_\varepsilon\leqslant v_{0,\varepsilon} \\
0 & \text{if } & v_{0,\varepsilon} \leqslant r_\varepsilon\\
\end{array}\right.
\end{equation}
and 
\begin{equation}\label{cvvarphi}
\int_{\{r_\varepsilon\leqslant v_{0,\varepsilon} \leqslant s_\varepsilon\}} \eta_\varepsilon |\varphi_\varepsilon '| ^2 \rightarrow 0
\end{equation}
Second, set 
\begin{equation}\label{w2epsilon}
w_{2,\varepsilon} = \left\{ \begin{array}{lll}
v_{0,\varepsilon} & \text{if } & t_\varepsilon\leqslant v_{0,\varepsilon} \\
\frac{k_\varepsilon}{k_\varepsilon-1}(v_{0,\varepsilon} - t_\varepsilon) + t_\varepsilon & \text{if } & s_\varepsilon\leqslant v_{0,\varepsilon} \leqslant t_\varepsilon\\
0 & \text{if } & v_{0,\varepsilon} \leqslant s_\varepsilon\\
\end{array}\right.
\end{equation}
Then $w_{2,\varepsilon} \leqslant v_{0,\varepsilon}$
\begin{equation}\label{w2right1}
w_{2,\varepsilon}\overset{L^2(0,L)}{\rightarrow} 1,
\end{equation}
and, for some constant $\mathcal{C} >0$,  
\begin{equation}\label{modicaw2eps}
\limsup\limits_{\varepsilon\rightarrow 0} MM_\varepsilon (w_{2,\varepsilon}) -MM_\varepsilon (v_{0,\varepsilon}) \leqslant \mathcal{C}\left(\frac{2 k_\varepsilon}{(k_\varepsilon-1)^2} + \frac{2t_\varepsilon}{(k_\varepsilon-1)(1-t_\varepsilon)^2} + \frac{2t_\varepsilon}{(1-t_\varepsilon)^2}\right).
\end{equation} 
\end{lemma}

\begin{remark}\label{quantity0}
Observe that the right hand side of \eqref{modicaw2eps} tends to $0$ with $\varepsilon$ as $t_\varepsilon \rightarrow 0$ and $k_\varepsilon \rightarrow +\infty$.
\end{remark} 

\begin{remark} 
The transition from the two states of $\varphi_\varepsilon$ is done on a domain on which $w_{2,\varepsilon}$ is equal to zero. 
\end{remark}

\begin{remark} 
The set $\{v_{0,\varepsilon} < t_\varepsilon\}$ might be empty. In this case, $w_{2,\varepsilon} = v_{0,\varepsilon}$ and $\varphi_\varepsilon =1$. 
\end{remark} 

\medskip

We postpone the proof of \autoref{w2eps} to the end of the proof of \autoref{gamma-limsup} and now we construct the sequence $(u_\varepsilon, v_\varepsilon)$ satisfying the requirements of \autoref{gamma-limsup}, using \autoref{w1eps} and \autoref{w2eps}. 

\medskip

Set 
\begin{equation}\label{veps}
v_\varepsilon = \min (w_{1,\varepsilon}, w_{2,\varepsilon}),
\end{equation}
and 
\begin{equation}\label{ueps}
u_\varepsilon = \varphi_\epsilon z_\varepsilon.
\end{equation}

\smallskip

\underline{$v_\varepsilon \leqslant v_{0,\varepsilon}$ :} This is a consequence of \autoref{w1eps} and \autoref{w2eps}. 

\underline{$(u_\varepsilon , v_\varepsilon) \rightarrow (u,1)$ in $L^2(0,L)^2$ :} The fact that $v_\varepsilon \rightarrow 1$ in $L^2(0,L)$ is a consequence of \eqref{w1right1} and \eqref{w2right1}. Concerning the convergence of $u_\varepsilon$, one has 

\[ \|u_\varepsilon - u \|_{L^2} \leqslant  \|u_\varepsilon - z_\varepsilon \|_{L^2} +  \|z_\varepsilon - u \|_{L^2}.\]
One already has from \eqref{wzright1} that $ \|z_\varepsilon - u \|_{L^2} \rightarrow 0$. On the other hand, as $u$ belongs to $SBV^2((0,L))$, it is a bounded function and using the definition \eqref{zeps} of $z_\varepsilon$ and the definition \eqref{phieps} of $\varphi_\varepsilon$,  

\[\int_0^L |\varphi_\varepsilon z_\varepsilon - z_\varepsilon |^2 = \int_0^L |z_\varepsilon|^2 |\varphi_\varepsilon - 1|^2 \leqslant \int_0^L |u|^2 |\varphi_\varepsilon - 1|^2 \leqslant \int_{\{v_{0,\varepsilon} \leqslant s_\varepsilon\}} |u|^2\leqslant  \|u \|_\infty^2 |\{v_{0,\varepsilon} \leqslant s_\varepsilon\}|\]
But the energy bound \eqref{energyboundv0} yields 
\begin{equation}\label{measure set v0}
\mathcal{C} \geqslant \int_{\{v_{0,\varepsilon} \leqslant s_\varepsilon\}} \frac{(1-v_{0,\varepsilon})^2}{\varepsilon} \geqslant \frac{(1-t_\varepsilon)^2}{\varepsilon} |\{v_{0,\varepsilon} \leqslant s_\varepsilon\}|
\end{equation}
Thus $|\{v_{0,\varepsilon} \leqslant s_\varepsilon\}| \rightarrow 0$ and $ \|u_\varepsilon - z_\varepsilon \|_{L^2} \rightarrow 0$. 

\medskip

\underline{$\limsup\limits_{\varepsilon \rightarrow 0} \overline{AT}_\varepsilon (u_\varepsilon , v_\varepsilon) \leqslant \overline{MS}(u)$ :}

\medskip

Fix $W_1$ a neighbourhood of $J_u\setminus \Gamma_0$ and $W_2$ a neighbourhood of $\Gamma_0$ so that $W_1 \cap W_2 = \emptyset$. Then one has 
\begin{equation}\label{ueps2}
v_\varepsilon = w_{1,\varepsilon} \, \text{on } W_1 \;\;\; ; \;\;\; u_\varepsilon =  \left\{ \begin{array}{ll}
z_\varepsilon & \text{on } W_1,\\
\varphi_\epsilon u & \text{on } W_2,\\
u & \text{outside of } W_1\sqcup W_2.\\
\end{array}\right.
\end{equation}
Indeed, outside of $W_2$, $v_{0,\varepsilon} \geqslant 1-M\varepsilon > t_\varepsilon$ (see \autoref{ellip2} applied to $\mu_\varepsilon =0$ and $x_{min} = L/2$) so, outside of $W_2$, $w_{2,\varepsilon} = v_{0,\varepsilon}$ and $v_\varepsilon = \min(w_{1,\varepsilon}, v_{0,\varepsilon}) = w_{1,\varepsilon}$ and $\varphi_\varepsilon$ is equal to $1$ outside of $W_2$. Similarly, for $\varepsilon$ small enough, $\{d\leqslant \alpha_\varepsilon\} \subset W_1$, so $z_\varepsilon = u$ outside of $W_1$.

\smallskip

Concerning the convergence of $MM_\varepsilon (v_\varepsilon)$, by \eqref{modicaw1} and \eqref{modicaw2eps}. 
\[\begin{array}{lll}
\displaystyle\limsup\limits_{\varepsilon \rightarrow 0} MM_\varepsilon (v_\varepsilon)&\leqslant& \displaystyle\limsup\limits_{\varepsilon \rightarrow 0} MM_\varepsilon^{W_1} (w_{1,\varepsilon}) + \limsup\limits_{\varepsilon \rightarrow 0} MM_\varepsilon (w_{2,\varepsilon})\\
&&\\
&\leqslant& \displaystyle 2 \mathcal{H}^0 (J_u \setminus \Gamma_0) + \limsup\limits_{\varepsilon \rightarrow 0} \left( MM_\varepsilon (v_{0,\varepsilon}) + \mathcal{C}\left(\frac{2 k_\varepsilon}{(k_\varepsilon-1)^2} + \frac{2t_\varepsilon}{(k_\varepsilon-1)(1-t_\varepsilon)^2} + \frac{2t_\varepsilon}{(1-t_\varepsilon)^2}\right)\right)\\
\end{array}\]
Then, according to \cite[Theorem 1.1 (iii),(iv)]{BMRb23}, $MM_\varepsilon (v_{0,\varepsilon}) \rightarrow 2 \mathcal{H}^0 (\Gamma_0)$. Thus, using \autoref{quantity0},
\[\displaystyle\limsup\limits_{\varepsilon \rightarrow 0} MM_\varepsilon (v_\varepsilon)\leqslant \displaystyle 2 \mathcal{H}^0(J_u \setminus \Gamma_0) + 2 \mathcal{H}^0 (\Gamma_0)
=  2 \mathcal{H}^0 (J_u \cup \Gamma_0).\]
Now we deal with the convergence of the elastic energy $\displaystyle \int_0^L (\eta_\varepsilon + v_\varepsilon^2) |u_\varepsilon '|^2$. According to \eqref{MSw1},
\[\limsup\limits_{\varepsilon\rightarrow 0} \int_{W_1} (\eta_\varepsilon + v_\varepsilon ^2 ) |u_\varepsilon '|^2 \leqslant \limsup\limits_{\varepsilon\rightarrow 0} \int_{W_1}(\eta_\varepsilon + w_{1,\varepsilon}^2 ) |z_\varepsilon '|^2 \leqslant \int_{W_1} |u'|^2,\]
and using the expression \eqref{ueps2} of $u_\varepsilon$,
\[\limsup\limits_{\varepsilon\rightarrow 0} \int_{[0,L] \setminus (W_1\sqcup W_2)} (\eta_\varepsilon + v_\varepsilon ^2 ) |u_\varepsilon '|^2 \leqslant \limsup\limits_{\varepsilon\rightarrow 0} \int_{[0,L] \setminus (W_1\sqcup W_2)} (\eta_\varepsilon + 1 ) |u'|^2  = \int_{[0,L] \setminus (W_1 \sqcup W_2)} |u'|^2.\]
It remains to deal with 
\[\limsup\limits_{\varepsilon\rightarrow 0} \int_{W_2} (\eta_\varepsilon + v_\varepsilon ^2 ) |u_\varepsilon '|^2.\]
To compute $\displaystyle\int_{W_2} (\eta_\varepsilon + v_\varepsilon ^2 ) |u_\varepsilon '|^2$, we recall the definition \eqref{phieps} of $\varphi_\varepsilon$ and we apply the inequality $|p+q|^2 \leqslant 2 (|p|^2+|q|^2)$ to $p=\varphi_\varepsilon ' z$ and $q=\varphi_\varepsilon z'$ on the set $\{r_\varepsilon \leqslant v_{0,\varepsilon} \leqslant s_\varepsilon\}$. Namely, 
\[\begin{array}{lll}
\displaystyle\int_{W_2} (\eta_\varepsilon + v_\varepsilon ^2 ) |(\varphi_\varepsilon u) '|^2 &=& \displaystyle\int_{\{s_\varepsilon\leqslant v_{0,\varepsilon}\}\cap W_2} (\eta_\varepsilon + v_\varepsilon)^2 |\varphi_\varepsilon ' u + \varphi_\varepsilon u' |^2 +\int_{{\{r_\varepsilon\leqslant v_{0,\varepsilon} \leqslant s_\varepsilon\}}} (\eta_\varepsilon + v_\varepsilon)^2 |\varphi_\varepsilon ' u + \varphi_\varepsilon u'|^2\\
&&\\
&\leqslant & \displaystyle\int_{\{s_\varepsilon\leqslant v_\varepsilon\}\cap W_2} (\eta_\varepsilon + 1) |u'|^2 + 2\int_{\{r_\varepsilon\leqslant v_\varepsilon \leqslant s_\varepsilon\}} \eta_\varepsilon \underbrace{|\varphi_\varepsilon|^2}_{\leqslant 1} |u'|^2 + 2\int_{\{r_\varepsilon\leqslant v_\varepsilon \leqslant s_\varepsilon\}} \eta_\varepsilon |\varphi_\varepsilon '|^2 |u|^2\\
&&\\
&\leqslant & \displaystyle\int_{\{s_\varepsilon\leqslant v_\varepsilon\} \cap W_2} (\eta_\varepsilon + 1) |u'|^2 + 2\eta_\varepsilon  \|u' \|_{L^2} + 2  \|u \|_\infty^2 \int_{\{r_\varepsilon\leqslant v_\varepsilon \leqslant s_\varepsilon\}} \eta_\varepsilon |\varphi_\varepsilon '|^2.
\end{array}\]
The limsup of this last term is smaller than $\int_{W_2} |u'|^2$, the second term goes to $0$ and the third term goes to $0$ according \eqref{cvvarphi}, which concludes the proof of the $\Gamma -$limsup. 
\end{proof}

We now turn to the proof of the lemma \autoref{w2eps} : 

\begin{proof}[Proof of \autoref{w2eps}]
We start with the construction of a cut-off function $\varphi_\varepsilon$ satisfying \eqref{phieps} and \eqref{cvvarphi}.
Divide the interval $[r_\varepsilon,s_\varepsilon]$ into $h_\varepsilon$ intervals $I_1, \ldots , I_{h_\varepsilon}$ of the same size, such that $\dfrac{\eta_\varepsilon}{\varepsilon} \cdot \dfrac{h_\varepsilon}{(s_\varepsilon - r_\varepsilon)^2} \rightarrow 0$ (which is possible as $\dfrac{\eta_\varepsilon}{\varepsilon} \cdot \dfrac{1}{(s_\varepsilon-r_\varepsilon)^2} \rightarrow 0$, take for example $h_\varepsilon = \left\lfloor\left(\dfrac{\eta_\varepsilon}{\varepsilon} \cdot \dfrac{1}{(s_\varepsilon - r_\varepsilon )^2} \right)^{-1/2}\right\rfloor)$. Note that by the energy bound \eqref{energyboundv0} : 
\[\mathcal{C} \geqslant \sum_{j=1}^{h_\varepsilon} \int_{I_j} \varepsilon |v_{0,\varepsilon} '|^2.\]
So there exists an interval $I_{\varepsilon}= [\alpha_\varepsilon, \beta_\varepsilon]$ of length $(s_\varepsilon - r_\varepsilon)/h_\varepsilon$ such that 
\[\int_{I_\varepsilon} \varepsilon |v_{0,\varepsilon} '|^2 \leqslant \frac{\mathcal{C}}{h_\varepsilon}.\]
Define $\varphi_\varepsilon = \min\left(\frac{(v_{0,\varepsilon} - \alpha_\varepsilon )^+}{\beta_\varepsilon - \alpha_\varepsilon} , 1 \right)$.
Then $\varphi_\varepsilon$ takes the desired values on $\{v_{0,\varepsilon} \leqslant r_\varepsilon\}$ and $\{v_{0,\varepsilon} \geqslant s_\varepsilon\}$ and 
\[\eta_\varepsilon \int_{\{r_\varepsilon\leqslant v_{0,\varepsilon} \leqslant s_\varepsilon\}} |\varphi_\varepsilon '|^2 = \eta_\varepsilon \int_{\{\alpha_\varepsilon \leqslant v_{0,\varepsilon} \leqslant \beta_\varepsilon\}} \frac{|v_{0,\varepsilon} '|^2}{(\beta_\varepsilon - \alpha_\varepsilon)^2} \leqslant \frac{\eta_\varepsilon}{\varepsilon} \frac{\mathcal{C}}{h_\varepsilon} \frac{h_\varepsilon ^2}{(s_\varepsilon-r_\varepsilon)^2} \rightarrow 0\]
which proves \eqref{cvvarphi}.

\smallskip

Now we turn to the proof of \eqref{w2right1} and \eqref{modicaw2eps}. Observe that $w_{2,\varepsilon} \leqslant v_{0,\varepsilon}$ as $v_{0,\varepsilon}$ is positive and on $\{s_\varepsilon \leqslant v_{0,\varepsilon} \leqslant t_\varepsilon\}$, according to \eqref{w2epsilon} :

\[w_{2,\varepsilon} = v_{0,\varepsilon} + \frac{1}{k_\varepsilon -1} (v_{0,\varepsilon} -t_\varepsilon ) \leqslant v_{0,\varepsilon}.\]
To prove \eqref{w2right1}, observe that $w_{2,\varepsilon}$ and $v_{0,\varepsilon}$ are taking values in $[0,1]$, so that

\[\int_0^L |w_{2,\varepsilon} - v_{0,\varepsilon}|^2 = \int_{\{v_{0,\varepsilon} \leqslant t_\varepsilon\}} |w_{2,\varepsilon} - v_{0,\varepsilon}|^2 \leqslant 4 |\{v_{0,\varepsilon} \leqslant t_\varepsilon\}|.\]
Then, using the energy bound \eqref{energyboundv0}, we see that 

\begin{equation}\label{setv0}
\mathcal{C} \geqslant \int_{\{v_{0,\varepsilon} \leqslant t_\varepsilon\}} \frac{(1-v_{0,\varepsilon})^2}{\varepsilon} \geqslant \frac{(1-t_\varepsilon)^2}{\varepsilon} |\{v_{0,\varepsilon} \leqslant t_\varepsilon\}|.
\end{equation}

So $|\{v_{0,\varepsilon} \leqslant t_\varepsilon\}| \rightarrow 0$ and $ \|w_{2,\varepsilon} - v_{0,\varepsilon} \|_2 \rightarrow 0$, which proves \eqref{w2right1}. 

We turn to the proof of \eqref{modicaw2eps}. We compute the difference in energy, starting with the difference between the Dirichlet energies. The expression \eqref{w2epsilon} of $w_{2,\varepsilon}$ and the energy bound \eqref{energyboundv0} yields    
\[\begin{array}{lll}
\displaystyle\int_0^L \varepsilon |w_{2,\varepsilon}'|^2 - \varepsilon |v_{0,\varepsilon} '|^2 &=& \displaystyle\int_{\{s_\varepsilon\leqslant v_{0,\varepsilon} \leqslant t_\varepsilon\}} \varepsilon \left(\frac{k_\varepsilon^2}{(k_\varepsilon-1)^2}-1\right) |v_{0,\varepsilon} '|^2 - \int_{\{v_{0,\varepsilon} \leqslant s_\varepsilon\}} \varepsilon |v_{0,\varepsilon} '|^2\\
&&\\
&\leqslant& \displaystyle\left(\frac{k_\varepsilon^2}{(k_\varepsilon-1)^2}-1\right)\int_{\{s_\varepsilon\leqslant v_{0,\varepsilon} \leqslant t_\varepsilon\}} \varepsilon |v_{0,\varepsilon}'|^2 \leqslant \mathcal{C} \left(\frac{k_\varepsilon^2}{(k_\varepsilon-1)^2}-1\right)\\
\end{array}.\]
On the other hand 
\[\begin{array}{lll}
\displaystyle \int_0^L \frac{(1-w_{2,\varepsilon})^2}{\varepsilon} - \frac{(1-v_{0,\varepsilon})^2}{\varepsilon} &=& \displaystyle \frac{1}{\varepsilon} \int_{\{v_{0,\varepsilon} \leqslant t_\varepsilon\}} (v_{0,\varepsilon} - w_{2,\varepsilon} ) (2-v_{0,\varepsilon} - w_{2,\varepsilon} ) \\
&&\\
&=& \displaystyle \frac{1}{\varepsilon}\int_{\{s_\varepsilon\leqslant v_{0,\varepsilon} \leqslant t_\varepsilon\}} \frac{1}{k_\varepsilon-1} \underbrace{(t_\varepsilon-v_{0,\varepsilon})}_{\leqslant t_\varepsilon}\underbrace{(2-v_{0,\varepsilon} - w_{2,\varepsilon})}_{\leqslant 2} + \frac{1}{\varepsilon}\int_{\{v_{0,\varepsilon} \leqslant s_\varepsilon\}} \underbrace{v_{0,\varepsilon}}_{\leqslant s_\varepsilon} \underbrace{(2-v_{0,\varepsilon})}_{\leqslant 2} \\
&&\\
&\leqslant & \displaystyle \frac{2t_\varepsilon}{k_\varepsilon-1} \frac{|\{v_{0,\varepsilon} \leqslant t_\varepsilon\}|}{\varepsilon} + \frac{2s_\varepsilon}{k_\varepsilon-1} \frac{|\{v_{0,\varepsilon} \leqslant s_\varepsilon\}|}{\varepsilon} \\
\end{array} \]
But the energy bound \eqref{energyboundv0} yields, in a similar fashion than \eqref{setv0}, for any $\alpha >0$,
\[\mathcal{C} \geqslant \int_{\{v_ {0,\varepsilon} \leqslant \alpha\}} \frac{(1-v_{0,\varepsilon})^2}{\varepsilon} \geqslant \frac{(1-\alpha)^2}{\varepsilon} |\{v_{0,\varepsilon} \leqslant \alpha\}|.\]
Hence, the last term is smaller than for  
\[\mathcal{C}\frac{2t_\varepsilon}{(k_\varepsilon-1)(1-t_\varepsilon)^2} + \mathcal{C} \frac{2s_\varepsilon}{(1-s_\varepsilon)^2} \leqslant \mathcal{C} \left(\frac{2t_\varepsilon}{(k_\varepsilon-1)(1-t_\varepsilon)^2} + \frac{2t_\varepsilon}{(1-t_\varepsilon)^2}\right).\] 
Summing both estimates yields the desired result. This achieves the proof of \autoref{w2eps}. 
\end{proof}

\subsection{Critical Points of $\overline{MS}$}

\autoref{Theorem 2} ensures that the minimizers of $\overline{AT}_\varepsilon$ converge to a pair $(1,u)$ that minimizes $\overline{MS}$ if we have the additional assumption of an energy bound on $(AT_\varepsilon)$. 

A natural extension is to ask whether the critical points of $\overline{AT}_\varepsilon$ also converge to critical points of $\overline{MS}$. Here, we prove that the critical points of $\overline{MS}$ are either affine or piecewise constant. 

Take $u$ a minimizer of $\overline{MS}$. First, $u$ has a finite number of discontinuity points, that we denote by $\{x_1, \ldots , x_d\}$. Let $\varphi$ in $SBV^2 (0,L)$ with $J_\varphi\subset J_u$ and take $u+t\varphi$ as a competitor for $t$ a real number. Then, $J_{u+t\varphi} = J_u$ so one has 
\[\int_0^L |u'|^2 \leqslant \int_0^L |u' +t \varphi '|^2 ,\]
which rewrites, after dividing by $t$, letting $t$ tend successively to $0^+$ and $0^-$, as : 
\[\int_0^L u' \varphi' = 0.\]
So $u$ is affine on each interval $(x_i, x_{i+1})$. In the case where $u$ admits at least one discontinuity point, say $x_i$, and take $\varphi$ equal to $0$ outside $(x_i, x_{i+1})$ and equal to $\dfrac{x-x_{i+1}}{x_i-x_{i+1}}$ on $(x_i, x_{i+1})$ (where $x_{i+1}$ can be replaced by $L$). Then $\varphi \in SBV^2(0,L)$ with $\varphi' = 1_{(x_i, x_{i+1})}, J_\varphi =\{x_i\}$ and  
\[u(x_{i+1})-u(x_i) = \int_{x_i}^{x_{i+1}} u' = \int_0^L u'\varphi' = 0,\]
so that $u(x_i) = u(x_{i+1})$, i.e $u$ is piecewise constant. 

\smallskip

Hence, critical points of $\overline{AT}_\varepsilon$ converge to critical points of $\overline{MS}$. However, there is a selection criterion as limit functions of critical points of $AT_\varepsilon$ must be non decreasing and admit at most one discontinuity point.  

\bibliographystyle{alpha}
\bibliography{bibliography}

\begin{thebibliography}{FRRO22}

\bibitem[AT92]{AT92}
Luigi Ambrosio and V.~M. Tortorelli.
\newblock On the approximation of free discontinuity problems.
\newblock {\em Boll. Un. Mat. Ital. B (7)}, 6(1):105--123, 1992.

\bibitem[BBH94]{BBH94}
Fabrice Bethuel, Ha\"{\i}m Brezis, and Fr\'{e}d\'{e}ric H\'{e}lein.
\newblock {\em Ginzburg-{L}andau vortices}, volume~13 of {\em Progress in
  Nonlinear Differential Equations and their Applications}.
\newblock Birkh\"{a}user Boston, Inc., Boston, MA, 1994.

\bibitem[BBO01]{BBO01}
F.~Bethuel, H.~Brezis, and G.~Orlandi.
\newblock Asymptotics for the {G}inzburg-{L}andau equation in arbitrary
  dimensions.
\newblock {\em J. Funct. Anal.}, 186(2):432--520, 2001.

\bibitem[BFM08]{BFM08}
Blaise Bourdin, Gilles~A. Francfort, and Jean-Jacques Marigo.
\newblock {\em The variational approach to fracture}.
\newblock Springer, New York, 2008.
\newblock Reprinted from J. Elasticity {\bf 91} (2008), no. 1-3 [MR2390547],
  With a foreword by Roger Fosdick.

\bibitem[BI23]{BI23}
Marco Bonacini and Flaviana Iurlano.
\newblock Convergence of critical points for a phase-field approximation of 1d
  cohesive fracture energies.
\newblock {\em Calculus of Variations and Partial Differential Equations}, 63,
  2023.

\bibitem[BMR23]{BMRb23}
Jean-Fran\c{c}ois Babadjian, Vincent Millot, and R\'{e}my Rodiac.
\newblock A note on the one-dimensional critical points of the
  {A}mbrosio-{T}ortorelli functional.
\newblock {\em Asymptot. Anal.}, 135(3-4):349--362, 2023.

\bibitem[BMR24]{BMRa23}
Jean-Fran\c{c}ois Babadjian, Vincent Millot, and R\'emy Rodiac.
\newblock On the convergence of critical points of the {A}mbrosio--{T}ortorelli
  functional.
\newblock {\em Ann. Inst. H. Poincar\'e{} C Anal. Non Lin\'eaire},
  41(6):1367--1417, 2024.

\bibitem[Bou07]{B07}
Blaise Bourdin.
\newblock Numerical implementation of the variational formulation for
  quasi-static brittle fracture.
\newblock {\em Interfaces Free Bound.}, 9(3):411--430, 2007.

\bibitem[FL03]{FL03}
Gilles~A. Francfort and Christopher~J. Larsen.
\newblock Existence and convergence for quasi-static evolution in brittle
  fracture.
\newblock {\em Comm. Pure Appl. Math.}, 56(10):1465--1500, 2003.

\bibitem[FLS09]{FLS09}
Gilles~A. Francfort, Nam~Q. Le, and Sylvia Serfaty.
\newblock Critical points of {A}mbrosio-{T}ortorelli converge to critical
  points of {M}umford-{S}hah in the one-dimensional {D}irichlet case.
\newblock {\em ESAIM Control Optim. Calc. Var.}, 15(3):576--598, 2009.

\bibitem[FM98]{FM98}
G.~A. Francfort and J.-J. Marigo.
\newblock Revisiting brittle fracture as an energy minimization problem.
\newblock {\em J. Mech. Phys. Solids}, 46(8):1319--1342, 1998.

\bibitem[FRRO22]{FO22}
Xavier Fern\'{a}ndez-Real and Xavier Ros-Oton.
\newblock {\em Regularity theory for elliptic {PDE}}, volume~28 of {\em Zurich
  Lectures in Advanced Mathematics}.
\newblock EMS Press, Berlin, [2022] \copyright 2022.

\bibitem[Gia05]{G05}
Alessandro Giacomini.
\newblock Ambrosio-{T}ortorelli approximation of quasi-static evolution of
  brittle fractures.
\newblock {\em Calc. Var. Partial Differential Equations}, 22(2):129--172,
  2005.

\bibitem[GT01]{GT01}
David Gilbarg and Neil~S. Trudinger.
\newblock {\em Elliptic partial differential equations of second order}.
\newblock Classics in Mathematics. Springer-Verlag, Berlin, 2001.
\newblock Reprint of the 1998 edition.

\bibitem[HT00]{HT00}
John~E. Hutchinson and Yoshihiro Tonegawa.
\newblock Convergence of phase interfaces in the van der
  {W}aals-{C}ahn-{H}illiard theory.
\newblock {\em Calc. Var. Partial Differential Equations}, 10(1):49--84, 2000.

\bibitem[Le10]{L10}
Nam~Q. Le.
\newblock Convergence results for critical points of the one-dimensional
  {A}mbrosio-{T}ortorelli functional with fidelity term.
\newblock {\em Adv. Differential Equations}, 15(3-4):255--282, 2010.

\bibitem[LR01]{LR01}
Fang-Hua Lin and Tristan Rivi\`ere.
\newblock A quantization property for static {G}inzburg-{L}andau vortices.
\newblock {\em Comm. Pure Appl. Math.}, 54(2):206--228, 2001.

\bibitem[Mod87]{M87}
Luciano Modica.
\newblock Gradient theory of phase transitions with boundary contact energy.
\newblock {\em Ann. Inst. H. Poincar\'{e} Anal. Non Lin\'{e}aire},
  4(5):487--512, 1987.

\bibitem[MS89]{MS89}
David Mumford and Jayant Shah.
\newblock Optimal approximations by piecewise smooth functions and associated
  variational problems.
\newblock {\em Comm. Pure Appl. Math.}, 42(5):577--685, 1989.

\bibitem[SS07]{SS07}
Etienne Sandier and Sylvia Serfaty.
\newblock {\em Vortices in the magnetic {G}inzburg-{L}andau model}, volume~70
  of {\em Progress in Nonlinear Differential Equations and their Applications}.
\newblock Birkh\"{a}user Boston, Inc., Boston, MA, 2007.

\bibitem[Ste88]{S88}
Peter Sternberg.
\newblock The effect of a singular perturbation on nonconvex variational
  problems.
\newblock {\em Arch. Rational Mech. Anal.}, 101(3):209--260, 1988.

\bibitem[Ton02]{T02}
Yoshihiro Tonegawa.
\newblock Phase field model with a variable chemical potential.
\newblock {\em Proc. Roy. Soc. Edinburgh Sect. A}, 132(4):993--1019, 2002.

\bibitem[Ton05]{T05}
Yoshihiro Tonegawa.
\newblock A diffused interface whose chemical potential lies in a {S}obolev
  space.
\newblock {\em Ann. Sc. Norm. Super. Pisa Cl. Sci. (5)}, 4(3):487--510, 2005.

\end{thebibliography}

\end{document}